\documentclass[12pt,oneside,reqo]{article}
\usepackage{amsfonts,amsmath,amssymb, amsthm,enumerate,esint, mathabx,
mathtools,tikz,hyperref,bbm,stackengine}
\usepackage[font=footnotesize]{caption}
\usepackage{subcaption}
\usepackage[edges]{forest}
\usepackage{float}
\usepackage[nottoc,numbib]{tocbibind}
\usepackage{xcolor}
\usepackage[margin=1.25in]{geometry}
\newtheorem{thm}{Theorem}[section]
\newtheorem{lem}[thm]{Lemma}
\newtheorem{prop}[thm]{Proposition}

\newtheorem{cor}[thm]{Corollary}
\newtheorem*{thm*}{Theorem}

\theoremstyle{definition}

\newtheorem{hypsub}{H}[thm]

\newtheorem*{rem*}{Remark}

\newtheorem*{claim*}{Claim}
\newtheorem{assm}{Assumption}
\theoremstyle{plain}

\DeclareMathOperator{\Ima}{im}

\DeclareMathOperator{\dom}{dom}

\newcommand{\dd}{{\mathrm{d}}}

\DeclareMathOperator{\spec}{sp}

\DeclareMathOperator{\Span}{span}
\DeclareMathOperator{\rank}{rank}
\DeclareMathOperator{\dist}{dist}

\newcommand{\bbR}{{\mathbb{R}}}

\newcommand{\bbZ}{{\mathbb{Z}}}

\newcommand{\xddots}{%
  \raise 4pt \hbox {.}
  \mkern 6mu
  \raise 1pt \hbox {.}
  \mkern 6mu
  \raise -2pt \hbox {.}
}

\newcommand{\snext}{\widehat}
\newcommand{\sprev}{\widecheck}
\newcommand{\salt}{\widetilde}

\numberwithin{equation}{section}

\title{Localization and Cantor spectrum for quasiperiodic discrete Schr\"odinger operators with asymmetric, smooth, cosine-like sampling functions}
\author{Yakir Forman and Tom VandenBoom}
\date{}

\begin{document}

\maketitle

\begin{abstract}
    We prove Cantor spectrum and almost-sure Anderson localization for quasi-periodic discrete Schr\"odinger operators $H = \varepsilon\Delta + V$ with potential $V$ sampled with Diophantine frequency $\alpha$ from an asymmetric, smooth, cosine-like function $v \in C^2(\mathbb{T},[-1,1])$ for sufficiently small interaction $\varepsilon \leq \varepsilon_0(v,\alpha)$.  We prove this result via an inductive analysis on scales, whereby we show that locally the Rellich functions of Dirichlet restrictions of $H$ inherit the cosine-like structure of $v$ and are uniformly well-separated.  
\end{abstract}

\tableofcontents
\newpage

\section{Introduction and main results}

Fix a bounded real sequence $V = \{V_j\}_{j \in \mathbb{Z}}$ and $\varepsilon > 0$.  We consider the discrete Schr\"odinger operator $H = H(V,\varepsilon)$ whose pointwise action on sequences in $\mathbb{C}^\mathbb{Z}$ is given by
\begin{align}\label{eq:DSOdef}
	(H\psi)(n) &:= \left((\varepsilon\Delta + V)\psi\right)(n) \\
	\nonumber &:= \varepsilon (\psi(n+1) + \psi(n-1)) + V_n\psi(n).
\end{align}

Restricted to $\ell^2(\mathbb{Z})$, $H$ is a bounded self-adjoint linear operator.  $H$ is called Anderson localized if it exhibits an $\ell^2(\mathbb{Z})$ basis of exponentially decaying eigenvectors.  The Anderson model, where $V$ is sampled random i.i.d. from a nontrivial probability distribution, is almost surely Anderson localized; in this sense, localization is a standard measure of ``randomness'' of the potential.  That said, certain non-random potentials, e.g. $V_n = \cos(\theta + n\alpha)$, are likewise almost-surely Anderson localized for small $\varepsilon$ and almost every $\alpha$.  Proofs of this fact have historically relied on either the analyticity of cosine and subsequent uniform positivity of the Lyapunov exponent or its inherent symmetry.  We offer a new perturbative proof of almost-sure localization and Cantor spectrum for potentials sampled from any $C^2$-smooth Morse function with two monotonicity intervals along a Diophantine rotation on the circle.

We consider ``cosine-like" sampling functions from $\mathbb{T} = \mathbb{R}/\mathbb{Z}$ into $[-1,1]$ having two non-degenerate critical points; that is, we consider functions $f \in C^2(\mathbb{T}, [-1,1])$ having two monotonicity intervals and satisfying a Morse condition
\begin{align*}
d_0 \leq |\partial_\theta f| + |\partial_\theta^2 f| \leq D_0.
\end{align*}
We say $\alpha \in [0,1] \setminus \mathbb{Q}$ is $(C,\tau)$-Diophantine when
\begin{align}\label{eq:dioph}
	\|n\alpha\|_{\mathbb{T}} \geq \frac{C}{|n|^{\tau}}, \quad n \in \mathbb{Z}.
\end{align}
In this case we write $\alpha \in DC_{C,\tau}$.  It is well-known that Lebesgue almost every $\alpha$ is Diophantine for some $(C, \tau)$.  We consider potentials $V = V(\theta_*)$ generated by sampling a cosine-like $v$ along an irrational rotation by $\alpha$ starting at $\theta_*$, i.e.
\begin{align*}
	V_n(\theta_*) = v(\theta_* + n\alpha),
\end{align*}
and denote by $H(\theta_*) = H(V(\theta_*),\varepsilon)$.

\begin{thm*}[Main Theorem]
	Let $v \in C^2(\mathbb{T},[-1,1])$ be a function with two monotonicity intervals satisfying a Morse condition $$d_0 \leq |\partial_\theta v| + |\partial_\theta^2 v| \leq D_0$$ and let $\alpha \in DC_{C,\tau}$.  Then there exists $\varepsilon_0 = \varepsilon_0(v,\alpha)$ such that, for $\varepsilon < \varepsilon_0$, 
	\begin{enumerate}
		\item The spectrum $\Sigma = \Sigma(v,\alpha)$ of $H(\theta_*)$ is a Cantor set, and
		\item For Haar almost every $\theta_* \in \mathbb{T}$, $H(\theta_*)$ is \textit{Anderson localized}, i.e. has pure-point spectrum with exponentially decaying eigenfunctions. 
	\end{enumerate}
\end{thm*}

Our approach to the proof of the Main Theorem is an inductive analysis of scales.  Multiscale analysis techniques are a well-established approach to localization in arbitrary spatial dimensions, cf. \cite{CKM87, DK89, EK16, FMSS85, FroSpe83, FroSpeWit90, GK01}, and are quite general in their scope; for a recent survey of multiscale analysis techniques, see \cite{Schlag21}.  In one spatial dimension, as in our model, one has access to additional tools (like the Lyapunov exponent and transfer matrices) which can simplify the analysis.  Indeed, uniform positivity of the Lyapunov exponent often lends itself to almost-sure \textit{nonperturbative} localization results \cite{BinKinVod16,BouGol00}.  Such results were initially established by Jitomirskaya for the Almost Mathieu operator \cite{Jit99}, and this strategy has been well-applied in the analytic regime due to the Hermann ``subharmonicity trick,'' see e.g. \cite{Bou05, BouGol00, BouGolSch01, BouSch00, SorSpe91}, as well as in the random i.i.d. regime due to F\"urstenberg's Theorem, see e.g. \cite{BDFGVWZ19-1,BDFGVWZ19-2,DFS20,GK17,JitZhu19}.

For smooth sampling functions, proving uniform positivity of the Lyapunov exponent becomes a substantial challenge.  In this case, only perturbative strategies have thus far been successful.  Fr\"ohlich-Spencer-Wittwer established a  perturbative approach to proving almost-sure Anderson localization for symmetric cosine-like smooth potentials $v$ with sufficiently small interaction for almost every irrational $\alpha$ \cite{FroSpeWit90}, with additional work in this direction pioneered by Sinai \cite{Sin87}.  Positive Lyapunov exponents and Cantor spectrum were verified via perturbative methods in the general asymmetric cosine-like case by Wang-Zhang \cite{WanZha15, WanZha16}.

\subsection{What this paper accomplishes}

Our result will naturally be compared to the aforementioned works of Fr\"ohlich-Spencer-Wittwer \cite{FroSpeWit90}, Sinai \cite{Sin87}, and Wang-Zhang \cite{WanZha15, WanZha16}.  Our main accomplishments relative to these prior works are, respectively:
\begin{enumerate}
	\item Removing the crucial symmetry assumption on the sampling function $v$ from \cite{FroSpeWit90} in part by an inductive analysis of differences of inverse functions; cf. Section 4.
	\item Establishing the existence of open spectral gaps via a novel Cauchy interlacing argument (cf. Theorem \ref{t:DRRelFns}) and a multiscale induction scheme which preserves the size of these gaps (cf. Proposition \ref{pr:induction}).
	\item Eliminating double-resonances, i.e., those phases $\theta$ for which the frequency $\alpha$ is problematically recurrent relative to the sampling function $v$ infinitely often (cf. Lemma \ref{l:badsets}).
\end{enumerate}
Each of these steps requires new ideas not present in the above works; we describe the difficulties which must be overcome, and our approaches to overcoming them, below.

\subsubsection{Removing symmetry}

In \cite{FroSpeWit90}, the authors prove an analogue to the Main Theorem under the additional assumption of even symmetry of the sampling function $v$ defining the potential: $v(\theta) = v(-\theta)$.  This assumption is crucial to their analysis because the symmetry is inherited by the Rellich functions (i.e., parametrized eigenvalues) $\mathbf{E}^\Lambda$ of any Dirichlet restriction $H^\Lambda$ of $H$ to an interval $\Lambda \subset \mathbb{Z}$; specifically, if $\Lambda = [c - L, c+L]$, then
\begin{align*}
	\mathbf{E}^\Lambda(\theta) = \mathbf{E}^\Lambda(-\theta - 2c\alpha), \quad \theta \in \mathbb{T}.
\end{align*}
In particular, it follows that there is antisymmetry in the derivative:
\begin{align*}
	\partial_\theta \mathbf{E}^\Lambda(\theta) = -(\partial_\theta \mathbf{E}^\Lambda)(-\theta - 2c\alpha).
\end{align*}
These symmetries come with two crucial benefits.  First, Rellich curves of $H^\Lambda$ have predetermined critical points when
\begin{align*}
	\|2\theta_c + 2c\alpha\|_\mathbb{T} = 0.
\end{align*}
These critical points are common to any Rellich curves of $H^\Lambda$.  This fact, alongside opposite-signed second derivatives (cf. \eqref{eq:feynman2}), allows one to conclude a uniform local eigenvalue separation from the classical pointwise eigenvalue separation for eigenvalues with eigenvectors localized on a common support  (cf. Lemma \ref{l:evalsep}).  The uniform local separation of Rellich curves guarantees that different Rellich curves cannot resonate with one another, which in turn allows one to, e.g., apply calculus and the Diophantine condition to single Rellich curves at every step of the induction.

The second important consequence of the symmetry of the Rellich functions is that the difference of local inverse functions -- which plays a crucial role in identifying double resonances (cf. Figure \ref{f:doubleres}) -- is effectively independent of the scale and the energy, cf. \cite[Lemma 5.7]{FroSpeWit90}.  This means that, in the even setting, one can define the bad sets of double resonant phases uniformly in the energy.  In our construction, the difference of local inverses of Rellich functions $\mathbf{E}_{s,+}^{-1} - \mathbf{E}_{s,-}^{-1}$ depends on the parent Rellich function, and thus so too do our bad sets of phases.  By carefully controlling the number of double resonances and child Rellich functions coming from each parent (cf. equation \eqref{eq:counting} below), we can discretize this energy dependence and again construct a uniform bad set of phases at each scale, whose limit superior we eventually eliminate to prove Anderson localization.

\subsubsection{Opening gaps}

Sinai's paper \cite{Sin87} is likewise foundational in establishing an approach to proving perturbative localization results in the quasiperiodic setting; indeed, general aspects of our argument, like the idea illustrated in Figure \ref{f:RelCollection} below, echo some of the ideas developed therein.  However, Sinai's paper suffers from fundamental flaws as written, each of which our approach overcomes.  

Sinai's approach to studying quasiperiodic operators is \textit{a priori} a natural one: for fixed frequency $\alpha$ having best rational approximants $\alpha_s = p_s/q_s$ with $q_s$ having controlled growth rate $q_{s+1} \lesssim (s+1)^2q_s$, approximate the whole-line operator $H$ via the $q_s$-periodic operators $H^{(s)}_\varepsilon(\theta)$.  In Sinai's construction, the transition between steps $s$ and $s+1$ involves approximating $H_\varepsilon^{(s+1)}(\theta)$ with shifts $H^{(s)}_\varepsilon(\theta + n\alpha_{s+1})$.  By the definition of the periodic approximants $\alpha_s$, one has
\begin{align*}
	|\alpha_s - \alpha_{s+1}| = \frac{1}{q_sq_{s+1}} \gtrsim \frac{1}{(s+1)^2q_s^2}
\end{align*}
and so the general error incurred in an inductive step $s \mapsto s+1$ is, at best, \textit{polynomial in $q_s^{-1}$}.  In contrast, in this inductive step there are double-resonant eigenfunctions whose centers of localization are separated at distances $q_s$.  The separation between the resultant Rellich functions is \textit{exponentially small in $q_s$}, comparable to $\varepsilon^{q_s}$.  Thus, after making an inductive step using periodic approximations, the uniform local separation between Rellich curves -- even coming from previous scales! -- is completely destroyed, taking along with it the ability to consider only self-resonances of a given Rellich curve.  

Critiques aside, Sinai's intuition regarding the importance of the gaps between children of resonant Rellich curves is informative.  Describing these gaps in our regime is perhaps the most novel contribution in this work; we do so by proving the Rellich functions at scale $s+1$ are simultaneously interlaced by a pair of auxiliary curves which are $C^1$ close to the parent curves from scale $s$, which must be monotone with opposite-signed derivatives; cf. Lemma \ref{l:localsepLem}.

\subsubsection{Eliminating double resonances}

Compared to the works of Fr\"ohlich-Spencer-Wittwer and Sinai, relatively recent progress on understanding localization for smooth quasiperiodic discrete Schr\"odinger operators was made by Wang and Zhang \cite{WanZha15}.  Therein, the authors show perturbatively that the Lyapunov exponent of the associated transfer matrices is arbitrarily close to $|\log\varepsilon|$, and they likewise prove a Large Deviation Theorem.  Their proof again proceeds by studying an auxiliary transfer matrix which is in some sense an asymptotic approximation to the Schr\"odinger transfer matrix as $\varepsilon$ decreases to zero.  In a follow-up paper \cite{WanZha16}, the authors use the characterization of spectral energies as corresponding to non-uniform hyperbolicity of the associated transfer matrix to prove Cantor spectrum in the model.

In the one-dimensional setting, uniformly positive Lyapunov exponents and a Large Deviation Theorem are key ingredients to proving Anderson localization; however, alone they are insufficient to identify a full-measure set of phases for which localization holds.  This final ingredient for localization, often called ``elimination of double resonances,'' requires a careful understanding of the sets of double resonant phases at each inductive step.  These bad sets of phases depend crucially on the energies for which they are resonant.

In our construction, given a Morse, two-monotonicity interval Rellich curve $\mathbf{E}_{s}$ at scale $s$, we can identify the bad sets of phases for $\mathbf{E}_s$ and control their sizes using the Morse condition.  One can control all of the bad sets simultaneously, then, by controlling the number of Rellich children $\mathbf{E}_{s+1}$ one constructs at each step.  Here again the Diophantine condition and the Morse, two-monotonicity interval structure of each $\mathbf{E}_s$ are crucial in order to separate the double-resonant energy regions, in turn allowing us to define at most $1/\rho^3_{s-1}$ children of $\mathbf{E}_s$, each having bad sets of phases of size at most $\rho_s^{1/5}$ (cf. the proof of Lemma \ref{l:badsets}).  Choosing our scales appropriately, we will have summability of the cumulative size of all bad sets at all scales, and the Borel-Cantelli lemma will apply.  We sketch the first step of our inductive procedure below.

\subsection{Idea of the proof}

For an interval $\Lambda = [a,b] \subset \mathbb{Z}$, we denote by $H^{\Lambda}$ the restriction of $H$ to $\Lambda$ with Dirichlet boundary conditions, $\psi(a-1) = \psi(b+1) = 0$.  Letting $|\Lambda| := b-a+1$, one can identify $H^{\Lambda}$ with the $|\Lambda| \times |\Lambda|$ matrix
\begin{align*}
	H^{\Lambda} &= \begin{bmatrix}
		V_a & \varepsilon & & & & \\
		\varepsilon & V_{a+1} & \varepsilon & & &  \\
		& \ddots & \ddots & \ddots & & \\
		& & \varepsilon & V_{b-1} & \varepsilon \\
		& & & \varepsilon & V_b
	\end{bmatrix}.
\end{align*}
Notice that the base dynamics of our model imply
\begin{align*}
	H^\Lambda(\theta + n\alpha) = H^{\Lambda + n}(\theta)
\end{align*}
for any $n \in \mathbb{Z}$, $\theta \in \mathbb{T}$; furthermore, by the regularity of our sampling function $v$, any one-parameter family of Dirichlet restrictions $H^\Lambda(\theta)$ exhibits $|\Lambda|$ Rellich functions (i.e., parametrized eigenvalues $\mathbf{E}(\theta)$) which are necessarily simple and $C^2$.

Generalized eigenvalues of the whole-line operator $H$ (i.e., energies $E$ with solutions $\psi$ to $H\psi = E\psi$ growing at most polynomially) are limit points of the eigenvalues of the finite-volume operators $H^\Lambda$ as $\Lambda$ grows to $\mathbb{Z}$; furthermore, they are spectrally dense in the spectrum of $H$ \cite{Schnol1981,Simon1981JFA}.  By the Poisson formula
\begin{align*}
	\psi(n) = \varepsilon R^\Lambda_{\theta,E}(n,a)\psi(a-1) + \varepsilon R^{\Lambda}_{\theta,E}(n,b)\psi(b+1), \quad n \in \Lambda,\, E \notin \spec(H^\Lambda(\theta))
\end{align*}
exponential off-diagonal decay of Green's functions $R^\Lambda_{\theta,E} = (H^\Lambda(\theta) - E)^{-1}$ can be favorably leveraged against the at-most polynomial growth of generalized eigenfunctions $\psi$ on $\Lambda$.  If we can inductively construct intervals of increasing length on which we have Green's function decay at (or near) generalized eigenvalues, we can thus prove exponential decay of generalized eigenfunctions, hence Anderson localization.

The initial step in our induction involves constructing intervals $\Lambda_1 = \Lambda_1(\theta_*, E_*)$ from the degenerate interval $\Lambda_0 = \{0\}$ in a way which is stable in $\theta_*$ and $E_*$.  At this zeroth scale, the eigenvalue of $H^{\Lambda_0}(\theta_*)$ is precisely the value of the sampling function $v(\theta_*)$.  The difficult aspect of the induction is identifying how the structural assumptions on $v$ are reflected by the Rellich functions of $H^{\Lambda_s}$ at future scales $s \geq 1$.

In \textit{non-resonant} situations where $|V_m - E| \geq \rho \gg \varepsilon$ is large for all $m \in \Lambda$, the Green's function exhibits off-diagonal decay at a rate of $\varepsilon/\rho$ on $\Lambda$ by a classical Neumann series argument (cf. Lemma \ref{l:initGrnDec}).  One thus focuses on those cases when $|V_m - E|$ is small:
\begin{align*}
	\mathcal{S}_0(\theta_*, E_*) = \{ m \in \mathbb{Z} : |V_m(\theta_*) - E_*| \lesssim \rho\}.
\end{align*}
If an element $m \in \mathcal{S}_0$ is separated by some relatively large distance $L$ from any other element of $\mathcal{S}_0$, we call it \textit{simple resonant} (for $v$, $\rho$, and $L$).  For simple resonant $m$, one can build an interval $\Lambda^{(1)}_1 = [m-L/2, m+L/2]$ so that any eigenvector of the associated Dirichlet restriction $H^{\Lambda_1^{(1)}}(\theta)$ with eigenvalue near $E_*$ will be localized near $m$, again by the Poisson formula \eqref{eq:poisson}.  By classical perturbation theory (cf. Lemma \ref{l:approxevect} and Proposition \ref{pr:AL1}), it follows that $H^{\Lambda^{(1)}_1}$ has a single Rellich function $\mathbf{E}_1(\theta)$ near $E_*$ for $\theta$ near $\theta_*$, and thus $\mathbf{E}_1$ will be well-approximated in $C^2$ by the sampling function $\mathbf{E}_0 := v$ (cf. Proposition \ref{pr:interscaleapprox1}).  

The obstruction to localization, then, are \textit{double resonances}, where (at least) two resonant sites $m_j$ are within $L$ of one another.  These double resonances are unavoidable; an illustration of such a resonance can be seen in Figure \ref{f:doubleres}.

\begin{figure}[h!]
	\begin{center}
		\begin{tikzpicture}[x=0.7pt,y=0.7pt,yscale=-0.6,xscale=0.7]
			
			\draw [line width=2.25]    (58,30.5) .. controls (195,36.5) and (326,411.5) .. (455,411.5) ;
			\draw [line width=2.25]    (455,411.5) .. controls (569,412.5) and (586,30.5) .. (650,32.5) ;
			\draw   (587,138.5) .. controls (587,133.83) and (584.67,131.5) .. (580,131.5) -- (368.48,131.5) .. controls (361.81,131.5) and (358.48,129.17) .. (358.48,124.5) .. controls (358.48,129.17) and (355.15,131.5) .. (348.48,131.5)(351.48,131.5) -- (211,131.5) .. controls (206.33,131.5) and (204,133.83) .. (204,138.5) ;
			\draw [color={rgb, 255:red, 155; green, 155; blue, 155 }  ,draw opacity=1 ] [dash pattern={on 0.84pt off 2.51pt}]  (53,139.5) -- (611.5,139.5) ;
			\draw  (-21,220.5) -- (693,220.5)(58.78,0.5) -- (58.78,434.5) (686,215.5) -- (693,220.5) -- (686,225.5) (53.78,7.5) -- (58.78,0.5) -- (63.78,7.5)  ;
			\draw [line width=1.5]    (53,102.5) -- (53,176.5) ;
			\draw [shift={(53,176.5)}, rotate = 270] [color={rgb, 255:red, 0; green, 0; blue, 0 }  ][line width=1.5]    (0,6.71) -- (0,-6.71)   ;
			\draw [shift={(53,102.5)}, rotate = 270] [color={rgb, 255:red, 0; green, 0; blue, 0 }  ][line width=1.5]    (0,6.71) -- (0,-6.71)   ;
			\draw [color={rgb, 255:red, 155; green, 155; blue, 155 }  ,draw opacity=1 ] [dash pattern={on 4.5pt off 4.5pt}]  (53,102.5) -- (689,102.5) ;
			\draw [color={rgb, 255:red, 155; green, 155; blue, 155 }  ,draw opacity=1 ] [dash pattern={on 4.5pt off 4.5pt}]  (53,176.5) -- (688,176.5) ;
			\draw [color={rgb, 255:red, 155; green, 155; blue, 155 }  ,draw opacity=1 ] [dash pattern={on 4.5pt off 4.5pt}]  (169,102.5) -- (169,221.5) ;
			\draw [color={rgb, 255:red, 155; green, 155; blue, 155 }  ,draw opacity=1 ] [dash pattern={on 4.5pt off 4.5pt}]  (225,176.5) -- (225,221.5) ;
			\draw [line width=1.5]    (169,231.5) -- (225,231.5) ;
			\draw [shift={(225,231.5)}, rotate = 180] [color={rgb, 255:red, 0; green, 0; blue, 0 }  ][line width=1.5]    (0,6.71) -- (0,-6.71)   ;
			\draw [shift={(169,231.5)}, rotate = 180] [color={rgb, 255:red, 0; green, 0; blue, 0 }  ][line width=1.5]    (0,6.71) -- (0,-6.71)   ;
			\draw [color={rgb, 255:red, 155; green, 155; blue, 155 }  ,draw opacity=1 ] [dash pattern={on 4.5pt off 4.5pt}]  (584,177) -- (584,221.5) ;
			\draw [color={rgb, 255:red, 155; green, 155; blue, 155 }  ,draw opacity=1 ] [dash pattern={on 4.5pt off 4.5pt}]  (605,102) -- (605,221.5) ;
			\draw [line width=1.5]    (584,232.5) -- (605,232.5) ;
			\draw [shift={(605,232.5)}, rotate = 180] [color={rgb, 255:red, 0; green, 0; blue, 0 }  ][line width=1.5]    (0,6.71) -- (0,-6.71)   ;
			\draw [shift={(584,232.5)}, rotate = 180] [color={rgb, 255:red, 0; green, 0; blue, 0 }  ][line width=1.5]    (0,6.71) -- (0,-6.71)   ;
			
			\draw (349,105.4) node [anchor=north west][inner sep=0.75pt]    {$n\alpha $};
			\draw (335,368.4) node [anchor=north west][inner sep=0.75pt]    {$v$};
			\draw (616,129.9) node [anchor=north west][inner sep=0.75pt]    {$E_{n}( v)$};
			\draw (24.5,157.07) node [anchor=north west][inner sep=0.75pt]  [rotate=-269.54]  {$J_{0,n}$};
			\draw (170,246.4) node [anchor=north west][inner sep=0.75pt]    {$I_{0,n,-}$};
			\draw (575,245.4) node [anchor=north west][inner sep=0.75pt]    {$I_{0,n,+}$};

		\end{tikzpicture}
	\end{center}
\caption{Double resonance of the sampling function $v$: if $\theta_{0,n,-} \in I_{0,n,-}$ is the value such that $v(\theta_{0,n,-}) = v(\theta_{0,n,-}+n\alpha) = E_n(v)$, then $0$ and $n$ are in $\mathcal{S}_0(\theta_{0,n,-}, E_n)$.}
\label{f:doubleres}
\end{figure}
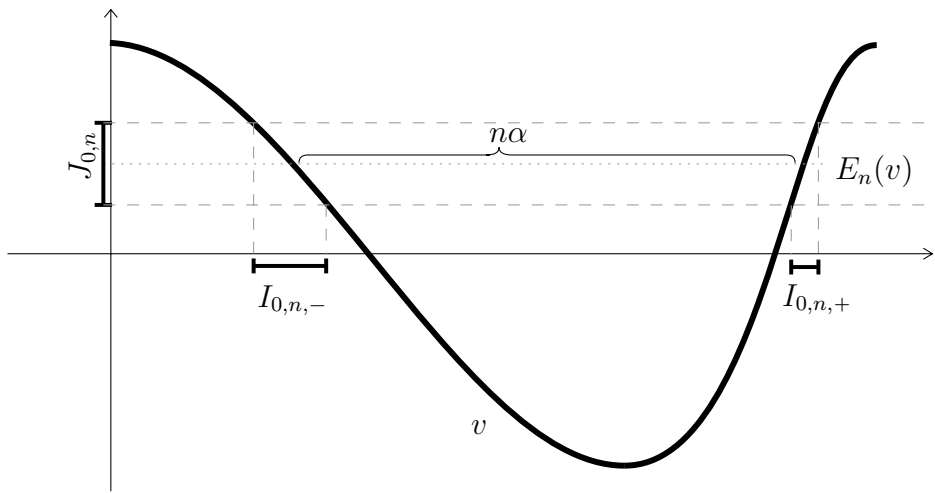
In our setting, the Diophantine condition and the structure of $v$ guarantee that at \textit{most} two resonant sites $m_1$ and $m_2$ can be nearby for appropriate choices of $\rho$ and $L$; furthermore, each of these nearby pairs $(m_1, m_2)$ must be well-separated from one another.   Indeed, let $v : \mathbb{T} \to [-1,1]$ be our Morse sampling function with two monotonicity intervals $\mathbf{I}_\pm$ and denote by $v_\pm := v|_{\mathbf{I}_\pm}$ the monotone restrictions of $v$.   By the Morse condition, one sees (cf. Lemma \ref{l:morsefnbd})
\begin{align}\label{eq:drinitIntro}
	|v_\pm(\theta) - v_\pm(\theta')| \gtrsim \|\theta - \theta'\|^2_\mathbb{T}, \quad \theta, \theta' \in \mathbf{I}_\pm.
\end{align}

By the Morse condition and \eqref{eq:drinitIntro}, double resonant sites on the orbit of $\theta_*$ cannot occur near critical points of $v$; we thus find a lower bound $\nu = \nu(\rho,L) \gg \rho$ on $|v'|$ for double resonant phases $\theta \in I_{0,n,\pm} := v_\pm^{-1}(J_{0,n})$.  What's more, given the two-monotonicity structure of $v$, there can be at most $4L$ such double-resonant intervals $I_{0,n,\pm}$, $|n| \leq L$ (cf. Figure \ref{f:doubleres}).  Thus, the region of double-resonant phases is of size at most $4L\rho/\nu$.  It also follows from the Diophantine condition \eqref{eq:dioph} and the two-monotonicity interval structure of $v$ that the sequence $V_m = v(\theta_* + m\alpha)$ can only recur to a neighborhood of $E$ for at most two nearby values of $m$:  
\begin{align*}
	\min_{m_1,m_2,m_3 \in \mathcal{S}_0}\{|m_1 - m_2|, |m_2 - m_3|\} \gtrsim \rho^{-1/2\tau}
\end{align*}
Thus, if we choose $L$ of order, e.g.,
\begin{align*}
	L^6 \sim \rho^{-1/2\tau}, 
\end{align*}
then resonant sites $m \in \mathcal{S}_0$ which are not simple resonant appear in distinct pairs $(m_1, m_2)$ with $|m_1 - m_2| \leq L$ such that each pair is separated from other pairs by at least $L^6$.  

Given the separation of these pairs, one can build an interval $\Lambda^{(2)}_1 = [m_1 - L^2/2, m_1 + L^2/2]$ such that, by classical perturbation theory (cf. Lemma \ref{l:approxevect} and Proposition \ref{pr:AL2}), $H^{\Lambda^{(2)}_1}$ has two Rellich functions $\mathbf{E}_{1,\vee} > \mathbf{E}_{1,\wedge}$ near $E_*$ for $\theta$ near $\theta_*$.  These new Rellich functions, however, must necessarily deviate significantly from the parent function $v$ near their the crossing point $E_n(v)$.  One of the primary technical thrusts of this paper is demonstrating that these new functions locally 1) retain the crucial Morse and two-monotonicity structural properties of the function $v$, and 2) separate with a stable, quantifiable gap between them:
\begin{thm*}[cf. Theorem \ref{t:DRRelFns}]
In our setting, double resonances of a Rellich function $\mathbf{E}_s$ of $H^{\Lambda_s}$ resolve as a pair of uniformly locally separated Morse Rellich functions $\mathbf{E}_{s+1,\vee} > \mathbf{E}_{s+1,\wedge}$ of $H^{\Lambda_{s+1}^{(2)}}$ with at most one critical point, cf. Figure \ref{f:DRRelFn}.
\end{thm*}

The gap demonstrated in the above Theorem serves two purposes: First, it ensures that any Rellich function $\mathbf{E}_1$ can only resonate with itself at future scales, which ultimately enables our induction.  Second, the size of the gap is sufficiently large to remain open through each step of the inductive procedure; the ubiquity of these gaps yields Cantor spectrum.  

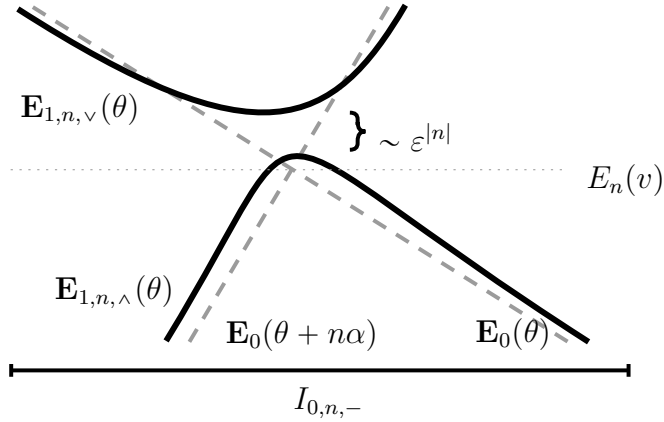
\begin{figure}[h!]

	\begin{center}
	\begin{tikzpicture}[x=0.6pt,y=0.6pt,yscale=-0.6,xscale=0.7]
		
		\draw [color={rgb, 255:red, 155; green, 155; blue, 155 }  ,draw opacity=1 ][line width=1.5]  [dash pattern={on 5.63pt off 4.5pt}]  (78.5,21) -- (560.5,373) ;
		\draw [color={rgb, 255:red, 155; green, 155; blue, 155 }  ,draw opacity=1 ][line width=1.5]  [dash pattern={on 5.63pt off 4.5pt}]  (219.5,373) -- (402.5,13) ;
		\draw [color={rgb, 255:red, 0; green, 0; blue, 0 }  ,draw opacity=1 ][line width=2.25]    (65,23) .. controls (283,186) and (348,150) .. (414,19) ;
		\draw [line width=2.25]    (198.5,372) .. controls (342,98) and (251.5,128) .. (579,373) ;
		\draw  [line width=1.5]  (364,171) .. controls (368.67,171) and (371,168.67) .. (371,164) -- (371,161.78) .. controls (371,155.11) and (373.33,151.78) .. (378,151.78) .. controls (373.33,151.78) and (371,148.45) .. (371,141.78)(371,144.78) -- (371,139) .. controls (371,134.33) and (368.67,132) .. (364,132) ;
		\draw [color={rgb, 255:red, 155; green, 155; blue, 155 }  ,draw opacity=1 ] [dash pattern={on 0.84pt off 2.51pt}]  (58.5,192) -- (558.5,192) ;
		\draw [line width=1.5]    (59,403) -- (615,403) ;
		\draw [shift={(615,403)}, rotate = 180] [color={rgb, 255:red, 0; green, 0; blue, 0 }  ][line width=1.5]    (0,6.71) -- (0,-6.71)   ;
		\draw [shift={(59,403)}, rotate = 180] [color={rgb, 255:red, 0; green, 0; blue, 0 }  ][line width=1.5]    (0,6.71) -- (0,-6.71)   ;
		
		\draw (387.08,140.63) node [anchor=north west][inner sep=0.75pt]  [rotate=-359.58,xslant=0.02]  {$\sim \varepsilon ^{|n|}$};
		\draw (250,346.4) node [anchor=north west][inner sep=0.75pt]    {$\mathbf{E}_0( \theta +n\alpha )$};
		\draw (65,108.4) node [anchor=north west][inner sep=0.75pt]    {$\mathbf{E}_{1,n,\lor }( \theta )$};
		\draw (475,345.4) node [anchor=north west][inner sep=0.75pt]    {$\mathbf{E}_0( \theta )$};
		\draw (96,300.4) node [anchor=north west][inner sep=0.75pt]    {$\mathbf{E} _{1,n,\land }( \theta )$};
		\draw (575,183.4) node [anchor=north west][inner sep=0.75pt]    {$E_{n}( v)$};
		\draw (311,418.4) node [anchor=north west][inner sep=0.75pt]    {$I_{0,n,-}$};

	\end{tikzpicture}
	\end{center}
	\caption{The resolution of a double resonance of $\mathbf{E}_0 = v$ into a pair of uniformly locally well-separated Rellich curves of a Dirichlet restriction $H^{\Lambda_1^{(2)}}$ of $H$.  The interval $\Lambda_1^{(2)}$ depends on the crossing value $E_n(v)$.}

\label{f:DRRelFn}
\end{figure}

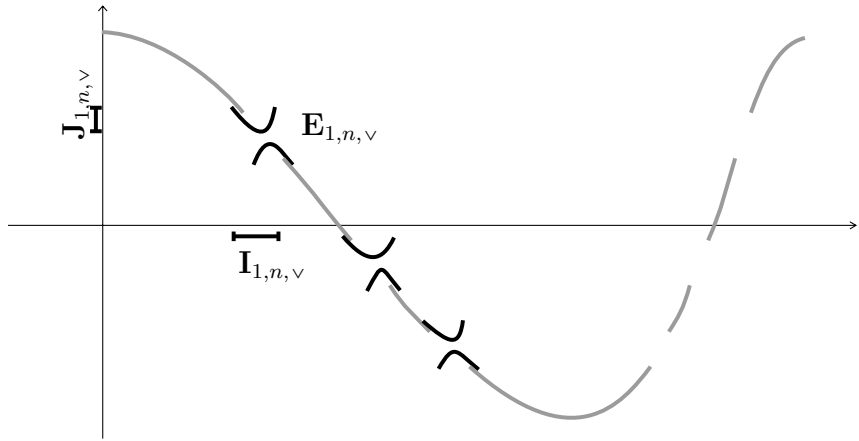
\begin{figure}[h!b]
\begin{center}
	
	\begin{tikzpicture}[x=0.6pt,y=0.6pt,yscale=-0.45,xscale=0.55]
		
		\draw  (-21,307.94) -- (951,307.94)(87.6,0.5) -- (87.6,607) (944,302.94) -- (951,307.94) -- (944,312.94) (82.6,7.5) -- (87.6,0.5) -- (92.6,7.5)  ;
		\draw [line width=1.5]    (79.74,143.04) -- (79.74,176) ;
		\draw [shift={(79.74,176)}, rotate = 270] [color={rgb, 255:red, 0; green, 0; blue, 0 }  ][line width=1.5]    (0,6.71) -- (0,-6.71)   ;
		\draw [shift={(79.74,143.04)}, rotate = 270] [color={rgb, 255:red, 0; green, 0; blue, 0 }  ][line width=1.5]    (0,6.71) -- (0,-6.71)   ;
		\draw [line width=1.5]    (237.66,323.31) -- (289,323.31) ;
		\draw [shift={(289,323.31)}, rotate = 180] [color={rgb, 255:red, 0; green, 0; blue, 0 }  ][line width=1.5]    (0,6.71) -- (0,-6.71)   ;
		\draw [shift={(237.66,323.31)}, rotate = 180] [color={rgb, 255:red, 0; green, 0; blue, 0 }  ][line width=1.5]    (0,6.71) -- (0,-6.71)   ;
		\draw [color={rgb, 255:red, 0; green, 0; blue, 0 }  ,draw opacity=1 ][line width=1.5]    (235,142) .. controls (273,201) and (280,173) .. (285,142) ;
		\draw [color={rgb, 255:red, 0; green, 0; blue, 0 }  ,draw opacity=1 ][line width=1.5]    (260.5,223) .. controls (277.5,169) and (290.5,203) .. (305.5,223) ;
		\draw [color={rgb, 255:red, 155; green, 155; blue, 155 }  ,draw opacity=1 ][line width=1.5]    (294.5,214) .. controls (343,281) and (327.5,264) .. (372.5,329) ;
		\draw [color={rgb, 255:red, 155; green, 155; blue, 155 }  ,draw opacity=1 ][line width=1.5]    (86.5,37) .. controls (158.45,39.58) and (228.5,119) .. (248.5,150) ;
		\draw [color={rgb, 255:red, 0; green, 0; blue, 0 }  ,draw opacity=1 ][line width=1.5]    (362.5,323) .. controls (390.5,363) and (407.5,361) .. (421.5,325) ;
		\draw [color={rgb, 255:red, 0; green, 0; blue, 0 }  ,draw opacity=1 ][line width=1.5]    (390.5,400) .. controls (409.5,357) and (404.5,364) .. (428.5,399) ;
		\draw [color={rgb, 255:red, 155; green, 155; blue, 155 }  ,draw opacity=1 ][line width=1.5]    (416.5,392) .. controls (435.5,427) and (442.5,431) .. (462.5,457) ;
		\draw [color={rgb, 255:red, 0; green, 0; blue, 0 }  ,draw opacity=1 ][line width=1.5]    (454.5,442) .. controls (485.5,475) and (496.5,480) .. (500.5,441) ;
		\draw [color={rgb, 255:red, 0; green, 0; blue, 0 }  ,draw opacity=1 ][line width=1.5]    (472.5,509) .. controls (490.5,471) and (492.5,484) .. (518.5,510) ;
		\draw [color={rgb, 255:red, 155; green, 155; blue, 155 }  ,draw opacity=1 ][line width=1.5]    (508,506) .. controls (627.5,636) and (682.5,560) .. (715.5,506) ;
		\draw [color={rgb, 255:red, 155; green, 155; blue, 155 }  ,draw opacity=1 ][line width=1.5]    (736.5,457) .. controls (746.5,438) and (753.5,425) .. (760.5,392) ;
		\draw [color={rgb, 255:red, 155; green, 155; blue, 155 }  ,draw opacity=1 ][line width=1.5]    (781.5,329) .. controls (798.5,281) and (800.5,262) .. (812.5,214) ;
		\draw [color={rgb, 255:red, 155; green, 155; blue, 155 }  ,draw opacity=1 ][line width=1.5]    (830.5,150) .. controls (842.5,110) and (859.5,54) .. (892.46,45.22) ;
		
		\draw (38.4,191.6) node [anchor=north west][inner sep=0.75pt]  [rotate=-270.02]  {$\mathbf{J}_{1,n,\lor }$};
		\draw (240.32,340.36) node [anchor=north west][inner sep=0.75pt]    {$\mathbf{I}_{1,n,\lor }$};
		\draw (312,147.4) node [anchor=north west][inner sep=0.75pt]    {$\mathbf{E}_{1,n,\lor }$};

	\end{tikzpicture}
\end{center}
	\caption{A cartoon output of the first inductive step: a collection $\mathcal{E}_1$ of local Rellich functions of various Dirichlet restrictions of $H$, whose domains (and their relevant translates, e.g. $\mathbf{I}_{1,n,\vee} + n\alpha$) cover the circle $\mathbb{T}$.  The curves in black come from double resonances, and the curves in gray are simple resonant.  Note that different curves need not agree on the overlap of their domains.}  
\label{f:RelCollection}
\end{figure}

From the initial scale function $\mathbf{E}_0 = v$, we thus construct a collection $\mathcal{E}_1$ of well-separated Rellich functions of certain Dirichlet restrictions of $H$ whose domains cover the circle $\mathbb{T}$ with the same structural properties as $\mathbf{E}_0$, cf. Figure \ref{f:RelCollection}.
The inductive argument proceeds on each such constructed function: by properly defining our parameters at scale $s$ relative to our initial choice of $\varepsilon =: \delta_0$, we can ensure that the size $\delta_s$ of the resonant eigenvectors near the edges of $\Lambda_s$ is much smaller than the eigenvalue separation $\rho_s$, allowing the procedure above to iterate; we choose these parameters in such a way so as to ensure the bad sets of phases have summable measures (cf. Lemma \ref{l:badsets}).  The Borel-Cantelli lemma guarantees that the collection $B$ of $\theta_* \in \mathbb{T}$ which are double-resonant infinitely often has zero measure; our full-measure set of localized phases is the complement $\Theta = \mathbb{T} \setminus B$.

The ultimate output of the induction is a tree of cosine-like Rellich curves, cf. Figure \ref{f:Rellichtree}.  Each fixed energy $E_* \in \mathbb{R}$ admits a path of ``$E_*$-relevant'' Rellich functions through this tree; this path either terminates at finite depth if $E_*$ is non-spectral or proceeds without end.  Given $\theta_* \in \Theta$, we consider the infinite path determined by some generalized eigenvalue $E_*= E(\theta_*)$ of $H(\theta_*)$; by our construction of $\Theta$, this path eventually consists \textit{only of simple-resonant children}.  The resulting stability of the approximate eigenfunctions $\psi_s$ will yield exponential localization of the corresponding generalized eigenfunction, showing that the generalized eigenvalue is an honest eigenvalue with decaying eigenfunction.  The spectral density of generalized eigenvalues guaranteed by Schnol's lemma completes the localization argument.  
Moreover, our construction ensures that each node of the tree \textit{has some double-resonant descendant}. Since each double resonance opens a spectral gap which will remain open for all future scales, this observation guarantees that spectral gaps open arbitrarily close to any fixed energy $E_* \in \mathbb R$; thus the spectrum is a Cantor set.

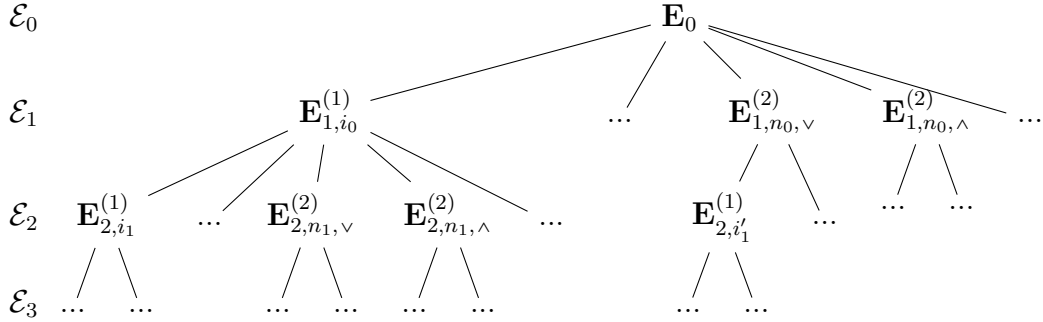
\begin{figure}
\begin{forest}
	my edge label/.style={
		edge label={node [midway,above,sloped, font=\tiny] {}}},
	for tree={%
		fit=band},
	before drawing tree={
		tikz+={
			\node [anchor=mid east, font=\sffamily] (m) at (current bounding box.west |- .mid) {$\mathcal{E}_0$};
		},
		tempcounta'=0,
		for tree={
			if={%
				>OR>{level}{tempcounta}%
			}{%
				tempcounta'+=1,
				tikz+/.process={ Rw {tempcounta} {
						\node [anchor=mid east, font=\sffamily] at (m.east |- .mid) {$\mathcal{E}_{#1}$};
					}%
				},
			}{},
		},
	},
	[$\mathbf{E}_0$
		[$\mathbf{E}_{1,i_0}^{(1)}$
				[$\mathbf{E}_{2,i_1}^{(1)}$
					[\text{...}
					]
					[\text{...}
					]
				]
				[\text{...}
				]
				[$\mathbf{E}_{2,n_1,\vee}^{(2)}$
					[\text{...}
					]
					[\text{...}
					]
				]
				[$\mathbf{E}_{2,n_1,\wedge}^{(2)}$
					[\text{...}
					]
					[\text{...}
					]
				]
				[\text{...}
				]
		]
		[\text{...}]
		[$\mathbf{E}_{1,n_0,\vee}^{(2)}$
			[$\mathbf{E}_{2,i_1'}^{(1)}$
				[\text{...}
				]
				[\text{...}
				]
			]
			[\text{...}
			]
		]
		[$\mathbf{E}_{1,n_0,\wedge}^{(2)}$
			[\text{...}
			]
			[\text{...}]
		]
		[\text{...}]
	]
\end{forest}
\caption{After completing the induction, we have a tree $\mathcal{E}$ of Rellich functions $\mathbf{E}_s : \mathbf{I}_s \to \mathbf{J}_s$.  Any energy $E_* \in \mathbb{R}$ admits a path through this tree.  The superscript $(j)$, $j \in \{1,2\}$, of each child indicates the resonance type relative to the parent.}
\label{f:Rellichtree}
\end{figure}

\subsubsection{A brief roadmap of the paper}

The article proceeds as follows: in the remainder of this section, we will establish some basic notation and foundational lemmas for use throughout the paper; the proofs of these lemmas can all be found in Appendix B.  Sections 2, 3, and 4 provide the crucial infrastructure for our induction.  In Sections 2 and 3, we develop the spectral tools to analyze descendants of a given Rellich curve in the simple- and double-resonant settings, respectively, at every scale.  The by-now classical but necessary details for proving Green's function decay on nonresonant intervals are included in Appendix A.  Section 4 consists of technical machinery allowing us to classify energy regions $J_{s,i}$ in the codomain of a cosine-like function $\mathbf{E}_s$ as simple- or double-resonant.  Section 5 applies the content of Sections 2, 3, and 4 in a multiscale induction scheme: given a Rellich function $\mathbf{E}_s$, we use the restrictions of $\mathbf{E}_s$ to the resonance regions $J_{s,i}$ to construct Rellich children $\mathbf{E}_{s+1}$ according to the recipes in Sections 2 and 3.    Section 6 characterizes the spectrum of $H$ in terms of the complete collection of Rellich functions constructed in Section 5, proving Cantor spectrum via the permanent stability of the gaps coming from double resonances; finally, we eliminate double resonances and prove Anderson localization as outlined above.

\subsection{Notation and preliminaries}

In this subsection, we establish some notation and collect the foundational lemmas which will serve as the starting point for our discussion; the proofs of each of these lemmas can be found in Appendix B.

Throughout the paper we use $I$ to denote intervals of phases $\theta$, $J$ to denote intervals of energies $E$, $\Lambda$ to denote intervals of integers, and we denote by $B_r(x)$ the open ball of radius $r$ centered at $x$.  We reserve a handful of parameters for our multiscale inductive procedure; these parameters are:
\begin{enumerate}
    \item[] $L$: Lengths of spacial intervals $\Lambda$.
    \item[] $\rho$: Resonance scale for phase and energy intervals $I$ and $J$.
    \item[] $\delta$: The size of a localized eigenfunction at the edge of an interval $\Lambda$.
    \item[] $\nu$: Local lower bound on the magnitude of a first derivative of a Rellich function.
    \item[] $(\ell,\gamma)$: Green's function decay parameters, see below.
\end{enumerate}
Due to the nature of the induction, we will occasionally consider up to three scales simultaneously, a ``current'' scale, a ``next'' scale, and a ``previous'' scale; we use hats $\snext{\cdot}$ or checks $\sprev{\cdot}$ above the corresponding parameters when we intend to suggest application to the next or previous scales, respectively.  With this convention, one should keep in mind the following general scale principles, which we clarify more precisely as necessary:
\begin{align*}
    \delta \ll \rho \ll \sprev{\rho} \ll \nu, \\
    \gamma \sim |\log\varepsilon|, \\
    L \ll \snext\ell \ll \snext{L}. 
\end{align*}
\indent We now fix the notation for our objects of study.  Let $V \in \ell^\infty(\mathbb{Z},\mathbb{R})$ be a real-valued potential.  Restricted to $\ell^2(\mathbb{Z})$, the linear operator $H(V,\varepsilon)$ is self-adjoint and bounded (with $\|H\| \leq \|V\|_\infty + 2\varepsilon$).  For an interval $\Lambda = [a,b] \subset \mathbb{Z}$ we denote by $H^{\Lambda}$ the restriction of $H$ to $\Lambda$ with Dirichlet boundary conditions, $\psi(a-1) = \psi(b+1) = 0$.  Letting $|\Lambda| := b-a+1$, one identifies $H^{\Lambda}$ with a $|\Lambda| \times |\Lambda|$ matrix by making an appropriate change of basis.  With this identification, we will frequently treat vectors $\sprev{\psi}$ defined only on a subset $\sprev{\Lambda} \subset \Lambda$ as though they are vectors in $\mathbb{C}^{|\Lambda|}$; in this situation, we abuse notation and conflate $\sprev{\psi}$ with its extension to $\Lambda$ given by $\sprev{\psi}(m) = 0$, $m \in \Lambda \setminus \sprev{\Lambda}$.

Consider a partition of $\Lambda$ given by $\mathcal{P} = \cup_{j=0}^p \Lambda_j$, $\Lambda_j = [a_j,b_j]$, where $b_{j} + 1 = a_{j+1}$ for all $0 \leq j \leq p-1$ and $a_0 = a, b_p = b$.  We define the partitioned operators $H^\Lambda_\mathcal{P}$ and $\Gamma^\Lambda_{\mathcal{P}}$ by
\begin{align*}
	H^\Lambda_\mathcal{P} &:= \bigoplus_{j} H^{\Lambda_j}, \\
	\Gamma^\Lambda_\mathcal{P} &:= H^\Lambda - H^\Lambda_\mathcal{P}.
\end{align*}
  We fix once and for all the trivial partition $\mathcal{P}_0$ with $a_j = b_j = j$, and note that in this special case
\begin{align*}
	H^\Lambda_{\mathcal{P}_0} &= V^\Lambda, \\
	\Gamma^\Lambda_{\mathcal{P}_0} &= \varepsilon\Delta^\Lambda.
\end{align*}

We denote by $\spec(A)$ the spectrum of an operator $A$.  For $E \notin \spec(H^{\Lambda})$, we denote by $R^{\Lambda}(E) := (H^{\Lambda}-E)^{-1}$ the resolvent operator and by $R^{\Lambda}_E(m,n) := \langle \delta_m, R^\Lambda(E)\delta_n\rangle$ the associated Green's function.  Given a partition $\mathcal{P}$ of $\Lambda$, we define $R^\Lambda_\mathcal{P}$ analogously, with $H^\Lambda_\mathcal{P}$ replacing $H^\Lambda$.  We will make frequent use of the resolvent identity
\begin{align*}
	R^\Lambda(E) - R^\Lambda_\mathcal{P}(E) &= -R^\Lambda(E) \,\Gamma_\mathcal{P}^\Lambda \,R^\Lambda_\mathcal{P}(E) .
\end{align*}
Let $E \notin \spec H^\Lambda$; given $\ell \in \mathbb{N}, \gamma>0$, we say that an interval $\Lambda \subset \mathbb{Z}$ satisfies the \textit{Green's function decay property for $(\ell,\gamma)$} if, for all 
$m,n\in \Lambda$ with $|m-n|\geq\ell$, $$\log|R^{\Lambda}(m,n)| \leq -\gamma|m-n|.$$

\subsubsection{Perturbation theory for differentiable one-parameter self-adjoint operator families}

We now recall some important lemmas for use throughout the paper; the proofs of these lemmas are appended.

We begin with a straightforward consequence of rank-nullity:

\begin{lem}\label{l:PQclose}
	Let $\chi$ and $P$ each be orthogonal projections on a finite-dimensional vector space.  If
	\begin{align*}
		\|(I-\chi)P\| < 1
	\end{align*}
	then $\rank(\chi) \geq \rank(P)$.
\end{lem}
We apply this lemma in the context of spectral projections for real symmetric matrices.  For a symmetric matrix $A$ having normalized eigenpairs $(\lambda_j, \psi_j)$, $\|\psi_j\| = 1$, and a Borel set $J \subset \mathbb{R}$, we denote by $\chi_{J}(A)$ the spectral projection
\begin{align*}
	\chi_J(A) := \sum_{\lambda_j \in J} \psi_j \psi_j^\top
\end{align*}  
onto the direct sum of those eigenspaces of $A$ with $\lambda_j \in J$.  We have the following fundamental perturbative result relating eigenvectors of a symmetric matrix $A$ having quantitatively separated spectrum to approximations thereof:

\begin{lem}\label{l:approxevect}
	Let $A$ be a real symmetric matrix, and let $E_* \in \bbR$ and $\delta > 0$.  Suppose there exists a nonzero orthogonal projection $P$ such that
	\begin{align*}
		\|(A - E_*)P\| \leq \delta.
	\end{align*}
	Then $\chi_{\overline B_{\delta}(E_*)}(A) \neq 0$, and for any $\rho > \delta$, %
	we have
	\begin{align}\label{eq:approxevect1}
		\left\|(I - \chi_{ B_{\rho}(E_*)}(A))P\right\| \leq \delta/\rho.
	\end{align}
	In particular, for any unit vector $\phi$ in the image of $P$, the spectral renormalization
	\begin{align*}
		\psi := \frac{\chi_{ B_{\rho}(E_*)}(A)\phi}{\|\chi_{ B_{\rho}(E_*)}(A)\phi\|}
	\end{align*}
	is well-approximated by $\phi$:
	\begin{align}\label{eq:approxevect2}
		\left\|\phi - \psi\right\| \leq \sqrt{2}\delta/\rho.
	\end{align}
\end{lem}

For a Borel set $J \subset \mathbb{R}$, we denote the \textit{partial resolvent}
\begin{align*}
	R_\perp(E;A,J) := \sum_{\lambda_j \notin J} \frac{\chi_{\{\lambda_j\}}(A)}{\lambda_j - E}
\end{align*}
and sometimes abbreviate the special case
\begin{align*}
	R_\perp(\lambda_j;A) := \sum_{\lambda_k \neq \lambda_j} \frac{\chi_{\{\lambda_k\}}(A)}{\lambda_k - \lambda_j}.
\end{align*}
Note that, by definition,
\begin{align}\label{eq:projdecomp}
	R_\perp(E;A,J)(A - E) = I - \chi_{J}(A)
\end{align}
is the spectral projection off of the eigenspaces of $A$ associated to $\lambda_j \in J$.  Furthermore, one has that
\begin{align*}
	\|R_\perp(\lambda_j;A)\| &= \frac{1}{\dist(\lambda_j, \spec A \setminus \{\lambda_j\})}.
\end{align*}

The partial resolvents $R_\perp$ are of great importance in computing derivatives of eigenvectors for smooth one-parameter families of symmetric matrices:

\begin{lem}[Feynman formulae]\label{l:feynman}
	Let $A(\theta)$ denote a one-parameter family of symmetric real matrices, and suppose that $A$ is $C^2$ at $\theta_*$, and that $(E_*,\psi_*)$ is a simple eigenpair of $A(\theta_*)$.   Then there is a family $(\mathbf{E},\psi)(\theta)$ of locally simple eigenpairs in a neighborhood of $\theta_*$, normalized so that $\|\psi(\theta)\| = 1$, which is twice differentiable at $\theta_*$.  Furthermore, one has
	\begin{align}\label{eq:feynman1}
		\partial_\theta \mathbf{E}(\theta_*) &= \langle \psi, A'(\theta_*) \psi \rangle,
	\end{align}
	\begin{align}\label{eq:feynmanEvec}
		\partial_\theta \psi(\theta_*) &= -R_\perp(A,E_*)A'(\theta_*)\psi_*
	\end{align}
	and
	\begin{align}\label{eq:feynman2}
		\partial_\theta^2 \mathbf{E}(\theta_*) &= \langle \psi, A''(\theta_*)\psi \rangle + 2\langle \partial_\theta \psi, A'(\theta_*)\psi\rangle.
	\end{align}
	Finally, one has that
	\begin{align}\label{eq:orthderiv}
		\langle \partial_\theta \psi , \psi \rangle = 0.
	\end{align}
\end{lem}

\subsubsection{Spectral analysis of discrete Schr\"odinger operators}

We recall some classical lemmas regarding eigenfunctions of discrete Schr\"odinger operators.  The first, which is sometimes referred to as the ``Poisson formula," relates the value of an eigenfunction at a particular site $n$ to Green's functions of Dirichlet restrictions of $H$:

\begin{lem}[Poisson Formula]\label{l:poisson}
	Let $\psi \in \mathbb{C}^\mathbb{Z}$ satisfy the formal Schr\"odinger difference equation $H\psi = E\psi$, and let $\Lambda = [a,b] \subset \mathbb{Z}$ be an interval.  Then%
	\begin{align}\label{eq:poisson}
		\psi(n) = \varepsilon R^\Lambda_{\theta,E}(n,a)\psi(a-1) + \varepsilon R^{\Lambda}_{\theta,E}(n,b)\psi(b+1), \quad n \in \Lambda
	\end{align}
	provided $E \notin \spec(H^\Lambda)$.
	
	The same relation holds if we replace $H$ with $H^{\hat\Lambda}$ for some interval $\hat\Lambda = [\hat a, \hat b] \supsetneq \Lambda$ and let $\psi \in \mathbb C^{|\hat\Lambda|}$ satisfy $H^{\hat\Lambda}\psi=E\psi$. In this case, we use the convention that $\psi(\hat a-1) = \psi(\hat b + 1) = 0$.
\end{lem}

The final result that we recall is a classical eigenvalue separation lemma (cf. e.g. \cite{FroSpeWit90, KirSim85}), which quantifiably separates the eigenvalues of simultaneously localized eigenvectors of a discrete Schr\"odinger operator:
\begin{lem}\label{l:evalsep}
	Let $\snext\psi_1$, $\snext\psi_2$ be distinct eigenfunctions for $H^{\snext\Lambda}$ with corresponding eigenvalues $\snext{E}_1$, $\snext{E}_2$.  Suppose there is an interval ${\Lambda} \subset \snext\Lambda$ supporting half of the mass of both $\snext\psi_1$ and $\snext\psi_2$, i.e., such that $\|\snext\psi_j\|_{{\Lambda}}^2 \geq \frac{1}{2}\|\snext\psi_j\|_{\snext\Lambda}^2$, $j = 1,2$.  Then
	\begin{align*}
		|\snext{E}_1 - \snext{E}_2| \geq \frac{\varepsilon}{6|{\Lambda}|^2}\left(\left(\frac{m_{{\Lambda}}}{\varepsilon}\right)^2 + 2 \right)^{-|{\Lambda}|} > \left(\frac{\varepsilon}{3}\right)^{3|{\Lambda}|}
	\end{align*}
	where $m_{{\Lambda}} := \max_{k \in {\Lambda}, j =1,2}|v(k) - \snext E_j|^2$.
	
	In particular, if $2\rho < (\varepsilon/3)^{3|\Lambda|}$, $|\snext{E}_1 - E_*| < \rho/2$, and $E_*$ is well-separated from any other eigenvalues of $H^{\snext\Lambda}$ (i.e., $\|R_\perp(E_*;H^{\snext\Lambda},\{\snext{E}_1,\snext{E}_2\})\| \leq \frac23\rho^{-1}$), we have the partial resolvent bound
	\begin{align}
		\|R_\perp(E; H^{\snext\Lambda}, \{\snext{E}_1\})\| \leq \rho^{-1}, \quad E \in B_{\rho/2}(E_*).
	\end{align}
\end{lem}

\subsection{Acknowledgments}

The authors would like to express their deep and sincere gratitude to Michael Goldstein and Wilhelm Schlag for their guidance and numerous helpful conversations throughout the duration of this project.
The work of Y.F. was supported in part by a fellowship award through the National Defense Science and Engineering Graduate (NDSEG) Fellowship Program, sponsored by the Air Force Research Laboratory (AFRL), the Office of Naval Research (ONR) and the Army Research Office (ARO).
\newpage

\section{Simple resonance}\label{sec:SR}

In this section, we study the situation of simple resonance, where an interval $\snext\Lambda \subset \mathbb{Z}$ contains precisely one resonant ``core" subinterval ${\Lambda} \subset \snext\Lambda$ such that $H^{\Lambda}(\theta_*)$ has a unique eigenvalue near a fixed $E_*$.  Under smallness and stability assumptions on the boundary values of the corresponding eigenvector, alongside off-diagonal Green's function decay on the maximal connected components of $\snext\Lambda \setminus {\Lambda}$ (the ``shoulders'' of $\snext\Lambda$), we prove that $H^{\snext\Lambda}$ has a unique resonant eigenpair child which is well-approximated by the parent eigenpair of $H^{{\Lambda}}$.

Fix $v \in C^2(\mathbb{T}, [-1,1])$ with $\|\partial_\theta v\|_\infty+\|\partial_\theta^2 v\|_\infty \leq D_0$, $(\theta_*, E_*) \in \mathbb{T} \times [-2,2]$, and $0 < \varepsilon < 1/7$.  Let $\snext\Lambda \subset \mathbb{Z}$ be an interval with subinterval ``core'' ${\Lambda}$ of length ${L} = |{\Lambda}|$.  We denote the left and right shoulders of $\snext\Lambda$ by ${\Lambda}_{l/r}$, respectively, and denote the corresponding partition of $\snext\Lambda = {\Lambda}_l \cup {\Lambda} \cup {\Lambda}_r$ by $\mathcal{P}$.  Fix constants $\rho, \delta, \gamma, \ell > 0$ satisfying the following relations:
\begin{align}
\label{eq:eps1} \delta < \rho^3/2 &< \rho/16, \\
\label{eq:gam1} \log 7 < \gamma &\leq |\log\varepsilon|
\end{align}
\begin{assm}\label{as:SR}
Suppose the following hold for $|\theta - \theta_*| < \rho/8D_0$:
\begin{enumerate}
\item (Simple resonance) The operator $H^{{\Lambda}}(\theta)$ has a Rellich pair $({\mathbf{E}},{\psi})(\theta)$ such that 
${\mathbf E}(\theta_*) = E_*$.
\item (Eigenvector decay) The unit eigenvector ${\psi}(\theta)$ of $H^{{\Lambda}}(\theta)$ has $\mathcal{P}$-boundary values bounded by $\delta/\varepsilon$; i.e. one has
\begin{align}\label{eq:SRloc}
\|\Gamma^{\snext\Lambda}_{\mathcal{P}} {\psi}\| \leq 2\delta.
\end{align}
Furthermore, for any Rellich pair $(\snext E,\snext\psi)(\theta)$ of $H^{\snext\Lambda}(\theta)$ with $|\snext E-E_*| < \frac32\rho$, the unit eigenvector $\snext\psi$ also has $\mathcal{P}$-boundary values no larger than $\delta/\varepsilon$; i.e. one has
\begin{align}\label{eq:SRloc2}
    \|\Gamma_\mathcal{P}^{\snext\Lambda} \snext\psi\| \leq 4\delta.
\end{align}
\item (Green's function decay on the shoulders of $\Lambda$)
For $m, n \in {\Lambda}_{l/r}$ with $|m-n| \geq \ell$, one has
\begin{align}\label{eq:SRdec}
\log|R^{{\Lambda}_{l/r}}_{\theta,E}(m,n)| \leq -\gamma|m-n|
\end{align}
for $|E-E_*| < \frac32\rho$.
\item (Eigenvalue separation) ${\mathbf{E}}(\theta)$ is the unique eigenvalue of $H^{\snext\Lambda}_\mathcal{P}(\theta)$ in $B_{7\rho/4}(E_*)$
\begin{align}\label{eq:SRsep2}
\chi_{\{{\mathbf{E}}(\theta)\}}(H^{\snext\Lambda}_\mathcal{P}(\theta)) &= \chi_{B_{7\rho/4}(E_*)}(H^{\snext\Lambda}_\mathcal{P}(\theta)).
\end{align}
In particular, we have the partial resolvent bound
\begin{align}
\label{eq:SRsep}
\|R_\perp(E; H^{\snext\Lambda}_\mathcal{P}(\theta), \{{\mathbf{E}}(\theta)\})\| \leq 4\rho^{-1}, \quad E \in B_{3\rho/2}(E_*).
\end{align}
\item (Eigenvector stability) The unit eigenvector ${\psi}(\theta)$ of $H^{{\Lambda}}(\theta)$ has stably small $\mathcal{P}$-boundary values:
\begin{align}\label{eq:SRevecstab}
    \|\Gamma_\mathcal{P}^{\snext\Lambda}(\partial_\theta {\psi})\| \leq 25D_0\delta\rho^{-1}.
\end{align}
\end{enumerate}
\end{assm}
\noindent The remainder of this section proceeds under Assumption \ref{as:SR}.

\subsection{The resonant eigenpair of $H^{\snext\Lambda}$}

Under Assumption \ref{as:SR}, we first show the existence of a unique eigenpair of $H^\Lambda$ which is localized near near ${\Lambda}$.  Specifically, denoting the localization region $\Lambda_{\textnormal{loc}}$ and modified shoulders $\snext\Lambda_{l/r}$ by
\begin{align*}
    \Lambda_{\text{loc}} &:= [\inf {\Lambda} - \ell, \sup {\Lambda} + \ell] \cap \snext\Lambda, \\
    \snext\Lambda_{l/r} &:= {\Lambda}_{l/r} \setminus \Lambda_{\text{loc}},
\end{align*}
we have the following:

\begin{prop}[Simple resonant eigenpair]\label{pr:AL1}
Under Assumption \ref{as:SR}, we have:
\begin{enumerate}
\item Any unit eigenvector $\snext\psi$ of $H^{\snext\Lambda}$ with corresponding eigenvalue $\snext{E}$ with $|\snext{E} - E_*| < \frac32\rho$ is Anderson localized on $\Lambda_{\textnormal{loc}}$:
\begin{align*}
\|\snext\psi\|_{\Lambda_{\textnormal{loc}}} &\geq 2/3, \\
\log|\snext\psi(j)| &\leq -\gamma \dist(j, {\Lambda}), \quad j \in \snext\Lambda_{l/r}
\end{align*}

\item There exists a Rellich pair $(\snext{\mathbf{E}}, \snext\psi)$ for $H^{\snext\Lambda}$ with $\snext{\mathbf{E}}: B_{\rho/8D_0(\theta_*)} \to B_{\rho/4}(E_*)$ and eigenvector $\snext\psi$ satisfying
\begin{align} \label{eq:SReval}
|\snext{\mathbf E} - \mathbf E| &\leq 2\delta, \\
\label{eq:SRevec}
\snext\psi &= {\psi} + \phi
\end{align}
where $\|\phi\| \leq %
3\delta/\rho$.

\item Uniformly for $\theta$ with $|\theta - \theta_*| < \rho/8D_0$ and $|E - E_*| < \frac54\rho$, there are no other Rellich functions of $H^{\snext\Lambda}$ within $\rho/4$ of $\snext{\mathbf{E}}$:
\begin{align}\label{eq:rperpbd}
\left\|R_\perp\left(E;H^{\snext\Lambda}(\theta),\{\snext{\mathbf{E}}(\theta)\}\right)\right\| \leq 4\rho^{-1}.
\end{align}
In particular, this applies to all $|E-\snext{\mathbf E}(\theta)| \leq \rho$.
\end{enumerate}
\end{prop}

\begin{proof}
Suppose $(\snext E, \snext\psi)$ is a Rellich pair for $H^{\snext\Lambda}$ such that $|\snext E - E_*| < \frac32\rho$ normalized so that $\|\snext\psi\| = 1$.  Then, denoting ${\Lambda}_l = [{a}_l, {b}_l]$, we have for $j \in \snext{\Lambda}_l$
\begin{align*}
\log|\snext\psi(j)| &= \log|\varepsilon R_{\theta, \snext E}^{{\Lambda}_l}({b}_l, j) \snext\psi({b}_l + 1)|\\
&\leq \log|\varepsilon| -\gamma|j-{b}_l| \\
&\leq -\gamma(|j-{b}_l|+1) \\
&= -\gamma \dist(j, {\Lambda})
\end{align*}
by the Poisson formula \eqref{eq:poisson} and \eqref{eq:SRdec}.  The case $j \in \snext\Lambda_r$ is entirely analogous.

Using this decay, we have that
\begin{align*}
    \|\snext\psi\|_{\snext{\Lambda}_{l/r}} &\leq \sum_{j\in \snext{\Lambda}_{l/r}} e^{-\gamma \dist(j,{\Lambda})} \\
    &\leq \frac{e^{-\gamma}}{1-e^{-\gamma}} < \frac{1}{6}.
\end{align*}
Consequently, we have that
\begin{align*}
    \|\snext\psi\|_{\Lambda_{\textnormal{loc}}} \geq 2/3
\end{align*}
for any resonant eigenvector $\snext\psi$.

Denoting by $P = \chi_{B_{3\rho/2}(E_*)}(H^{\snext\Lambda})$ and by $\chi = \chi_{B_{3\rho/2}(E_*)}(H^{\snext\Lambda}_{\mathcal{P}})$, suppose $(\snext{\mathbf E},\snext\psi)$ is a resonant eigenpair of $H^{\snext\Lambda}$ with $|\snext{\mathbf E} - E_*| < 3\rho/2$.  Then we have
\begin{align*}
    \|(I - \chi)P\| &= \|R_\perp(\mathbf E; H^{\snext\Lambda}_\mathcal{P}, {\mathbf E})(H^{\snext\Lambda}_\mathcal{P} - \snext{\mathbf E})P\| \\
    &\leq 4\rho^{-1}\|\Gamma^{\snext\Lambda}_\mathcal{P} P\| \leq 16\delta\rho^{-1}.
\end{align*}
Lemma \ref{l:PQclose} and assumption \eqref{eq:SRsep2} imply $H^{\snext\Lambda}$ has at most one such resonant eigenvector. %

The existence of a Rellich pair $(\snext{\mathbf{E}}, \snext\psi)$ with $\snext{\mathbf{E}} : B_{\rho/8D_0}(\theta_*) \to B_{\rho/4}(E_*)$ and satisfying \eqref{eq:SReval} follows from Lemma \ref{l:approxevect} after noting that
\begin{align}\label{eq:pertevalpf}
\|(H^{\snext\Lambda}(\theta) - {\mathbf{E}}(\theta)){\chi}\| \leq \|\Gamma^{\snext\Lambda}_{\mathcal{P}}{\psi}\| &\leq 2\delta
\end{align}
and, by \eqref{eq:feynman1},
\begin{align*}
|{\mathbf{E}}(\theta) - E_*| &\leq D_0|\theta - \theta_*| %
< \rho/8.
\end{align*}
Thus, $\snext{\mathbf{E}}$ is the unique eigenvalue of $H^{\snext\Lambda}$ with $|\snext{\mathbf{E}} - E_*| < 3\rho/2$, and \eqref{eq:rperpbd} holds.
Finally, since $\snext\psi$ is the unique unit vector (up to a sign) in the image of $\chi_{B_{\rho}({\mathbf E})}(H^{\snext\Lambda})$, \eqref{eq:approxevect2} implies \eqref{eq:SRevec}.
\end{proof}

\subsection{Local Rellich function structure}

In the simple resonant setting, the resonant Rellich pair $(\snext{\mathbf{E}},\snext\psi)$ of $H^{\snext\Lambda}$ is well-approximated by the resonant Rellich pair $({\mathbf{E}},{\psi})$ of $H^{{\Lambda}}$:

\begin{prop}\label{pr:interscaleapprox1}
Uniformly for $\theta \in B_{\rho/8D_0}(\theta_*)$, there is a constant $C = C(v)$ depending only on $v$ so that
\begin{align}\label{eq:evalapprox1}
\left|\partial_\theta^k (\snext{\mathbf{E}} - {\mathbf{E}})\right| \leq C\frac{\delta}{\rho^k}, \quad 0 \leq k \leq 2
\end{align}
and
\begin{align}\label{eq:evecapprox1}
\|\partial_\theta^k(\snext\psi - {\psi})\| \leq C\frac{\delta}{\rho^{k+1}}, \quad 0 \leq k \leq 1.
\end{align}
\end{prop}

\begin{proof}
By 
the previous proposition, 
we have
\begin{align*}
|\snext{\mathbf{E}} - {\mathbf{E}}| &\leq 2\delta, \\
\|\snext\psi - {\psi}\| &\leq 3\frac{\delta}{\rho}.
\end{align*}
Thus, by the Feynman formula \eqref{eq:feynman1} and Cauchy-Schwarz,
\begin{align*}
\left|\partial_\theta (\snext{\mathbf{E}} - {\mathbf{E}})\right| &= \left| \langle \snext\psi, V'\snext\psi \rangle - \langle {\psi}, V'{\psi}\rangle \right| \\
&\leq 2\|V'\|_\infty\|\snext\psi - {\psi}\| \\
&\leq 6D_0 \frac{\delta}{\rho}.
\end{align*}

By \eqref{eq:feynmanEvec}, \eqref{eq:SRsep}, and \eqref{eq:rperpbd}, we immediately get
\begin{align*}
\|\partial_\theta \psi\| \leq \frac{4D_0}{\rho}, \quad
\|\partial_\theta \snext\psi\| \leq \frac{4D_0}{\rho}.
\end{align*}
By differentiating the relation $H^{\snext\Lambda}\snext\psi = \snext{\mathbf{E}}\snext\psi$, we likewise find that
\begin{align*}
(H^{\snext\Lambda} - \snext{\mathbf{E}})(\partial_\theta (\snext\psi - {\psi})) &= (\partial_\theta \snext{\mathbf{E}} - V')\snext\psi - (\partial_\theta {\mathbf{E}} - V'){\psi} + (\Gamma^{\snext\Lambda}_\mathcal{P} + ({\mathbf{E}} - \snext{\mathbf{E}}))\partial_\theta{\psi} \\
&= (\partial_\theta \snext{\mathbf{E}} - V')(\snext\psi - {\psi}) + (\partial_\theta (\snext{\mathbf{E}} - {\mathbf{E}})){\psi} + (\Gamma^{\snext\Lambda}_\mathcal{P} + ({\mathbf{E}} - \snext{\mathbf{E}}))\partial_\theta{\psi}
\end{align*}
and so, by the triangle inequality, the assumption \eqref{eq:SRevecstab}, and the above estimates,
\begin{align*}
\|(H^{\snext\Lambda} - \snext{\mathbf{E}})(\partial_\theta (\snext\psi - {\psi}))\| &\leq 6D_0  \frac{\delta}{\rho} + 6D_0  \frac{\delta}{\rho} + 25D_0\frac{\delta}{\rho} + 8D_0 \frac{\delta}{\rho} \\
&\leq 45D_0\frac{\delta}{\rho}
\end{align*}
Denote by $\snext{\chi}$ the spectral projection onto $\snext\psi$.  Then, since $R_\perp^{\snext\Lambda}(\snext{\mathbf E})(H^{\snext\Lambda} - \snext{\mathbf E}) = I - \snext\chi$, we have
\begin{align*}
\|\partial_\theta (\snext\psi - {\psi})\| &\leq \|R_\perp^{\snext\Lambda}(\snext{\mathbf E})\|\|(H^{\snext\Lambda} - \snext{\mathbf E})(\partial_\theta (\snext\psi - {\psi}))\| + |\langle \partial_\theta(\snext\psi - {\psi}), \snext\psi \rangle| \\
&= \|R_\perp^{\snext\Lambda}(\snext{\mathbf E})\|\|(H^{\snext\Lambda} - \snext{\mathbf E})(\partial_\theta (\snext\psi - {\psi}))\| + |\langle \snext\psi - {\psi}, \partial_\theta\psi \rangle| \\
&\leq 180D_0\frac{\delta}{\rho^2} + 12D_0 \frac{\delta}{\rho^2} \\
&\leq 200 D_0 \frac{\delta}{\rho^2}
\end{align*}
where in the second line we have used that $\langle \partial_\theta \psi, \psi \rangle = 0$ (and similarly for $\snext{\psi}$).

Finally, by \eqref{eq:feynmanEvec}, \eqref{eq:feynman2}, and Cauchy-Schwarz, we have
\begin{align*}
\left|\partial_\theta^2(\snext{\mathbf{E}} - {\mathbf{E}}) \right| &\leq \left| \langle \snext\psi, V''\snext\psi \rangle - \langle {\psi}, V''{\psi}\rangle\right| + 2\left|\langle \partial_\theta\snext\psi, V'\snext\psi \rangle - \langle \partial_\theta{\psi}, V'{\psi}\rangle \right| \\
&\leq 2\|V''\|_\infty\|\snext\psi - {\psi}\| + 2\|V'\|_\infty\left(\|\partial_\theta (\snext\psi - {\psi})\| + \|\partial_\theta \snext\psi\|\|\snext\psi - {\psi}\|\right) \\
&\leq 6D_0\frac{\delta}{\rho} + 2D_0\left(200D_0 \frac{\delta}{\rho^2} + 12D_0 \frac{\delta}{\rho^2} \right) \\
&\leq 500D_0^2 \frac{\delta}{\rho^2}.\end{align*}
Taking $C = C(v) = 500D_0^2$ proves the proposition.
\end{proof}

\newpage
\section{Double resonance}\label{sec:DR}

In this section, we study the more complicated situation of double resonance, where an interval $\snext\Lambda \subset \mathbb{Z}$ contains precisely two resonating subintervals ${\Lambda}_\pm \subset \snext\Lambda$ such that $H^{{\Lambda}_\pm}(\theta_*)$ each has a unique eigenvalue ${\mathbf{E}}_\pm$ near a fixed $E_{**}$; in our setting, these eigenvalues will have relatively large, opposite signed derivatives: $\pm \partial_\theta {\mathbf{E}}_\pm \geq \nu > 0$.  Under smallness and stability assumptions on the boundary values of the corresponding eigenvector, alongside off-diagonal Green's function decay on the maximal connected components of $\snext\Lambda \setminus ({\Lambda}_- \cup {\Lambda}_+)$, we prove that $H^{\snext\Lambda}$ has precisely two well-separated resonant Rellich children (cf. Figure \ref{f:DRRelFn}).

Fix $v \in C^2(\mathbb{T},[-1,1])$ with $\|\partial_\theta v\|_\infty + \|\partial_\theta^2 v\|_\infty \leq D_0$, $(\theta_{**}, E_{**}) \in \mathbb{T} \times [-2,2]$, and $0 < \varepsilon < 1/7$.   Let $\snext\Lambda \subset \mathbb{Z}$ be an interval containing two subintervals ${\Lambda}_-, {\Lambda}_+$  with $\dist({\Lambda}_-,{\Lambda}_+) \geq \max_\pm\{|{\Lambda}_\pm|\}$ and denote the corresponding partition $\snext\Lambda =: {\Lambda}_l \cup {\Lambda}_- \cup {\Lambda}_c \cup {\Lambda}_+ \cup {\Lambda}_r$ by $\mathcal{P}$; despite the notation, we do not insist ${\Lambda}_-$ be left of ${\Lambda}_+$. Let ${\Lambda} = {\Lambda}_- \cup {\Lambda}_c \cup {\Lambda}_+$, ${L} = |{\Lambda}|$, %
and fix constants $\sprev{\rho}, \delta, \gamma, \ell, \nu > 0$ satisfying the following relations:
\begin{align}
\label{eq:eps2}
\delta < \sprev{\rho}^3/2 &< \frac{\sprev{\rho}^2\nu}{2400D_0} < 1/16, \\
\label{eq:gam2}
\log 7 \leq \gamma &\leq |\log\varepsilon|, \\
\label{eq:rho2}
8\delta \leq \frac{\sprev{\rho}^2\nu^3}{192D_0^3} &\leq \frac{\sprev{\rho}}{48D_0^2}.
\end{align}

\begin{assm}\label{as:DR}
Suppose the following hold for $|\theta - \theta_{**}| < {\sprev{\rho}}/8D_0$:
\begin{enumerate}
\item (Double resonance) Each operator $H^{{\Lambda}_\pm}(\theta_{**})$ has a unique eigenpair $({\mathbf{E}}_\pm,{\psi}_\pm)(\theta_{**})$ such that
${\mathbf{E}}_+(\theta_{**})={\mathbf{E}}_-(\theta_{**})=E_{**}$. 
\item (Eigenvector decay) The unit eigenvector ${\psi}_\pm(\theta)$ of $H^{{\Lambda}_\pm}(\theta)$ has $\mathcal{P}$-boundary values no larger than $\delta/\varepsilon$; i.e. one has
\begin{align}\label{eq:DRloc1}
\|\Gamma_\mathcal{P}^{\snext\Lambda} {\psi}_\pm\| \leq 2\delta, \quad j = 1,2.
\end{align}
Furthermore, for any resonant eigenpair $(\snext{E},\snext\psi)$ of $H^{\snext{\Lambda}}$ with $|\snext{E}-E_{**}| < \frac32\sprev{\rho}$, the unit eigenvector $\snext{\psi}$ also has $\mathcal{P}$-boundary values no larger than $\delta/\varepsilon$; i.e., one has
\begin{align}\label{eq:DRloc2}
    \|\Gamma_\mathcal{P}^{\snext\Lambda} \snext\psi\| \leq 4\delta.
\end{align}
\item (Green's function decay off ${\Lambda}_- \cup {\Lambda}_+$)
For $m, n \in {\Lambda}_{l/c/r}$ with $|m-n| \geq \ell$ one has
\begin{align}\label{eq:DRdec}
\log|R^{{\Lambda}_{l/c/r}}_{\theta,E}(m,n)| \leq -\gamma|m-n|
\end{align}
for $|E-E_{**}| < \frac32\sprev{\rho}$.
\item (Eigenvalue separation) ${\mathbf{E}}_\pm(\theta)$ are the only eigenvalues of $H^{\snext{\Lambda}}_\mathcal{P}(\theta)$ in $B_{7\sprev{\rho}/4}(E_{**})$:
\begin{align}\label{eq:DRsep2}
\chi_{\{{\mathbf{E}}_-(\theta)\}\cup\{{\mathbf{E}}_+(\theta)\}}(H^{\snext{\Lambda}}_\mathcal{P}(\theta)) = \chi_{B_{7\sprev{\rho}/4}(E_{**})}(H^{\snext{\Lambda}}_\mathcal{P}(\theta)).
\end{align}
In particular, we have
\begin{align}
\label{eq:DRsep1}
\|R_\perp(E;H^{\snext\Lambda}_\mathcal{P}(\theta),\{{\mathbf{E}}_-(\theta)\} \cup \{{\mathbf{E}}_+(\theta)\})\| &\leq {4\sprev{\rho}^{-1}}, \quad E \in B_{3\sprev{\rho}/2}(E_{**}).
\end{align}
\item (Eigenvector stability) The unit eigenvectors ${\psi}_\pm(\theta)$ of $H^{{\Lambda}_\pm}(\theta)$ have stably small $\mathcal{P}$-boundary values:
\begin{align}\label{eq:DRevecstab}
    \|\Gamma_\mathcal{P}^{\snext\Lambda}(\partial_\theta {\psi}_\pm)\| \leq 25D_0\delta\sprev{\rho}^{-1}.
\end{align}
\item (Transversality of Rellich functions) The eigenpairs ${\mathbf{E}}_\pm$ have large, opposite-signed derivatives: %
\begin{align}\label{eq:DRtransv}
\pm \partial_\theta {\mathbf{E}}_\pm(\theta) \geq \nu.
\end{align}
\end{enumerate}
\end{assm}
\noindent The remainder of this section proceeds under Assumption \ref{as:DR}.

\subsection{The resonant eigenpairs of $H^{\snext{\Lambda}}$}

Denote by
\begin{align*}
    \Lambda_{\text{loc}} &:= [\inf {\Lambda}_\pm - \ell, \sup {\Lambda}_\pm + \ell] \cap \snext\Lambda, \\
    {\snext\Lambda}_{l/r} &:= {\Lambda}_{l/r} \setminus \Lambda_{\text{loc}}.
\end{align*}
Under Assumption \ref{as:DR}, the operator $H^{\snext{\Lambda}}$ has precisely two resonant eigenpairs $(\snext{\mathbf{E}}_\bullet, \snext{\psi}_\bullet)$, $\bullet \in \{\vee,\wedge\}$, and $\snext{\psi}_\bullet$ are simultaneously localized on $\Lambda_{\text{loc}}$:

\begin{prop}[Double resonant eigenpairs]\label{pr:AL2}
Under Assumption \ref{as:DR} above, we have:
\begin{enumerate}
\item Any unit eigenvector $\snext\psi$ of $H^{\snext\Lambda}$ with corresponding eigenvalue $\snext{E}$ with $|\snext{E} - E_{**}| < \frac32\sprev{\rho}$ is Anderson localized on $\Lambda_{\textnormal{loc}}$:
\begin{align*}
\|\snext\psi\|_{\Lambda_{\textnormal{loc}}} &\geq 2/3, \\
\log|\snext\psi(j)| &\leq -\gamma \dist(j, {\Lambda}), \quad j \in {\snext\Lambda}_{l/r}.
\end{align*}
\item There exist two Rellich pairs $(\snext{\mathbf{E}}_\bullet, \snext{\psi}_\bullet)$, $\bullet \in \{\vee,\wedge\}$ for $H^{\snext{\Lambda}}$ with 
$\snext{\mathbf{E}}_\bullet : B_{\sprev{\rho}/8D_0}(\theta_{**}) \to B_{\sprev{\rho}/4}(E_{**})$.  
We normalize so that $\snext{\mathbf{E}}_\vee > \snext{\mathbf{E}}_\wedge$.  The eigenvectors satisfy 
\begin{align}
\label{eq:psi+}\snext{\psi}_\vee &= A{\psi}_+ + B {\psi}_- + \phi_{\vee}  \\
\label{eq:psi-}\snext{\psi}_\wedge &= B{\psi}_+ - A{\psi}_- + \phi_{\wedge}
\end{align}
where $A(\theta)$ and $B(\theta)$ satisfy $A^2 + B^2 = 1$ and $\|\phi_\bullet\| \leq 24\delta{\sprev{\rho}^{-1}}$.
 
\item Uniformly for $\theta$ with $|\theta - \theta_{**}| < \sprev{\rho}/8D_0$ and $|E - E_{**}| < \frac54\sprev{\rho}$, there are no other Rellich functions within $\sprev{\rho}/4$ of $E$:
\begin{align}\label{eq:rperpperpbd}
\|R_{\perp}(E; H^{\snext\Lambda}(\theta), \{\snext{\mathbf{E}}_\vee(\theta)\} \cup \{\snext{\mathbf{E}}_\wedge(\theta)\})\| \leq {4\sprev{\rho}^{-1}}.
\end{align}
\end{enumerate}
\end{prop}

\begin{proof}
The decay of any resonant eigenvector follows from the Green's function decay \eqref{eq:DRdec} and the Poisson formula \eqref{eq:poisson} identically to the single resonant case, so here we refer the reader to the proof of Proposition \ref{pr:AL1}.

Using this decay, we have that
\begin{align*}
    \|\snext\psi\|_{{\snext\Lambda}_{l/r}} &\leq \sum_{j \in \Lambda_\pm} e^{-\gamma \dist(j, {\Lambda})} \\
    &\leq \frac{e^{-\gamma}}{1-e^{-\gamma}} < \frac{1}{6}.
\end{align*}
Consequently, we have that
\begin{align*}
    \|\snext\psi\|_{\Lambda_{\text{loc}}} \geq 2/3
\end{align*}
for any resonant unit eigenvector $\snext\psi$.

We fix $\theta \in B_{\sprev{\rho}/8D_0}(\theta_{**})$ and will suppress its notation. 
Denoting by $P = \chi_{B_{3\sprev{\rho}/2}(E_{**})}(H^{\snext{\Lambda}})$ and by $\chi = \chi_{B_{3\sprev{\rho}/2}(E_{**})}(H^{\snext{\Lambda}}_{\mathcal{P}})$, suppose $(\snext{E},\snext\psi)$ is a resonant eigenpair of $H^{\snext{\Lambda}}$ with $|\snext{E} - E_{**}| < 3\sprev{\rho}/2$.  Then we have
\begin{align*}
    \|(I - \chi)P\| &= \|R_\perp(\snext{E}; H^{\snext{\Lambda}}_\mathcal{P}, B_{\sprev{\rho}}(E_{**}))(H^{\snext{\Lambda}}_\mathcal{P} - \snext{E})P\| \\
    &\leq 4{\sprev{\rho}^{-1}}\|\Gamma^{\snext\Lambda}_\mathcal{P} P\| \leq 16\delta{\sprev{\rho}^{-1}}.
\end{align*}
Lemma \ref{l:PQclose} and the assumption \eqref{eq:DRsep2} imply $H^{\snext{\Lambda}}$ has at most two resonant eigenpairs $(\snext{E},\snext\psi)$ satisfying $|\snext{E} - E_{**}| < 3\sprev{\rho}/2$.
If such an eigenpair exists, since $\|(H^{\snext{\Lambda}}_{\mathcal P}-\snext{E})P\| \leq 4\delta$, and $\{{\psi}_+,{\psi}_-\}$ forms an orthonormal basis for the image of $\chi_{B_{\sprev{\rho}}}(\snext{E})$, Lemma \ref{l:approxevect} implies $\snext\psi$ must be of the form \eqref{eq:psi+}.

To see the existence of two eigenpairs $(\snext{\mathbf{E}}_\bullet,\snext{\psi}_\bullet)$ with $E_\bullet : B_{\sprev{\rho}/8D_0}(\theta_{**}) \to B_{\sprev{\rho}/4}(E_{**})$, $\bullet \in \{\vee,\wedge\}$, first note that,
\begin{align*}
\|(H^{\snext\Lambda} - {\mathbf{E}}_\pm){\psi}_\pm\| &\leq \|\Gamma^{\Lambda}_{\mathcal{P}}{\psi}_\pm\| \leq 2\delta < \sprev{\rho}/8
\end{align*}
and, by \eqref{eq:feynman1}, 
\begin{align*}
    |{\mathbf{E}}_\pm - E_{**}| \leq D_0|\theta - \theta_{**}| %
    < \sprev{\rho}/8, 
\end{align*}
so at least one such eigenpair (e.g., $(\snext{\mathbf{E}}_\vee,\snext{\psi}_\vee)$) must exist by Lemma \ref{l:approxevect}.  Let $A$ and $B$ be defined as in \eqref{eq:psi+}, and denote by ${\psi}_\vee := A{\psi}_+ + B{\psi}_-$ and by ${\psi}_\wedge := B{\psi}_+ - A{\psi}_-$.

Denote by 
$P^c_\vee := I - \snext{\psi}_\vee\snext{\psi}_\vee^\top$ and by $\chi^c_\vee := I- {\psi}_\vee{\psi}_\vee^\top$.  By \eqref{eq:psi+}, we have $\|P_\vee^c - \chi_\vee^c\| \leq 48\delta{\sprev{\rho}^{-1}}$; thus, letting $\snext{H}_\wedge := P_\vee^c H^{\snext{\Lambda}} P_\vee^c$ and ${H}_\wedge := \chi_\vee^c H^{\snext{\Lambda}} \chi_\vee^c$, we have 
\begin{align*}
    \|(\snext{H}_\wedge - E_{**}){\psi}_\wedge\| &\leq \|(\snext{H}_\wedge - {H}_\wedge){\psi}_\wedge\| + \|({H}_\wedge - E_{**}){\psi}_\wedge\| \\
    &\leq 2\|H^{\snext{\Lambda}}\|\|P_\vee^c - \chi_\vee^c\| + \|\chi_\vee^c (H^{\snext{\Lambda}} - E_{**}){\psi}_\wedge\| \\
    &\leq 192\delta{\sprev{\rho}^{-1}} + |A|(|{\mathbf{E}}_- - E_{**}| + 2\delta) + |B|(|{\mathbf{E}}_+ - E_{**}| + 2\delta) \\
    &\leq 192\delta{\sprev{\rho}^{-1}} + \sqrt{2}(\sprev{\rho}/8 + 2\delta) \\
    &< \sprev{\rho}/4
\end{align*}
where in the penultimate line we have used that $|A| + |B| \leq \sqrt{2}$.  Thus, $\snext{H}_\wedge$ (and consequently $H^{\snext{\Lambda}}$) must have an eigenpair $(\snext{\mathbf{E}}_\wedge,\snext{\psi}_\wedge)$ where $\snext{\psi}_\wedge$ satisfies \eqref{eq:psi-}.
\end{proof}

\subsection{Local Rellich function structure}

Keeping in mind the eigenvalue separation Lemma \ref{l:evalsep}, we fix a separation constant
\begin{align} \label{eq:sepconst}
    \snext\sigma &\leq  \left(\varepsilon/3\right)^{3|\Lambda_{\textnormal{loc}}|}.
\end{align}

We also define the crossed parent curves
\begin{align*}
    {\mathbf{E}}_\vee := \max\{{\mathbf{E}}_+, {\mathbf{E}}_-\}, \quad 
    {\mathbf{E}}_\wedge := \min\{{\mathbf{E}}_+, {\mathbf{E}}_-\}.
\end{align*}

In this subsection, we will prove a precise formulation of the heuristic demonstrated in Figure \ref{f:DRRelFn}.  In order to do so, we will introduce two new parameters: an intermediate parameter $\eta \ll \sprev{\rho}$, representing the distance from $\theta_{**}$ where $\snext{\mathbf{E}}_\bullet$ begins to deviate from ${\mathbf{E}}_\bullet$, and a second resonance parameter $\rho \ll \eta$ with respect to which, away from $B_\eta(\theta_{**})$, the parent curves ${\mathbf{E}}_\bullet$ are simple-resonant.
\begin{thm}\label{t:DRRelFns}
Under the assumptions above, the two Rellich pairs $(\snext{\mathbf{E}}_\bullet,\snext{\psi}_\bullet)$, $\bullet \in \{\vee,\wedge\}$, for $H^{\snext{\Lambda}}$ from Proposition \ref{pr:AL2} satisfy the following:
\begin{enumerate}
    \item The Rellich functions $\snext{\mathbf{E}}_\bullet$ are Morse with precisely one critical point in $B_{\sprev{\rho}/8D_0}(\theta_{**})$, with Morse constants
    \begin{align*}
        d &= \nu/12, \\
        D &= 2D_0(1+D_0\snext\sigma^{-1})
    \end{align*}
    \item The Rellich functions $\snext{\mathbf{E}}_\bullet$ are uniformly separated on $B_{\sprev{\rho}/8D_0}(\theta_{**})$:
    \begin{align*}
        \inf_{B_{\sprev{\rho}/8D_0}(\theta_{**})} \snext{\mathbf{E}}_\vee(\theta) - \sup_{B_{\sprev{\rho}/8D_0}(\theta_{**})} \snext{\mathbf{E}}_\wedge(\theta) \geq \frac{\nu\snext\sigma}{2D_0 + \nu}
    \end{align*}
    \item Fix $\eta < \frac{\nu^2\sprev{\rho}}{100D_0^3}$ and $\rho < \frac89\nu\eta$.  There is a constant $C = C(v)$ depending only on $v$ so that, for $\theta \in B_{\sprev{\rho}/8D_0}(\theta_{**}) \setminus B_{\eta}(\theta_{**})$, 
    \begin{align}\label{eq:evalapproxDRedge}
        \left|\partial_\theta^k (\snext{\mathbf{E}}_\bullet - {\mathbf{E}}_\bullet)\right| \leq C\frac{\delta}{\rho^k}, \quad 0 \leq k \leq 2.
    \end{align}
\end{enumerate}
\end{thm}

\subsubsection{Lower bounds on the second derivative}

By Proposition \ref{pr:AL2}, there exist two Rellich pairs $(\snext{\mathbf{E}}_\bullet, \snext{\psi}_\bullet)$ for $H^{\snext{\Lambda}}$ with $\snext{\mathbf{E}}_\bullet : B_{\sprev{\rho}/8D_0}(\theta_{**}) \to B_{
\sprev{\rho}/4}(E_{**})$.  We will prove lower bounds on the second derivative near $\theta_{**}$:
\begin{prop}\label{pr:2ndDerivLwrBds}
Let $\eta < \frac{\nu^2\sprev{\rho}}{100D_0^3}$. 
There exists a constant $C = C(v)$ depending only on $v$ such that, for $\theta \in B_\eta(\theta_{**})$,
\begin{align*}
\left|\partial_\theta \snext{\mathbf{E}}_\vee(\theta)\right| \leq \nu/12 &\implies \partial_\theta^2 \snext{\mathbf{E}}_\vee (\theta) \geq C\frac{\nu^{5/2} \sprev{\rho}}{\delta}, \\
\left|\partial_\theta \snext{\mathbf{E}}_\wedge(\theta)\right| \leq \nu/12 &\implies -\partial_\theta^2 \snext{\mathbf{E}}_\wedge (\theta) \geq C\frac{\nu^{5/2} \sprev{\rho}}{\delta}.
\end{align*}
\end{prop}

We prove this Proposition via a series of lemmas.  Suppose %
that $|\partial_\theta{\mathbf{E}}_+(\theta_{**})| \geq |\partial_\theta {\mathbf{E}}_-(\theta_{**})|$ (the argument is similar for the opposite case), and define $1 \leq r \leq D_0/\nu$ such that
\begin{align*}
|\partial_\theta{\mathbf{E}}_+(\theta_{**})| = r|\partial_\theta {\mathbf{E}}_-(\theta_{**})|.
\end{align*}
\begin{lem}
Uniformly for $\theta \in B_{\eta}(\theta_{**})$, 
\begin{align*}
\left| \partial_\theta \left({\mathbf{E}}_+ + r{\mathbf{E}}_-\right)(\theta)\right| \leq 
8D_0^3\frac{\eta}{\sprev{\rho}\nu} %
\end{align*}
\end{lem}
\begin{proof}
By equations \eqref{eq:feynman2} and \eqref{eq:DRsep1}, we have the uniform bound
\begin{align*}
\left|\partial_\theta^2 {\mathbf{E}}_\pm\right| \leq D_0 + 2D_0^2\|R_\perp^{{\Lambda}_\pm}({\mathbf{E}}_\pm)\| \leq D_0 + 2D_0^2{\sprev{\rho}^{-1}}.
\end{align*}
The claimed bound follows by integrating; specifically, since $\partial_\theta {\mathbf{E}}_+$ and $\partial_\theta {\mathbf{E}}_-$ have opposite signs, we have 
\begin{align*}
\partial_\theta {\mathbf{E}}_+(\theta_{**}) + r\partial_\theta {\mathbf{E}}_-(\theta_{**}) = 0,
\end{align*}
and so
\begin{align*}
\left| \partial_\theta \left({\mathbf{E}}_+ + r{\mathbf{E}}_-\right)(\theta)\right| &= \left| \int_{\theta_{**}}^\theta \partial_\theta^2 ({\mathbf{E}}_+ + r{\mathbf{E}}_-) \dd t\right| \\
&\leq \int_{\theta_{**}}^{\theta} \left|\partial_\theta^2 {\mathbf{E}}_+\right| |\dd t| + r\int_{\theta_{**}}^{\theta} \left|\partial_\theta^2 {\mathbf{E}}_-\right| |\dd t| \\
&\leq (1+r)(D_0 + 2D_0^2{\sprev{\rho}^{-1}})\eta \\
&\leq 8D_0^3\frac{\eta}{\sprev{\rho}\nu} \qedhere
\end{align*}
\end{proof}

We now prove Proposition \ref{pr:2ndDerivLwrBds} for $\snext{\mathbf{E}}_\vee$, noting that the case $\snext{\mathbf{E}}_\wedge$ is entirely analogous.  To begin, we use Proposition \ref{pr:AL2} to relate $\partial_\theta \snext{\mathbf{E}}_\vee$ to the Rellich functions from the previous scale:
\begin{lem}
With the notation from Proposition \ref{pr:AL2} and $r$ as above, 
if $|\partial_\theta \snext{\mathbf{E}}_\vee| \leq \nu/12$, then
\begin{align*}
A^2 \geq \frac{5}{12r}, \quad
B^2 \geq \frac{1}{4}.
\end{align*}
\end{lem}

\begin{proof}
We simply apply the Feynman formula \eqref{eq:feynman1} and make the substitutions \eqref{eq:psi+} and \eqref{eq:psi-}, using the fact that ${\psi}_\pm$ are disjointly supported.  Specifically,
\begin{align*}
\partial_\theta \snext{\mathbf{E}}_\vee &= \langle \snext{\psi}_\vee, V'\snext{\psi}_\vee \rangle \\
&= A^2\langle {\psi}_+, V'{\psi}_+ \rangle + B^2 \langle {\psi}_-, V'{\psi}_- \rangle + 2\langle \phi_\vee, V'\snext{\psi}_\vee\rangle - \langle \phi_\vee,V'\phi_\vee \rangle \\
&= A^2 \partial_\theta {\mathbf{E}}_+ + B^2\partial_\theta {\mathbf{E}}_- + 2\langle \phi_\vee, V'\snext{\psi}_\vee\rangle - \langle \phi_\vee,V'\phi_\vee \rangle \\
&= (B^2-rA^2)\partial_\theta {\mathbf{E}}_- + \xi,
\end{align*}
where the error term $\xi$ has
\begin{align*}
|\xi| \leq 96D_0\frac{\delta}{\sprev{\rho}} + 8D_0^3\frac{\eta}{\sprev{\rho}\nu} < \frac{1}{12}\nu
\end{align*}
by Cauchy-Schwarz, Proposition \ref{pr:AL2}, equations \eqref{eq:eps2} and \eqref{eq:rho2}, and the previous lemma.

By assumption, we have $|\partial_\theta {\mathbf{E}}_-| \geq \nu$; thus, if $|\partial_\theta \snext{\mathbf{E}}_\vee| \leq \nu/12$, we have
\begin{align*}
\frac{\nu}{12} > |B^2 - rA^2|\nu - |\xi|,
\end{align*}
and so
\begin{align*}
|B^2 - rA^2| < \frac{1}{12} + \frac{|\xi|}{\nu} < \frac{1}{6}.
\end{align*}
It follows that
\begin{align*}
(1+r)A^2 = 1 - (B^2 -rA^2) \geq 5/6;
\end{align*}
the inequalities on $A^2$ and $B^2$ follow.
\end{proof}

We now proceed to prove Proposition \ref{pr:2ndDerivLwrBds}.

\begin{proof}[Proof of Proposition \ref{pr:2ndDerivLwrBds}]
Let $\theta \in B_\eta(\theta_{**})$.  By expanding the Feynman-type formula \eqref{eq:feynman2}, we have
\begin{align} \label{eq:2ndderiv}
\partial_\theta^2 \snext{\mathbf{E}}_\vee &= \langle \snext{\psi}_\vee, V''\snext{\psi}_\vee \rangle -\frac{2 \langle \snext{\psi}_\wedge, V'\snext{\psi}_\vee\rangle^2}{\snext{\mathbf{E}}_\wedge - \snext{\mathbf{E}}_\vee} - 2\langle R^\Lambda_\perp(\snext{\mathbf{E}}_\vee;H^{\snext{\Lambda}}(\theta), \{\snext{\mathbf{E}}_\vee\}\cup\{\snext{\mathbf{E}}_\wedge\})V'\snext{\psi}_\vee, V'\snext{\psi}_\vee \rangle.
\end{align}
The first term is bounded by $D_0$.  Furthermore, 
since $|\snext{\mathbf{E}}_\vee(\theta)-E_{**}|<\sprev{\rho}/4$, 
the third term is bounded by $8D_0^2{\sprev{\rho}^{-1}}$ by Proposition \ref{pr:AL2}. 
We will show that the second term is large.

By Proposition \ref{pr:AL2} and the previous lemma,
\begin{align*}
\langle \snext{\psi}_\wedge, V'\snext{\psi}_\vee\rangle^2 &= 
\left(AB\left(\langle {\psi}_+,V'{\psi}_+\rangle - \langle {\psi}_-,V'{\psi}_-\rangle\right) + \langle \phi_\wedge,V'\snext{\psi}_\vee\rangle+\langle\snext{\psi}_\wedge,V'\phi_\vee\rangle-\langle\phi_\wedge,V'\phi_\vee\rangle\right)^2\\
&\geq (AB)^2\left(\partial_\theta( {\mathbf{E}}_+ - {\mathbf{E}}_-)\right)^2 -288D_0^2\delta/\sprev{\rho} \\
&\geq \frac{5}{48r}(1+r)^2\nu^2 -288D_0^2\delta/\sprev{\rho} \\
&\geq \frac{5}{48}\nu^2 -288D_0^2\delta/\sprev{\rho} \geq \frac{1}{16}\nu^2,
\end{align*}
where the last inequality follows from \eqref{eq:eps2}.  

It remains to show that the denominator $\snext{\mathbf{E}}_\wedge - \snext{\mathbf{E}}_\vee$ is small.  By Proposition \ref{pr:AL2}, %
\begin{align*}
|A||{\mathbf{E}}_+ - \snext{\mathbf{E}}_\vee| &\leq  |{\mathbf{E}}_+ - \snext{\mathbf{E}}_\vee|\left(|\langle {\psi}_+, \snext{\psi}_\vee \rangle| + |\langle {\psi}_+, \phi_\vee \rangle| \right) \\
&\leq |\langle {\psi}_+, (H^{{\Lambda}_+}-\snext{\mathbf{E}}_\vee)\snext{\psi}_\vee \rangle| + |{\mathbf{E}}_+ - \snext{\mathbf{E}}_\vee||\langle {\psi}_+, \phi_\vee \rangle| \\
&\leq 40\frac{\delta}{\sprev{\rho}}.
\end{align*}
Similarly, $|B||{\mathbf{E}}_+ - \snext{\mathbf{E}}_\wedge| \leq 40\delta/\sprev{\rho}$.  Since
\begin{align*}
\min\{|A|,|B|\} \geq \frac{1}{2\sqrt{r}} \geq \sqrt{\frac{\nu}{D_0}},
\end{align*}
it follows that
\begin{align*}
|\snext{\mathbf{E}}_\vee - \snext{\mathbf{E}}_\wedge| \leq |\snext{\mathbf{E}}_\wedge -{\mathbf{E}}_+| + |{\mathbf{E}}_+ - \snext{\mathbf{E}}_\vee| \leq 80\frac{\delta}{\sprev{\rho}}\sqrt{\frac{D_0}{\nu}}.
\end{align*}

Combining the above estimates, we get
\begin{align*}
\frac{2 \langle \snext{\psi}_\wedge, V'\snext{\psi}_\vee\rangle^2}{|\snext{\mathbf{E}}_\wedge - \snext{\mathbf{E}}_\vee|} &\geq \frac{1}{640\sqrt{D_0}}\frac{\sprev{\rho}\nu^{5/2}}{\delta},
\end{align*} 
and so
\begin{align*}
|\partial_\theta^2 \snext{\mathbf{E}}_\vee| &\geq C\frac{\sprev{\rho}\nu^{5/2}}{\delta}
\end{align*}
with $C = C(v) = (800\sqrt{D_0})^{-1}$, e.g..  The bound for $\partial_\theta^2 \snext{\mathbf{E}}_\wedge$ is similar.  By \eqref{eq:rho2}, the middle term 
in \eqref{eq:2ndderiv} dominates; the sign of 
$\partial_\theta^2\snext{\mathbf{E}}_\vee$ then must match the sign of $\snext{\mathbf{E}}_\vee-\snext{\mathbf{E}}_\wedge$, 
which is positive for $\partial_\theta^2 \snext{\mathbf{E}}_\vee$.
\end{proof}

\subsubsection{Uniform local separation of Rellich functions}

By Proposition \ref{pr:AL2}, $\snext{\psi}_\vee$ and $\snext{\psi}_\wedge$ are simultaneously localized on the interval $\Lambda_{\text{loc}}$, so Lemma \ref{l:evalsep} guarantees Rellich function separation 
\begin{equation} \label{eq:ptwisesep}
\snext{\mathbf{E}}_\vee(\theta) - \snext{\mathbf{E}}_\wedge(\theta) \geq \snext\sigma 
\end{equation}
for each $\theta\in B_{\sprev{\rho}/8D_0}$.  In fact, the Rellich functions are uniformly separated:

\begin{prop}\label{pr:unifmLocalSep}
The two resonant Rellich functions $\snext{\mathbf{E}}_\bullet : B_{\sprev{\rho}/8D_0}(\theta_{**}) \to  B_{\sprev{\rho}/4}(E_{**})$, $\bullet \in \{\vee,\wedge\}$ of $H^{\snext{\Lambda}}$ are uniformly separated for all $\theta \in B_{\sprev{\rho}/8D_0}(\theta_{**})$; specifically,
\begin{align*}
\inf_{B_{\sprev{\rho}/8D_0}(\theta_{**})} \snext{\mathbf{E}}_\vee(\theta) - \sup_{B_{\sprev{\rho}/8D_0}(\theta_{**})} \snext{\mathbf{E}}_\wedge(\theta) \geq \frac{\nu\snext\sigma}{2D_0 + \nu}
\end{align*}
\end{prop}
The proof of this proposition involves a novel argument utilizing the Cauchy Interlacing Theorem:
\begin{thm}[Cauchy Interlacing Theorem]\label{t:cauchinter}
Let $A$ be an $n\times n$ Hermitian matrix, let $m \leq n$, and let $P$ be an $m \times n$ matrix such that $PP^* = I_{m \times m}$.  Let $B = PAP^*$ be a compression of $A$, and denote the (ordered) eigenvalues of $A$ (respectively $B$) by $\alpha_1 \leq \alpha_2 \leq \dots \leq \alpha_n$ (resp., $\beta_1 \leq \beta_2 \leq \dots \leq \beta_m$).  Then
\begin{align*}
\alpha_k \leq \beta_k \leq \alpha_{k+n-m}.
\end{align*}
\end{thm}
The interlacing theorem is proven via a standard Min-Max argument, cf. Appendix B.  We will use the interlacing theorem to compare $H^{\snext{\Lambda}}$ to a pair of auxiliary operators, which we now define.  Let ${P}_\pm := {\psi}_\pm{\psi}_\pm^\top$, let ${Q}_\pm := I - {P}_\pm$, and consider the auxiliary operators
\begin{align*}
{H}_\pm &= {Q}_\pm H^{\snext{\Lambda}} {Q}_\pm + {\mathbf{E}}_\pm {P}_\pm.
\end{align*}

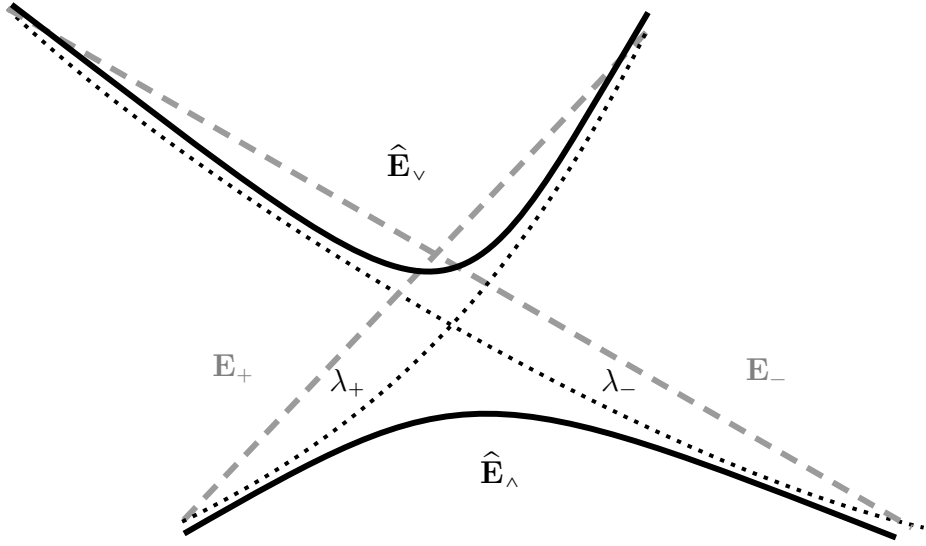
\begin{figure}
    \centering
    \begin{tikzpicture}[x=0.75pt,y=0.75pt,yscale=-1,xscale=1]

\draw [line width=1.5]  [dash pattern={on 1.69pt off 2.76pt}]  (100,30) .. controls (181,110) and (406,256) .. (562,292) ;
\draw [color={rgb, 255:red, 155; green, 155; blue, 155 }  ,draw opacity=1 ][line width=2.25]  [dash pattern={on 6.75pt off 4.5pt}]  (100,30) -- (556,292) ;
\draw [color={rgb, 255:red, 155; green, 155; blue, 155 }  ,draw opacity=1 ][line width=2.25]  [dash pattern={on 6.75pt off 4.5pt}]  (422,42) -- (188,289) ;
\draw [line width=1.5]  [dash pattern={on 1.69pt off 2.76pt}]  (188,289) .. controls (289,240) and (357,171) .. (422,42) ;
\draw [line width=2.25]    (102,28) .. controls (337,207) and (319,207) .. (423,32) ;
\draw [line width=2.25]    (189,295) .. controls (330,214) and (337,214) .. (548,297) ;

\draw (398,210.4) node [anchor=north west][inner sep=0.75pt]    {$\lambda _{-}$};
\draw (471,205.4) node [anchor=north west][inner sep=0.75pt]  [color={rgb, 255:red, 128; green, 128; blue, 128 }  ,opacity=1 ]  {${\mathbf{E}}_{-}$};
\draw (202,203.4) node [anchor=north west][inner sep=0.75pt]  [color={rgb, 255:red, 128; green, 128; blue, 128 }  ,opacity=1 ]  {${\mathbf{E}}_{+}$};
\draw (261,210.4) node [anchor=north west][inner sep=0.75pt]    {$\lambda _{+}$};
\draw (290,97.4) node [anchor=north west][inner sep=0.75pt]    {$\snext{\mathbf{E}}_{\lor }$};
\draw (337,252.4) node [anchor=north west][inner sep=0.75pt]    {$\snext{\mathbf{E}}_{\land }$};

\end{tikzpicture}
    \caption{A cartoon illustration of the interlacing argument in Lemma \ref{l:localsepLem}; the dashed lines represent the current scale Rellich functions ${\mathbf{E}}_\pm$, the full lines are the next scale eigenvalues $\snext{\mathbf{E}}_\bullet$, $\bullet \in \{\vee,\wedge\}$, and the dotted lines are the monotone interlacing curves $\lambda_\pm$ well-approximating ${\mathbf{E}}_\pm$.  The differences between $\lambda_\pm$ and ${\mathbf{E}}_\pm$ are exaggerated for emphasis.}
    \label{f:auxcurves}
\end{figure}

\begin{lem}\label{l:localsepLem}
The compressions ${Q}_\pm H^{\snext{\Lambda}} {Q}_\pm$ have unique eigenvalues $\lambda_\mp: B_{\sprev{\rho}/8D_0}(\theta_{**}) \to B_{\sprev{\rho}/4}(E_{**})$ near $E_{**}$.  These eigenvalues interlace $\snext{\mathbf{E}}_\bullet$, $\bullet \in \{\vee,\wedge\}$: 
\begin{align*}
\snext{\mathbf{E}}_\wedge(\theta) \leq \lambda_\mp(\theta) \leq \snext{\mathbf{E}}_\vee(\theta)
\end{align*}
and are uniformly $C^1$-close to previous-scale eigenvalues; that is, there exists a constant $C = C(v)$ depending only on $v$ so that
\begin{align*}
|\partial_\theta^k(\lambda_\pm - {\mathbf{E}}_\pm)| &\leq C\frac{\delta}{\sprev{\rho}^{2k}}, \quad k=0,1.
\end{align*}
In particular, $\lambda_\pm$ are monotone having different signed derivatives, and 
\begin{align*}
\pm\partial_\theta \lambda_\pm \geq \nu/2
\end{align*}
uniformly in $B_{\sprev{\rho}/8D_0}(\theta_{**})$.
\end{lem}

\begin{proof}
Since ${Q}_\pm$ and ${P}_\pm$ are orthogonal projections and $(H^{\snext{\Lambda}} - {\mathbf{E}}_\pm){P}_\pm = \Gamma_\mathcal{P}^{\snext\Lambda} {P}_\pm$, we have
\begin{align*}
H^{\snext{\Lambda}} - {H}_\pm &= {Q}_\pm H^{\snext{\Lambda}} {P}_\pm + {P}_\pm H^{\snext{\Lambda}} {Q}_\pm + {P}_\pm(H^{\snext{\Lambda}} - {\mathbf{E}}_\pm){P}_\pm \\
&= {Q}_\pm(H^{\snext{\Lambda}} - {\mathbf{E}}_\pm) {P}_\pm + {P}_\pm (H^{\snext{\Lambda}} - {\mathbf{E}}_\pm) {Q}_\pm + {P}_\pm(H^{\snext{\Lambda}} - {\mathbf{E}}_\pm){P}_\pm \\
&= {Q}_\pm \Gamma_\mathcal{P}^{\snext\Lambda} {P}_\pm + {P}_\pm \Gamma_\mathcal{P}^{\snext\Lambda} {Q}_\pm + {P}_\pm\Gamma_\mathcal{P}^{\snext\Lambda} {P}_\pm \\
&= \Gamma_\mathcal{P}^{\snext\Lambda} {P}_\pm + {P}_\pm \Gamma_\mathcal{P}^{\snext\Lambda} {Q}_\pm
\end{align*}
By \eqref{eq:DRloc1}, we have $\|\Gamma_\mathcal{P}^{\snext\Lambda}{P}_\pm\| \leq 2\delta$, and so
\begin{align*}
\|H^{\snext{\Lambda}} - {H}_\pm\| \leq 4\delta.
\end{align*}
It follows from the Min-Max principle that any eigenvalue of $H^{\snext{\Lambda}}$ must be within $4\delta$ of an eigenvalue of ${H}_\pm$, and conversely.  In particular, ${H}_\pm$ each have at least two eigenvalues in $B_{\frac14\sprev{\rho}+4\delta}(E_{**})$ and at most two eigenvalues in $B_{\frac32\sprev{\rho}-4\delta}(E_{**})$.  One of those eigenvalues must be ${\mathbf{E}}_\pm$ by definition; the other is an eigenvalue of ${Q}_\pm H^{\snext{\Lambda}} {Q}_\pm$.  

Denote by $\lambda_\mp$ the unique eigenvalue of ${Q}_\pm H^{\snext{\Lambda}} {Q}_\pm$ in $B_{\frac14\sprev{\rho}+4\delta}(E_{**})$.  By Cauchy Interlacing, the eigenvalues of ${Q}_\pm H^{\snext{\Lambda}} {Q}_\pm$ interlace those of $H^{\snext{\Lambda}}$.  Since ${Q}_\pm H^{\snext{\Lambda}} {Q}_\pm$ has precisely one eigenvalue in $B_{\frac14\sprev{\rho}+4\delta}(E_{**})$, it must lie between $\snext{\mathbf{E}}_\vee$ and $\snext{\mathbf{E}}_\wedge$; that is,
\begin{align*}
\snext{\mathbf{E}}_\wedge(\theta) \leq \lambda_\mp(\theta) \leq \snext{\mathbf{E}}_\vee(\theta).
\end{align*}

We will show now that $\lambda_\mp$ is $C^1$-close to ${\mathbf{E}}_\mp$.  By \eqref{eq:DRloc1}, we have
\begin{align*}
\|({Q}_\pm H^{\snext{\Lambda}} {Q}_\pm - {\mathbf{E}}_\mp){\psi}_\mp\| &= \|{Q}_\pm(H^{\snext{\Lambda}} - {\mathbf{E}}_\mp){\psi}_\mp\| \\
&= \|{Q}_\pm\Gamma_\mathcal{P}^{\snext\Lambda} {\psi}_\mp\| \leq 2\delta.
\end{align*} 
Let $\varphi_\mp$ denote the unit eigenvector of ${H}_\pm$ corresponding to $\lambda_\mp$.
Since $\lambda_\mp$ is the unique eigenvalue of ${Q}_\pm H^{\snext{\Lambda}} {Q}_\pm$ in $B_{\frac32\sprev{\rho}-4\delta}(E_{**}) \supset B_{\frac54\sprev{\rho}-4\delta}({\mathbf{E}}_\mp)$, it follows from Lemma \ref{l:approxevect} that we have
\begin{align*}
    |\lambda_\mp - {\mathbf{E}}_\mp| \leq 2\delta
\end{align*}
and (up to a choice of sign for $\varphi_\pm$) 
\begin{align*}
\|\varphi_\mp - {\psi}_\mp\| \leq \frac{2\sqrt{2}\delta}{\frac54\sprev{\rho} - 4\delta} \leq \frac{4\delta}{\sprev{\rho}}.
\end{align*}

By the Feynman formula \eqref{eq:feynman1}, we have
\begin{align*}
\partial_\theta \lambda_\mp &= \langle \varphi_\mp, \partial_\theta({Q}_\pm H^{\snext{\Lambda}}{Q}_\pm) \varphi_\mp \rangle \\
&= \langle \varphi_\mp, {Q}_\pm V' {Q}_\pm \varphi_\mp \rangle - 2\langle \varphi_\mp, {P}_\pm' H^{\snext{\Lambda}} {Q}_\pm\varphi_\mp \rangle.
\end{align*}
Since ${Q}_\pm{\psi}_\mp = {\psi}_\mp$, we have
\begin{align*}
|\langle \varphi_\mp, {Q}_\pm V'{Q}_\pm\varphi_\mp\rangle - \partial_\theta {\mathbf{E}}_\mp| &= |\langle \varphi_\mp, {Q}_\pm V'{Q}_\pm\varphi_\mp\rangle - \langle {\psi}_\mp, {Q}_\pm V'{Q}_\pm {\psi}_\mp \rangle| \\
&= |\langle \varphi_\mp - {\psi}_\mp, {Q}_\pm V'{Q}_\pm\varphi_\mp\rangle + \langle {\psi}_\mp, {Q}_\pm V'{Q}_\pm (\varphi_\mp -{\psi}_\mp) \rangle| \\
&\leq 2\|V'\|\|\varphi_\mp - {\psi}_\mp\| \\
&\leq 8D_0\frac{\delta}{\sprev{\rho}}.
\end{align*}
It remains to show that $|\langle \varphi_\mp, {P}_\pm'H^{\snext{\Lambda}} {Q}_\pm \varphi_\mp \rangle|$ is small.  By \eqref{eq:feynmanEvec} we have
\begin{align*}
-{P}_\pm' &= R_\perp^{{\Lambda}_\pm}({\mathbf{E}}_\pm)V' {P}_\pm + {P}_\pm V' R_\perp^{{\Lambda}_\pm}({\mathbf{E}}_\pm)
\end{align*}
On the one hand, we have
\begin{align*}
\|{P}_\pm H^{\snext{\Lambda}} {Q}_\pm \varphi_\mp \| &= \|{P}_\pm(H^{\snext{\Lambda}} - {\mathbf{E}}_\mp) {Q}_\pm \varphi_\mp \| \\
&= \|{P}_\pm(H^{\snext{\Lambda}} - {\mathbf{E}}_\mp) {Q}_\pm (\varphi_\mp - {\psi}_\mp) + {P}_\pm\Gamma_\mathcal{P}^{\snext\Lambda} {\psi}_\mp \| \\
&\leq 24\frac{\delta}{\sprev{\rho}},
\end{align*}
and so
\begin{align*}
|\langle \varphi_\mp, R_\perp^{{\Lambda}_\pm}V'{P}_\pm H^{\snext\Lambda}{Q}_\pm\varphi_\mp \rangle| &= |\langle V' R_\perp^{{\Lambda}_\pm}\varphi_\mp, {P}_\pm H^{\snext\Lambda}{Q}_\pm\varphi_\mp \rangle| \\
&\leq \|V'\|\|R_\perp^{{\Lambda}_\pm}\| \|{P}_\pm H^{\snext\Lambda}{Q}_\pm\varphi_\mp\| \\
&\leq 96D_0\frac{\delta}{\sprev{\rho}^2}.
\end{align*}
On the other hand, since ${P}_\pm {\psi}_\mp = 0$, we have
\begin{align*}
|\langle \varphi_\mp, {P}_\pm V' R_\perp^{{\Lambda}_\pm} H^{\snext{\Lambda}} {Q}_\pm \varphi_\mp \rangle | &= |\langle \varphi_\mp - {\psi}_\mp, {P}_\pm V' R_\perp^{{\Lambda}_\pm} H^{\snext{\Lambda}} {Q}_\pm \varphi_\mp \rangle | \\
&\leq \|\varphi_\mp - {\psi}_\mp\| \|V'\|\|R_\perp^{{\Lambda}_\pm}\| \|H^{\snext{\Lambda}} {Q}_\pm \varphi_\mp\| \\
&\leq 48D_0\frac{\delta}{\sprev{\rho}^2}
\end{align*}
Bringing it all together, it follows that
\begin{align*}
|\partial_\theta \lambda_\mp - \partial_\theta {\mathbf{E}}_\mp| &\leq 400D_0\frac{\delta}{\sprev{\rho}^2};
\end{align*}
the other claims follow.
\end{proof}

\begin{proof}[Proof of Proposition \ref{pr:unifmLocalSep}]
By the previous lemma, we have separated the Rellich functions $\snext{\mathbf{E}}_\vee$ and $\snext{\mathbf{E}}_\wedge$ by two transverse curves.  We use the quantitative transversality to now derive the size of the gap.

Let $\theta_\bullet$ denote the minimizing/maximizing values of $\theta$ for $\snext{\mathbf{E}}_\bullet$, $\bullet \in \{\vee,\wedge\}$, and let $h = |\theta_\vee - \theta_\wedge|$.  
By \eqref{eq:ptwisesep}, we have
\begin{align*}
\snext{\mathbf{E}}_\vee(\theta_\vee) - \snext{\mathbf{E}}_\wedge(\theta_\wedge) \geq \snext{\mathbf{E}}_\vee(\theta_\vee) - \snext{\mathbf{E}}_\wedge(\theta_\vee) - |\snext{\mathbf{E}}_\wedge(\theta_\vee) - \snext{\mathbf{E}}_\wedge(\theta_\wedge)| \geq \snext\sigma - D_0h,
\end{align*}
which is an effective bound for small $h$.
On the other hand, %
we have by transversality of the bounding curves $\lambda_\pm$
\begin{align*}
\snext{\mathbf{E}}_\vee(\theta_\vee) - \snext{\mathbf{E}}_\wedge(\theta_\wedge) \geq \frac{\nu}{2}|\theta_\vee - \theta_\wedge| = \frac{\nu h}2,
\end{align*}
which is an effective bound for large $h$.
Taking a convex combination of the two bounds yields
\begin{align*}
\snext{\mathbf{E}}_\vee(\theta_\vee) - \snext{\mathbf{E}}_\wedge(\theta_\wedge) &\geq \frac{\nu}{2D_0 + \nu}(\snext\sigma - D_0h) + \frac{2D_0}{2D_0 + \nu}\frac{\nu h}{2} \\
&\geq \frac{\nu\snext\sigma}{2D_0 + \nu},
\end{align*}
which was the claim.
\end{proof}

\subsubsection{Simple resonance of $\snext{\mathbf{E}}_\vee$ and $\snext{\mathbf{E}}_\wedge$ away from $\theta_{**}$}

In this section, we will show the following:
\begin{lem}\label{l:DRedgeSR}
Let $\theta_{*} \in B_{\sprev{\rho}/8D_0}(\theta_{**})\setminus B_{\eta}(\theta_{**})$, and suppose $E_{*} = {\mathbf{E}}_\pm(\theta_{*})$.  Then $H^{\snext{\Lambda}}$ satisfies Assumption \ref{as:SR} for $(\theta_*, E_*)$ with ${\Lambda} = {\Lambda}_\pm$, $\rho, \delta$ as above, and Green's function decay parameters $\salt{\ell} \geq 16|\log\varepsilon||\log\rho|$ and $\salt{\gamma} = \gamma - 6|\log\varepsilon||\log\rho|/\salt{\ell}$ as in Lemma \ref{l:NRdecay}.

In particular, 
for $\bullet \in \{\vee,\wedge\}$, 
there exists a constant $C = C(v)$ depending only on $v$ such that we have
\begin{align*}
    \left|\partial_\theta^k(\snext{\mathbf{E}}_\bullet - {\mathbf{E}}_\bullet)\right| \leq C\frac{\delta}{\rho^k}, \quad 0 \leq k \leq 2, 
\end{align*}
and 
\begin{align*}
    \|R_\perp(E;H^{\snext{\Lambda}}(\theta),\{\snext{\mathbf{E}}_\bullet(\theta)\})\| \leq 4\rho^{-1}
\end{align*}
uniformly for $\theta \in B_{\sprev{\rho}/8D_0}(\theta_{**})\setminus B_{\eta}(\theta_{**})$ 
and $E \in B_\rho(\snext{\mathbf{E}}_\bullet(\theta))$.
\end{lem}

\begin{proof}
Items 1, 2, and 5 are immediate by the definition of ${\Lambda}={\Lambda}_\pm$ and $E_* = {\mathbf{E}}_\pm(\theta_*)$.

Suppose that ${\Lambda} = {\Lambda}_- < {\Lambda}_+$ (the other cases are completely analogous), and consider now the interval $\overline{\Lambda}_r  := {\Lambda}_c \cup {\Lambda}_+ \cup {\Lambda}_r$.  To verify items 3 and 4, we need to verify that $\overline{\Lambda}_r$ has the Green's function decay property, and that $H^{\overline{\Lambda}_r}$ has no eigenvalues in $B_{7\rho/4}(E_*) = B_{7\rho/4}({\mathbf{E}}_-(\theta_*))$.

To see the eigenvalue separation, notice that, by Assumption \ref{as:DR}, $H^{\overline{\Lambda}_r}$ satisfies Assumption \ref{as:SR} %
with $\sprev{\rho}$ replacing $\rho$.
Thus, $H^{\overline{\Lambda}_r}$ has a unique eigenvalue $\overline{\mathbf{E}}_+(\theta): B_{\sprev{\rho}/8D_0}(\theta_*) \to B_{3\sprev{\rho}/2}(E_{**})$ such that
\begin{align*}
    |\overline{\mathbf{E}}_+(\theta) - {\mathbf{E}}_+(\theta)| \leq 2\delta.
\end{align*}
The necessary separation follows from the transversality of ${\mathbf{E}}_\pm$ and the definition of $\rho$; indeed, since $\eta > 9\rho/8\nu$, for $|\theta - \theta_*| < \rho/8D_0$ we have
\begin{align*}
    |E_* - \overline{\mathbf{E}}_+(\theta)| &\geq |{\mathbf{E}}_-(\theta_*) - {\mathbf{E}}_+(\theta_*)| - D_0|\theta_*-\theta| - |{\mathbf{E}}_+(\theta) - \overline{\mathbf{E}}_+(\theta)| \\
    &\geq 2\nu\eta - \rho/8 - 2\delta > 2\rho.
\end{align*}

It remains to verify Item 3 for $\salt{\gamma}$ and $\salt{\ell}$.  The Green's function $R^{\Lambda_l}$ has $(\ell,\gamma)$ decay (and thus $(\salt{\ell}, \salt{\gamma})$ decay) by assumption.  On the other hand, 
for $|E-E_*|<\frac32\rho$, the bound we just established implies $\|R^{\overline{\Lambda}_r}_{\theta,E}\| \leq 2\rho^{-1}$; thus, 
the interval $\overline{\Lambda}_r$ satisfies Assumption \ref{as:NR}, and so $R^{\overline{\Lambda}_r}$ also has $(\salt{\ell}, \salt{\gamma})$ decay by Lemma \ref{l:NRdecay}.
\end{proof}

\subsubsection{Proof of Theorem \ref{t:DRRelFns}}

\begin{proof}[Proof of Theorem \ref{t:DRRelFns}]
First, by \eqref{eq:feynman1} and \eqref{eq:feynman2}, we have
\begin{align*}
    |\partial_\theta \snext{\mathbf{E}}_\bullet| + |\partial_\theta^2 \snext{\mathbf{E}}_\bullet| \leq D_0 + D_0 + 2D_0^2\|R^\Lambda_\perp(\snext{\mathbf{E}}_\bullet)\|. 
\end{align*}
By Proposition \ref{pr:AL2} and Lemma \ref{l:evalsep}, for all $\theta \in B_{\sprev{\rho}/8D_0}(\theta_{**})$, we have
\begin{align*}
    \|R^\Lambda_\perp(\snext{\mathbf{E}}_\bullet)\| \leq \snext\sigma^{-1}.
\end{align*}
The Morse upper bound $D$ for $\snext{\mathbf{E}}_\bullet$ follows immediately.

The Morse lower bound $d$ for $\snext{\mathbf{E}}_\bullet$ for $\theta \in B_{\sprev{\rho}/8D_0}(\theta_{**}) \setminus B_\eta(\theta_{**})$ follows from Lemma \ref{l:DRedgeSR} and the transversality assumption on $\partial_\theta {\mathbf{E}}_\pm$; specifically,
\begin{align*}
    |\partial_\theta \snext{\mathbf{E}}_\bullet| \geq |\partial_\theta {\mathbf{E}}_\bullet| - |\partial_\theta (\snext{\mathbf{E}}_\bullet - {\mathbf{E}}_\bullet)| \geq \nu - C\delta\sprev{\rho}^{-1} \geq \nu/12
\end{align*}
For $\theta \in B_\eta(\theta_{**})$, the Morse lower bound follows from Proposition \ref{pr:2ndDerivLwrBds} and the fact that
\begin{align*}
    C\frac{\sprev{\rho}\nu^{5/2}}{\delta} \geq \nu/12.
\end{align*}

That $\snext{\mathbf{E}}_\bullet$ can each have at most one critical point in $B_{\sprev{\rho}/8D_0}(\theta_{**})$ follows again from Proposition \ref{pr:2ndDerivLwrBds}, Lemma \ref{l:DRedgeSR}, and continuity of the second derivative.  In particular, by Lemma \ref{l:DRedgeSR}, $\snext{\mathbf{E}}_\bullet$ can only have critical points in $B_\eta(\theta_{**})$, and by Proposition \ref{pr:2ndDerivLwrBds}, the concavity of $\snext{\mathbf{E}}_\bullet$ at any such critical point is uniquely determined.  %

Items 2 and 3 were shown in Proposition \ref{pr:unifmLocalSep} and Lemma \ref{l:DRedgeSR}, respectively.
\end{proof}

\subsubsection{Approximation of $\snext{\mathbf{E}}_\vee$ and $\snext{\mathbf{E}}_\wedge$ by previous scale}
To conclude this section, we note that the Rellich functions $\snext{\mathbf{E}}_\vee,\snext{\mathbf{E}}_\wedge$ are uniformly close (on order $\delta$) to the previous-scale functions ${\mathbf{E}}_\vee,{\mathbf{E}}_\wedge$; this implies upper bounds on the size of the vertical gap between the two Rellich functions and on the horizontal deviation of the Rellich functions' critical points from the center of resonsance $\theta_{**}$.
Since $\delta \gg \snext\sigma$, these estimates are not fine enough to contribute meaningfully to the lower bounds on the second derivative and the size of the gap. Moreover, these results are not necessary to prove our Main Theorem; indeed, larger gaps only help localization and Cantor spectrum, as we outlined in the introduction. Nevertheless, we include these results to provide a more complete picture of the local Rellich function structure.

\begin{prop} \label{pr:drclose}
For all $\theta \in B_{\sprev{\rho}/8D_0}(\theta_{**})$, we have 
\begin{equation} \label{eq:drclose}
|\snext{\mathbf{E}}_\bullet(\theta) - {\mathbf{E}}_\bullet(\theta)| \leq 4\delta, \quad \bullet \in \{\vee,\wedge\}.
\end{equation}
\end{prop}
\begin{proof}
Fix $\theta \in B_{\sprev{\rho}/8D_0}(\theta_{**})$. 
Letting ${P}_\pm = {\psi}_\pm{\psi}_\pm^\top$, we have $\|(H^{\snext{\Lambda}}-{\mathbf{E}}_\pm){P}_\pm\| = \|\Gamma^{\snext\Lambda}_{\mathcal P} {P}_\pm\| \leq 2\delta$ by \eqref{eq:DRloc1}. Thus, by Lemma \ref{l:approxevect}, 
$H^{\snext{\Lambda}}$ must have an eigenvalue in $\overline B_{2\delta}({\mathbf{E}}_\vee)$ and in $\overline B_{2\delta}({\mathbf{E}}_\wedge)$. If these two intervals are disjoint, \eqref{eq:drclose} most hold.
 
Otherwise, let ${P} = {P}_+ + {P}_-$, and consider that 
\begin{align*}
\left\|\left(H^{\snext{\Lambda}} - \frac{{\mathbf{E}}_\vee + {\mathbf{E}}_\wedge}2\right){P}\right\| &\leq \|\Gamma^{\snext\Lambda}_{\mathcal P}{P}\| + \frac{{\mathbf{E}}_\vee - {\mathbf{E}}_\wedge}2 \\ &\leq 4\delta + \frac{{\mathbf{E}}_\vee - {\mathbf{E}}_\wedge}2.
\end{align*}
Thus, by Lemmas \ref{l:PQclose} and \ref{l:approxevect}, $H^{\snext{\Lambda}}$ must have two eigenvalues in the interval $[{\mathbf{E}}_\wedge - 4\delta, {\mathbf{E}}_\vee + 4\delta]$. Since $H^{\snext{\Lambda}}$ must have an eigenvalue in $\overline B_{2\delta}({\mathbf{E}}_\vee)$, the greater of the two eigenvalues must be in $\overline B_{4\delta}({\mathbf{E}}_\vee)$; similarly, the lesser of the two eigenvalues must be in $\overline B_{4\delta}({\mathbf{E}}_\wedge)$. Thus \eqref{eq:drclose} holds.
\end{proof}

\begin{cor} \label{c:gapupperbound}
We have the following upper bound on the size of the gap between $\snext{\mathbf{E}}_\vee$ and $\snext{\mathbf{E}}_\wedge$:
    \begin{align*}
        \inf_{B_{\sprev{\rho}/8D_0}(\theta_{**})} \snext{\mathbf{E}}_\vee(\theta) - \sup_{B_{\sprev{\rho}/8D_0}(\theta_{**})} \snext{\mathbf{E}}_\wedge(\theta) \leq 8\delta.
    \end{align*}
\end{cor}
\begin{proof}
We observe that 
\begin{align*}
 \inf_{B_{\sprev{\rho}/8D_0}(\theta_{**})} \snext{\mathbf{E}}_\vee(\theta) - \sup_{B_{\sprev{\rho}/8D_0}(\theta_{**})} \snext{\mathbf{E}}_\wedge(\theta) &\leq
\snext{\mathbf{E}}_\vee(\theta_{**}) - \snext{\mathbf{E}}_\wedge(\theta_{**}) \\&\leq (\snext{\mathbf{E}}_\vee(\theta_{**})-E_{**}) + (E_{**}-\snext{\mathbf{E}}_\wedge(\theta_{**})) \\&\leq 4\delta + 4\delta = 8\delta,
\end{align*}
where the last inequality follows from Proposition \ref{pr:drclose}.
\end{proof}

\begin{cor} \label{c:horizontaldist}
Let $\theta_c$ be a critical point of $\snext{\mathbf{E}}_\vee$ or $\snext{\mathbf{E}}_\wedge$. Then $|\theta_c-\theta_{**}|< 8\delta/\nu$.
\end{cor}
\begin{proof}
Suppose $\theta_c$ is the critical point where $\snext{\mathbf{E}}_\vee$ attains its minimum. If $|\theta_c-\theta_{**}|\geq 8 \delta/\nu$, then we would have, by Proposition \ref{pr:drclose},  
\begin{align*}
\snext{\mathbf{E}}_\vee(\theta_c) &\geq {\mathbf{E}}_\vee(\theta_c) - 4\delta \\
&\geq {\mathbf{E}}_\vee(\theta_{**}) + \nu|\theta_c-\theta_{**}| - 4\delta \\
&\geq {\mathbf{E}}_\vee(\theta_{**}) + 4\delta \\
&\geq \snext{\mathbf{E}}_\vee(\theta_{**}),
\end{align*}
contradicting the fact that $\snext{\mathbf{E}}_\vee$ attains its minimum at $\theta_c$.

The proof procedes analogously if $\theta_c$ is the critical point where $\snext{\mathbf{E}}_\wedge$ attains its maximum.
\end{proof}

\newpage
\section{Resonance via inverse functions}\label{sec:invfn}

The present section describes carefully the double-resonant situation diagrammed in Figure \ref{f:doubleres} and contains the final pieces of machinery required to handle our multiscale inductive procedure.  The ultimate goal of this discussion is the construction of a covering of the codomain of our cosine-like function on whose components we can control recurrence; cf. Propositions \ref{pr:C2cover} and \ref{pr:SRDRsep} below.  

It is here that the ``cosine-like'' properties of our potential and its descendants are explicitly utilized.  Specifically, for a cosine-like function $f : I \to J$, the two-monotonicity interval structure allows us to describe double resonances as zeroes of a uniquely defined difference of inverse functions, and the Morse condition will yield upper bounds on the size of images of these inverse functions.  Combined with the Diophantine assumption on the frequency $\alpha$, we can then quantifiably separate the double resonances in terms of the Diophantine and Morse parameters.  The procedure requires somewhat careful assumptions; we demonstrate the robustness of these assumptions under $C^2$ perturbations at the end of the section.

Let $I_\pm \subset \mathbb{T}$ be two closed intervals with disjoint interiors, let $I := I_- \cup I_+$, 
and consider a $C^2$ function 
\begin{align*}
f : I \to J, \quad f_\pm := f|_{I_\pm}, \quad \pm \partial_\theta f_\pm \geq 0.
\end{align*}
\begin{assm}\label{as:C2fn}
Suppose the following hold:
\begin{enumerate}
\item The function $f$ is Morse on $I$: $$d \leq |\partial_\theta f| + |\partial_\theta^2 f| \leq D$$
\item Each function $f_\pm$ maps onto $J$: $$f_\pm(I_\pm) = J$$
\item There is a constant $0 < \nu < d/2$ such that $|\partial_\theta f|\geq\nu$ on the boundary points of $I$.
\end{enumerate}
\end{assm}
With these assumptions, we can define a function $T_f : J \to (I_+ - I_-)$ by
\begin{align*}
T_f(E) := f_+^{-1}(E) - f_-^{-1}(E).
\end{align*}
We also fix a constant $D_0$ such that
\begin{align*}
\sup_{\theta \in I}|\partial_\theta f(\theta)| \leq D_0 \leq D.
\end{align*}

\subsection{Preimages and crossing points}

We begin with a few lemmas about $f_\pm$ and $T_f$:

\begin{lem}\label{l:monotonederiv}
The derivative $\partial_\theta f_\pm$ is monotone on each connected component of
\begin{align*}
I_{\pm,< d/2} := \left\{\theta \in I_\pm : |\partial_\theta f_\pm(\theta)| < d/2 \right\}.
\end{align*}
In particular, any critical point of $f_\pm$ must lie at the boundary of $I_\pm$; and each connected component of  $I_{\pm,<\nu}$ must contain a critical point.
\end{lem}
\begin{proof}
That $\partial_\theta f_\pm$ is monotone on each such connected component follows immediately from the Morse condition, since $|\partial_\theta^2 f| \geq d/2$ and $\partial_\theta^2 f$ is continuous.  Since $\pm\partial_\theta f_\pm$ is strictly monotone and nonnegative, zeros must lie on the boundary of $I_\pm$. 
By monotonicity of $\partial_\theta f_\pm$, any connected component of $I_{\pm,<\nu}$ must contain a boundary point of $I_\pm$; since that point cannot be a boundary point of $I$, it must belong to $I_+ \cap I_-$, i.e., it must be a critical point.
\end{proof}

\begin{lem}\label{l:morsefnbd}
For any $\theta$ and $\theta_*$ in $I_\pm$,
\begin{align*}
|f_\pm(\theta) - f_\pm(\theta_*)| \geq \frac{d}{12}|\theta - \theta_*|^{2}.
\end{align*} 
In particular, for any subinterval $J_0 \subset J$, 
\begin{align*}
\frac{|J_0|}{D_0} \leq |f_\pm^{-1}(J_0)| \leq \sqrt{\frac{12}{d}|J_0|}
\end{align*}
\end{lem}

\begin{proof}
Suppose without loss of generality that $\theta_* \leq \theta$.  Then
\begin{align*}
|f_\pm(\theta) - f_\pm(\theta_*)| &= \int_{[\theta_*,\theta]} |\partial_\theta f_\pm(t)| |\dd t|.
\end{align*}
Consider the set
\begin{align*}
I_{\pm,<d/2} &:= \{ \theta \in I_\pm : |\partial_\theta f_\pm(\theta)| < d/2\}.
\end{align*}
As in the previous lemma, $\partial_\theta f_\pm$ is monotone on each connected component of this set; furthermore, there are at most two connected components of $[\theta_*,\theta] \cap I_{\pm,<d/2}$, and, by monotonicity of $\partial_\theta f$ on $I_{\pm,<d/2}$, these components must lie at the edges of $[\theta_*, \theta]$.  Thus, we may write
\begin{align*}
[\theta_*, \theta] = [\theta_*, \theta_1) \cup [\theta_1, \theta_2] \cup (\theta_2, \theta],
\end{align*}
where $[\theta_*,\theta] \cap I_{\pm,<d/2} = [\theta_*, \theta_1) \cup (\theta_2, \theta]$, and $|\partial_\theta f_\pm|$ is increasing on $[\theta_*,\theta_1)$ and decreasing on $(\theta_2, \theta]$.

Since $|\partial_\theta f_\pm|$ is increasing on $[\theta_*, \theta_1)$, we have
\begin{align*}
|\partial_\theta f_\pm(t)| \geq \left|\partial_\theta f_\pm(t) - \partial_\theta f_\pm(\theta_*)\right| &= \int_{[\theta_*, t)}|\partial_\theta^2 f_\pm(s)||\dd s| \\
&\geq \frac{d}{2}\left|t-\theta_*\right|, \quad t \in [\theta_*, \theta_1).
\end{align*}
Similarly, since $|\partial_\theta f_\pm|$ is decreasing on $(\theta_2, \theta]$, 
\begin{align*}
|\partial_\theta f_\pm(t)| \geq \frac{d}{2}\left|t - \theta\right|, \quad t \in (\theta_2, \theta].
\end{align*}
Combining these observations and using that $|\theta_1 - \theta_2| < 1$, we get
\begin{align*}
|f_\pm(\theta) - f_\pm(\theta_*)| &\geq \frac{d}{4}\left(|\theta_* - \theta_1|^2 + |\theta_2 - \theta|^2\right) + \frac{d}{2}|\theta_1 - \theta_2| \\
&\geq \frac{d}{4}\left(|\theta_* - \theta_1|^2 + |\theta_1 - \theta_2|^2 + |\theta_2 - \theta|^2\right) \\
&\geq \frac{d}{12}|\theta_* - \theta|^2,
\end{align*}
where the last step is a standard inequality following from Cauchy-Schwarz.

Since $f_\pm^{-1}(J_0)$ are intervals, the upper bound on $|f_\pm^{-1}(J_0)|$ is immediate.  The lower bound follows from the observation that
\begin{align*}
|J_0| &= \int_{f^{-1}_\pm(J_0)} |\partial_\theta f_\pm(t)| |\dd t| \leq D_0|f^{-1}_\pm(J_0)|. \qedhere
\end{align*}
\end{proof}

We now turn our attention to the difference of inverse functions $T_f$.  

\begin{lem} \label{l:tfbound}
We have uniform bounds on $|\partial_\theta f_\pm|$ when $\|T_f\|_\mathbb{T}$ is large; specifically, if $E_* = f_\pm(\theta_\pm) \in J$, then
\begin{align*}
\|T_f(E_*)\|_\mathbb{T} \geq \frac{7\nu}d \implies |\partial_\theta f_\pm(\theta_\pm)| \geq \nu
\end{align*}
\end{lem}
\begin{proof}
Suppose $|\partial_\theta f_-(\theta_-)|<\nu$ (the case $|\partial_\theta f_+(\theta_+)|<\nu$ is analogous). By Lemma \ref{l:monotonederiv}, there is a critical point $\theta_c$ of $f$ such that for all $\theta \in [\theta_-,\theta_c]$ (we assume $\theta_c>\theta_-$; the reverse case is exactly analogous), $|\partial_\theta f_-(\theta)|<\nu$. Since $|\partial_\theta^2f_-|\geq d/2$ on this interval, by the Mean Value Theorem, $$\nu > |\partial_\theta f_-(\theta)-0|\geq \frac{d}2|\theta_--\theta_c|,$$ and so $|\theta_--\theta_c|<2\nu/d$; similarly, 
$$|E_*-f(\theta_c)| < \nu|\theta_--\theta_c| < \frac{2\nu^2}d.$$ Let $J_0:=[f(\theta_c),E_*]$; by Lemma \ref{l:morsefnbd}, $$|f_+^{-1}(J_0)|\leq \sqrt{\frac{12}d|J_0|}<\frac{5\nu}d,$$
and so \begin{align*}\|T_f(E_*)\|_{\mathbb{T}} &\leq |f_-^{-1}(J_0)| + |f_+^{-1}(J_0)| < \frac{2\nu}d+\frac{5\nu}d = \frac{7\nu}d. \qedhere\end{align*}
\end{proof}

\begin{lem}\label{l:invdiffuniq}
For all $n$ and irrational $\alpha$, there exists at most one value $E_n = E(n,\alpha,f)$ such that
\begin{align*}
T_f(E_n) - n\alpha \mod 1 = 0.
\end{align*}
\end{lem}
\begin{proof}
First, note the function $T_f$ is strictly increasing on the interior of $J$, since, for any $E_* = f_\pm(\theta_\pm)$ such that $\partial_\theta f_\pm(\theta_\pm) \neq 0$, we have
\begin{align*}
\partial_E T_f(E_*) &= \frac{1}{\partial_\theta f_+(\theta_+)} - \frac{1}{\partial_\theta f_-(\theta_-)} \\
&= \frac{1}{|\partial_\theta f_+(\theta_+)|} + \frac{1}{|\partial_\theta f_+(\theta_-)|} \\
&\geq \frac{2}{D_0}
\end{align*}
by the inverse function theorem.  By Lemma \ref{l:monotonederiv}, $\partial_\theta f_\pm$ is nonzero on the interior of $I_\pm$, and monotonicity follows on the interior of the interval $J$.

Since $\alpha$ is irrational, $n\alpha \mod 1$ is distinct for all $n$.  Uniqueness of $E_n$ follows immediately from strict monotonicity of $T_f$.
\end{proof}
For $\alpha$ irrational and $l \in \mathbb{N}$, define
\begin{align*}
    \mathcal{N}(\alpha,f) := \{ n \in \mathbb{Z} : n\alpha \mod 1 \in T_f(J)\}.
\end{align*}
\begin{lem}\label{l:Naf}
There exists $\theta \in I_-$ such that $\theta+n\alpha \in I_+$ if and only if $n \in \mathcal{N}(\alpha,f)$.
\end{lem}
\begin{proof}
One direction is immediate: indeed, if $n \in \mathcal{N}(\alpha,f)$, then there exists $E = E_n \in J$ such that $T_f(E_n) - n\alpha \mod 1 = 0$; the relevant $\theta$ value is $f_-^{-1}(E_n)$.

Suppose that $\theta \in I_-$ and $\theta + n\alpha \in I_+$.  By monotonicity of $T_f$, $T_f(J) = [\inf I_+ - \sup I_-, \sup I_+ - \inf I_-]$.  Since
\begin{align*}
    \inf I_+ - \sup I_- < (\theta + n\alpha \mod 1) - \theta < \sup I_+ - \inf I_-,
\end{align*}
we get $n\alpha \mod 1 \in T_f(J)$.
\end{proof}

\subsection{Double resonance}

Let $\alpha \in DC_{C,\tau}$ and $L^{(1)} \in \mathbb{N}$ be fixed satisfying $$\frac{C}{2(L^{(1)})^\tau} \geq \frac{7\nu}d,$$ and denote by
\begin{align*}
  \mathcal{N}(L^{(1)},\alpha,f) := \{n \in \mathcal{N}(\alpha,f) : \; 0 < |n| \leq L^{(1)}\}
\end{align*}%
We can separate the distinct points $E_n$, $n \in \mathcal{N}$ developed in Lemma \ref{l:invdiffuniq} via the Diophantine condition:
\begin{lem}\label{l:DRensep}
For $n \neq m \in \mathcal{N}(L^{(1)},\alpha,f)$ such that $E_n$ and $E_m$ are in the interior of $J$, we have
\begin{align*}
|E_n - E_m| \geq \frac{C\nu}{2(L^{(1)})^\tau}
\end{align*}
\end{lem}
\begin{proof}
Letting $\|T_f(E_*)\|_\mathbb{T} \geq C/(L^{(1)})^\tau$ and $\theta_{\pm}$ denote the unique points in $I_\pm$ such that $f_\pm(\theta_{\pm}) = E_*$, we have by Lemma \ref{l:tfbound} that
\begin{align*}
|\partial_\theta f_\pm (\theta_{\pm})| \geq \nu.
\end{align*}
Provided $E_*$ is in the interior of $J$, it follows that
\begin{align*}
\partial_E T_f(E_*) &= \frac{1}{|\partial_\theta f_+(\theta_{+})|} + \frac{1}{|\partial_\theta f_-(\theta_{-})|} \leq \frac{2}{\nu}.
\end{align*}
By monotonicity of $T_f$ and the definition of $\|\cdot \|_\mathbb{T}$, the set
\begin{align*}
J_0(L^{(1)},\alpha,f) := \left\{E : \|T_f(E)\|_\mathbb{T} \geq \frac{C}{2(L^{(1)})^\tau} > 0\right\}
\end{align*}
is a connected subinterval of $J$.  By definition and the Diophantine condition, for $n \in \mathcal{N}(L^{(1)},\alpha,f)$ we have
\begin{align*}
\|T_f(E_n)\|_\mathbb{T} = \|n\alpha\|_\mathbb{T} \geq \frac{C}{|n|^\tau} \geq \frac{C}{(L^{(1)})^\tau}.
\end{align*} 
Thus, for any $n,m \in \mathcal{N}(L^{(1)},\alpha,f)$, the interval $(E_n, E_m)$ must lie inside of $J_{0}$ (and, by openness of $(E_n,E_m)$, the interior of $J$).  The result follows from the mean value theorem applied to $T_f$; specifically,
\begin{align*}
\frac{C}{(L^{(1)})^\tau} \leq \|(n-m)\alpha\|_\mathbb{T} &\leq |T_f(E_n) - T_f(E_m)| \leq \frac{2}{\nu}|E_n - E_m|. \qedhere
\end{align*}
\end{proof}
We can also separate these points from any critical values of $f$, if one exists:
\begin{lem} \label{l:DRensepfromcrit}
For $n \in \mathcal N(L^{(1)},\alpha,f)$ and $E_c$ a critical value of $f$, we have
\begin{equation*}
|E_n - E_c| \geq \frac{C^2d}{48(L^{(1)})^{2\tau}}.
\end{equation*}
\end{lem}
\begin{proof}
Let $\theta_\pm$ again denote the unique points in $I_\pm$ such that $f_\pm(\theta_\pm) = E_n$. Assume $E_c < E_n$ (the reverse case is analogous), and let $J_c = [E_c,E_n]$. By Lemma \ref{l:morsefnbd}, 
\begin{align*}
    |E_n - E_c| = |J_c| \geq \frac{d}{12}|f_\pm^{-1}(J_c)|^2.
\end{align*}
Since $f_-^{-1}(J_c) \cup f_+^{-1}(J_c) = [\theta_-,\theta_+]$, we have $|f_-^{-1}(J_c)|+|f_+^{-1}(J_c)| = \|n\alpha\|_{\mathbb T}$; thus $|f_\pm^{-1}(J_c)| \geq \|n\alpha\|_{\mathbb T}/2$ for some choice of sign. Then
\begin{align*}
    |E_n - E_c| &\geq \frac{d}{48}\|n\alpha\|_\mathbb{T}^2 \\
    &\geq \frac{C^2d}{24(L^{(1)})^{2\tau}}. \qedhere
\end{align*}
\end{proof}

We fix now a length scale $L^{(2)} \gg L^{(1)}$ and notions of resonance
\begin{align*}
\rho \ll \bar{\rho} \ll \sprev{\rho} < \min\left\{ \frac{d}{ D^2}\left(\frac{C}{24(L^{(1)})^{\tau}}\right)^2, \frac{C\nu}{12(L^{(1)})^\tau}, \frac{C\nu}{3(L^{(2)})^\tau}\right\}
\end{align*}
In order to handle technicalities that arise near the boundaries of the functions we consider, we fix notation for ``modified codomains" $\salt{\mathbf{J}}(f)$ of functions $f$ satisfying Assumption \ref{as:C2fn}.  Specifically, we denote
\begin{align*}
	\mu_l(f) &= \begin{cases} \rho & \text{ if $f$ attains its minimum at a critical point} \\
		-\frac98\sprev\rho & \text{ otherwise} \end{cases} \\
	\mu_r(f) &= \begin{cases} \rho & \text{ if $f$ attains its maximum at a critical point} \\
		-\frac98\sprev\rho & \text{ otherwise} \end{cases}
\end{align*}
and define 
\begin{equation} \label{eq:Jcheck}
	\salt{\mathbf{J}}(f):=[\inf J-\mu_l(f), \sup J+\mu_r(f)].
\end{equation}
Define $\salt{\salt{\mathbf J}}(f)$ similarly, with $\frac54$ replacing $\frac98$,
and denote $$\salt J := \{E \in \mathbb R : \overline{B}_\rho(E) \subset \salt{\mathbf J}(f)\} \subset \salt{\mathbf J}(f) \cap J$$ and $\salt{\salt J} = \salt{\salt{\mathbf J}}(f) \cap J$.

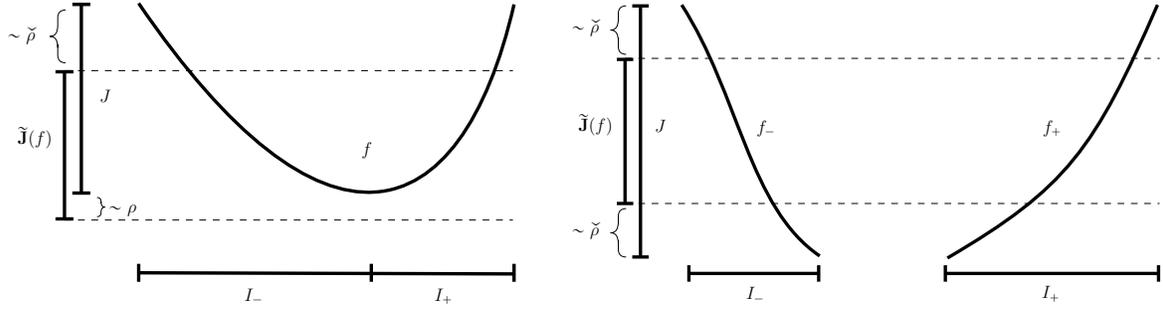
\begin{figure}
    \begin{subfigure}{.45\textwidth}
    \resizebox{\textwidth}{!}{%

\begin{tikzpicture}[x=0.75pt,y=0.75pt,yscale=-1,xscale=1]

\draw [line width=2.25]    (181,41) .. controls (344,280) and (467,250) .. (515,42) ;
\draw [line width=2.25]    (181,281) -- (515,282) ;
\draw [line width=2.25]    (181,273) -- (181,289) ;
\draw [line width=2.25]    (515,273) -- (515,289) ;
\draw [line width=2.25]    (388,273) -- (388,289) ;
\draw [line width=2.25]    (130,41) -- (130,209) ;
\draw [line width=2.25]    (115,101) -- (115,234) ;
\draw [line width=2.25]    (123,102) -- (107,102) ;
\draw [line width=2.25]    (123,233) -- (107,233) ;
\draw [line width=2.25]    (138,210) -- (122,210) ;
\draw [line width=2.25]    (138,42) -- (122,42) ;
\draw   (116,47) .. controls (111.33,47) and (109,49.33) .. (109,54) -- (109,60.5) .. controls (109,67.17) and (106.67,70.5) .. (102,70.5) .. controls (106.67,70.5) and (109,73.83) .. (109,80.5)(109,77.5) -- (109,87) .. controls (109,91.67) and (111.33,94) .. (116,94) ;
\draw   (144,231) .. controls (146.33,231) and (147.5,229.83) .. (147.5,227.5) -- (147.5,227.5) .. controls (147.5,224.17) and (148.67,222.5) .. (151,222.5) .. controls (148.67,222.5) and (147.5,220.83) .. (147.5,217.5)(147.5,219) -- (147.5,217.5) .. controls (147.5,215.17) and (146.33,214) .. (144,214) ;
\draw  [dash pattern={on 4.5pt off 4.5pt}]  (115,101) -- (514,101) ;
\draw  [dash pattern={on 4.5pt off 4.5pt}]  (115,234) -- (514,234) ;

\draw (274,293.4) node [anchor=north west][inner sep=0.75pt]    {$I_{-}$};
\draw (443,293.4) node [anchor=north west][inner sep=0.75pt]    {$I_{+}$};
\draw (378,163.4) node [anchor=north west][inner sep=0.75pt]    {$f$};
\draw (145,116.4) node [anchor=north west][inner sep=0.75pt]    {$J$};
\draw (72,149.4) node [anchor=north west][inner sep=0.75pt]    {$\salt{\mathbf{J}}(f)$};
\draw (63,61) node [anchor=north west][inner sep=0.75pt]    {$\sim \sprev{\rho }$};
\draw (153,219) node [anchor=north west][inner sep=0.75pt]    {$\sim \rho $};

\end{tikzpicture}

}%
\end{subfigure}
\hspace{.03\textwidth}
\begin{subfigure}{.52\textwidth}
\resizebox{\textwidth}{!}{%
\begin{tikzpicture}[x=0.75pt,y=0.75pt,yscale=-1,xscale=1]

\draw [line width=2.25]    (403.91,283.29) -- (596.08,283.29) ;
\draw [line width=2.25]    (403.91,275.87) -- (403.91,290.72) ;
\draw [line width=2.25]    (172.76,275.87) -- (172.76,290.72) ;
\draw [line width=2.25]    (596.08,275.87) -- (596.08,290.72) ;
\draw [line width=2.25]    (129.13,41) -- (129.13,269.37) ;
\draw [line width=2.25]    (115.2,89.27) -- (115.2,220.17) ;
\draw [line width=2.25]    (122.63,90.2) -- (107.77,90.2) ;
\draw [line width=2.25]    (122.63,220.17) -- (107.77,220.17) ;
\draw [line width=2.25]    (136.55,268.44) -- (121.7,268.44) ;
\draw [line width=2.25]    (136.55,41.93) -- (121.7,41.93) ;
\draw   (116.13,41.93) .. controls (111.46,41.93) and (109.13,44.26) .. (109.13,48.93) -- (109.13,53.74) .. controls (109.13,60.41) and (106.8,63.74) .. (102.13,63.74) .. controls (106.8,63.74) and (109.13,67.07) .. (109.13,73.74)(109.13,70.74) -- (109.13,78.56) .. controls (109.13,83.23) and (111.46,85.56) .. (116.13,85.56) ;
\draw  [dash pattern={on 4.5pt off 4.5pt}]  (115.2,220.17) -- (597,220.17) ;
\draw [line width=2.25]    (172.76,283.29) -- (290.66,283.29) ;
\draw [line width=2.25]    (289.73,275.87) -- (289.73,290.72) ;
\draw   (116.13,224.81) .. controls (111.46,224.81) and (109.13,227.14) .. (109.13,231.81) -- (109.13,236.62) .. controls (109.13,243.29) and (106.8,246.62) .. (102.13,246.62) .. controls (106.8,246.62) and (109.13,249.95) .. (109.13,256.62)(109.13,253.62) -- (109.13,261.44) .. controls (109.13,266.11) and (111.46,268.44) .. (116.13,268.44) ;
\draw  [dash pattern={on 4.5pt off 4.5pt}]  (115.2,89.27) -- (597,89.27) ;
\draw [line width=2.25]    (405.77,269.37) .. controls (513.45,206.24) and (538.52,170.97) .. (595.15,41) ;
\draw [line width=2.25]    (166.26,41) .. controls (218.25,121.76) and (226.6,223.88) .. (290.66,267.51) ;

\draw (223.24,294.19) node [anchor=north west][inner sep=0.75pt]    {$I_{-}$};
\draw (489.64,293.26) node [anchor=north west][inner sep=0.75pt]    {$I_{+}$};
\draw (140.76,142.94) node [anchor=north west][inner sep=0.75pt]    {$J$};
\draw (72,138.09) node [anchor=north west][inner sep=0.75pt]    {$\salt{\mathbf{J}}(f)$};
\draw (65.75,52.75) node [anchor=north west][inner sep=0.75pt]    {$\sim \sprev{\rho }$};
\draw (65.75,235.63) node [anchor=north west][inner sep=0.75pt]    {$\sim \sprev{\rho }$};
\draw (232.38,143.8) node [anchor=north west][inner sep=0.75pt]    {$f_{-}$};
\draw (490.42,143.8) node [anchor=north west][inner sep=0.75pt]    {$f_{+}$};

\end{tikzpicture}
}%
    \end{subfigure}
    \caption{An illustration of two different functions $f$ satisfying Assumption \ref{as:C2fn} alongside their modified codomains $\salt{\mathbf{J}}(f)$; note how the definition of $\salt{\mathbf{J}}$ depends on the presence of critical extrema.  The intervals $\salt{\salt{\mathbf{J}}}(f)$ are visually analogous.}
    \label{f:Jcheck}
\end{figure}

For $n \in \mathcal{N}(L^{(1)},\alpha,f)$, define the double-resonant interval
\begin{align*}
J_n^{\textnormal{DR}} := J_n^{\textnormal{DR}}(\bar{\rho},L^{(1)},\alpha,f) := \overline{B}_{\bar{\rho}}(E_n),
\end{align*}
and recall the special interval
\begin{align*}
J_0(L^{(1)},\alpha,f) = \left\{E : \|T_f(E)\|_\mathbb{T} \geq \frac{C}{2(L^{(1)})^\tau} > 0\right\}.
\end{align*}
\begin{lem}\label{l:DRnubd}
For all $n \in \mathcal{N}(L^{(1)},\alpha,f)$, $J_n^{\textnormal{DR}} \subset J_0(L^{(1)},\alpha,f)$.
\end{lem}
\begin{proof}
This is a corollary of monotonicity of $f_\pm$ and Lemma \ref{l:morsefnbd}; specifically, if $E_* \in J_n^{\textnormal{DR}}$, then $|E_* - E_n| \leq  3\sprev{\rho}$, and
\begin{align*}
\|T_f(E_*) - T_f(E_n)\|_\mathbb{T} \leq |f_+^{-1}((E_*,E_n))| + |f_-^{-1}((E_*,E_n))| \leq 12\sqrt{\frac{\sprev{\rho}}{d}} \leq \frac{C}{2(L^{(1)})^\tau}.
\end{align*}
In particular,
\begin{align*}
\|T_f(E_*)\|_\mathbb{T} \geq \|T_f(E_n)\|_\mathbb{T} - \|T_f(E_*) - T_f(E_n)\|_\mathbb{T} \geq \frac{C}{2(L^{(1)})^\tau},
\end{align*}
which was the claim.
\end{proof}

We define the complementary region
\begin{align*}
	J^{SR} := \salt{J} \setminus \bigcup_{n \in \mathcal{N}_{s+1}} B_{\bar\rho-3\rho}(E_n).
\end{align*}
We now divide it into simple-resonant intervals of comparable size to the double-resonant intervals:

\begin{prop} \label{pr:C2cover}
	$\salt{\salt{J}}$ can be covered by closed intervals $J_{i}^{(j)}$, $j \in \{1,2\}$, such that
	\begin{enumerate}
		\item $J_{n}^{(2)} = \overline{B}_{\bar{\rho}}(E_{n})$, where %
		$n \in \mathcal N(L^{(1)},\alpha,f)$. 
		\item $\bar{\rho} \leq |J_{i}^{(j)}| \leq 2\bar{\rho}$. 
		\item $J_{i} \cap J_{i'} \neq \emptyset \implies |J_{i}\cap J_{i'}| = 3\rho$. 
		\item For any $E \in J_{i}^{(j)}$, $\overline{B}_{\rho}(E) \subset \salt{\mathbf{J}}(f)$.  %
		\item The total number of such intervals $J_{i}^{(j)}$ does not exceed $5\lceil |J|\bar{\rho}^{-1}\rceil$.
	\end{enumerate}
	Denoting the corresponding preimages
	\begin{align*}
		I_{i,\pm}^{(j)} := 	f_\pm^{-1}(J_{i}^{(j)}),
		\quad I_{i}^{(j)} := I_{i,-}^{(j)} \cup I_{i,+}^{(j)},
	\end{align*}
	each function $f|_{I_{i}^{(j)}}: I_{i}^{(j)} \to J_{i}^{(j)}$ satisfies Assumption \ref{as:C2fn} with Morse constants $d,D$ and boundary derivative constant
	\begin{equation*}
		\salt\nu := \min\left\{\nu, \frac{3d}{8D_0}\bar{\rho}\right\}.
	\end{equation*}
\end{prop}

\begin{proof}
If $[E_m + \bar\rho - 3\rho, E_n - \bar\rho + 3\rho]$ is a connected component of $J^{SR}$, then, by Lemma \ref{l:DRensep}, it has length at least 
\begin{align*}
	|E_n - E_m| - 2\bar\rho + 6\rho &\geq C\nu/2((L^{(1)})^{\tau}) - 2\bar\rho + 6\rho \\&\geq \bar\rho. 
\end{align*}
Similarly, if a connected component of $J^{SR}$ contains a critical point of $f$, it has length at least 
\begin{align*} C^2d/48((L^{(1)})^{2\tau}) - \bar\rho + 3\rho  &\geq \bar\rho \end{align*} by Lemma \ref{l:DRensepfromcrit}. 
We can cover each connected component of at least this size (by the above remarks, this includes all connected components except possibly those at the boundaries of $J^{SR}$ where $f$ does not attain a critical point) by at most $4\lceil |E_n - E_m|/\bar\rho \rceil$ closed intervals of size between $\bar\rho$ and $2\bar\rho$ overlapping with only their nearest neighbors by exactly $3\rho$. 
We thus construct 
a collection of at most $4\lceil |J|/\bar\rho \rceil$ closed intervals $J_{i}^{(1)}$. 
Together with the collection of intervals $$\{J_n^{(2)} := J_n^{DR} \;:\; n \in \mathcal N(L^{(1)},\alpha,f) \text{ s.t. } J_n^{DR} \subset \salt{J}\},$$ 
we have a total of at most $5\lceil |J|/\bar\rho \rceil$ intervals $J_{i}^{(j)}$ ($j \in \{1,2\}$) satisfying the conditions laid out in the proposition. 
Moreover, these intervals cover $\salt{J}$, with the possible exception of intervals of length at most $2\bar\rho$ at the boundaries of $\salt{J}$ where $f$ does not attain a critical point; 
since $2\bar\rho < \frac18\sprev\rho-\rho$, they thus cover $\salt{\salt{J}}$.

It is clear that restrictions of $f$ to $I_i^{(j)}$ satisfy Assumption \ref{as:C2fn}; all that remains is to compute the new boundary derivative constant. By Lemma \ref{l:monotonederiv}, if $|\partial_\theta f(\theta_*)| < \nu$ for some $\theta_* \in I$, $\theta_*$ belongs to a connected component $I_{\pm,<\nu} \subset I$ containing a critical point $\theta_c$ and on which $|\partial_\theta f| < \nu$. If $\theta_*$ is a boundary point of some $I_{i}^{(j)}$, then, by construction, 
\begin{align*}
    |f(\theta_*) - f(\theta_c)| &\geq \bar\rho-3\rho \\ &\geq \frac34\bar\rho;
\end{align*} 
and by Lemma \ref{l:morsefnbd}, 
\begin{align*}
    |\theta_* - \theta_c| &\geq \frac3{4D_0}\bar\rho.
\end{align*}
Since $|\partial_\theta f| < \nu < d/2$ on $I_{\pm,\nu}$, $|\partial_\theta^2 f| > d_s/2$ on this component; thus,
\begin{align*}
    |\partial_\theta f(\theta_*)| &\geq \int_{\theta_c}^{\theta_*} |\partial_\theta^2 f(\theta)| d\theta \\ &\geq \frac{3d}{8D_0}\bar\rho. \qedhere
\end{align*}
\end{proof}

We denote the two preimages of each crossing point $E_n$ by
\begin{align*}
\theta_{n,\pm} &:= f_\pm^{-1}(E_n), \quad n \in \mathcal{N}(L^{(1)},\alpha,f).
\end{align*}

\begin{prop}\label{pr:SRDRsep}
We have the following:
\begin{enumerate}
\item (Simple resonance): If $\theta_* \in I_i^{(1)}$, then, for any $0 \neq |n| \leq L^{(1)}$ such that $\theta_* \in I-n\alpha$, we have 
\begin{align*}
|f(\theta_*) - f(\theta_*+n\alpha)| > 3\rho,
\end{align*}
and for all $\theta \in B_{\rho/8D_0}(\theta_*) \cap I \cap (I - n\alpha)$,
\begin{align*}
|f(\theta) - f(\theta+n\alpha)| > 2\rho.
\end{align*}
\item (Double resonance): If $\theta_* = \theta_{n_0,-} \in I_{n_0}^{(2)}$ for $n_0 \in \mathcal{N}(L^{(1)},\alpha,f)$, then for any $|n| \leq L^{(2)}$, $n \notin\{0,n_0\}$ such that $\theta_* \in I-n\alpha$, we have 
\begin{align*}
|f(\theta_*) -f(\theta_*+n\alpha)| &> 3\sprev{\rho};
\end{align*}
and for all $\theta \in B_{\sprev{\rho}/8D_0}(\theta_*) \cap I \cap (I - n\alpha)$, 
\begin{align*}
|f(\theta) -f(\theta+n\alpha)| &> 2\sprev{\rho}
\end{align*}
and
\begin{align*}
\min\{-\partial_\theta f(\theta),\, \partial_\theta f(\theta + n_0\alpha) \} \geq \nu.
\end{align*}
\end{enumerate}
\end{prop}

\begin{proof}
We begin with the simple-resonant case, i.e. $\theta_* \in I_{i}^{(1)}$.  Suppose $\theta_* \in I_-$ (the case $\theta_* \in I_+$ is completely analogous).
If $\theta_* + n\alpha \in I_-$, $0 \neq |n| \leq L^{(1)}$, we can apply Lemma \ref{l:morsefnbd} to get that
\begin{align*}
|f(\theta_*) - f(\theta_*+n\alpha)| \geq \frac{d}{12}\|n\alpha\|^2_\mathbb{T} \geq \frac{d}{12}\left(\frac{C}{(L^{(1)})^{\tau}}\right)^2 > 3\sprev{\rho} > 3\rho.
\end{align*}
We thus suppose that $\theta_* + n\alpha \in I_+$; by Lemma \ref{l:Naf}, $n \in \mathcal{N}(L^{(1)},\alpha,f)$. %

Assume $\theta_* \leq \theta_{n,-}$.  Then $\theta_* + n\alpha \leq \theta_{n,+} = \theta_{n,-} + n\alpha$, and, by the monotonicity of $f$ on $I_\pm$, $f(\theta_*+n\alpha) \leq E_n$.  Thus,
\begin{align*}
|f(\theta_*) - f(\theta_*+n\alpha)| &= f(\theta_*) - f(\theta_*+n\alpha) \\
&\geq f(\theta_*) - E_n > \bar\rho-3\rho > 3\rho.
\end{align*}
If $\theta_* \geq \theta_{n,-}$, then similarly $f(\theta_* + n\alpha) \geq E_n$, and
\begin{align*}
|f(\theta_*) - f(\theta_*+n\alpha)| &= f(\theta_*+n\alpha) - f(\theta_*) \\
&\geq E_n - f(\theta_*) > 3\rho.
\end{align*}
The claim for $\theta \in B_{\rho/8D_0}(\theta_*) \cap I \cap (I-n\alpha)$ follows from the Mean Value Theorem and the uniform bound $|\partial_\theta f| \leq D_0$.  This concludes the simple resonant case.

We now consider the case $\theta_* = \theta_{n_0,-}$.  
Suppose $n$ with $0 \neq |n| \leq L^{(2)}$ satisfies $\theta_* + n\alpha \in I_-$ and $|f(\theta_* + n\alpha) - f(\theta_*)| \leq 3\sprev{\rho}$.  Then by Lemma \ref{l:DRnubd},
\begin{align*}
3\sprev{\rho} \geq |f(\theta_*) - f(\theta_*+n\alpha)| \geq \nu\|n\alpha\|_\mathbb{T}  \geq \frac{C\nu}{(L^{(2)})^\tau},
\end{align*}
a contradiction.  Thus, for any $0 \neq |n| \leq L^{(2)}$ such that $\theta_* + n\alpha \in I_-$,  $|f(\theta_* + n\alpha) - f(\theta_*)| > 3\sprev{\rho}$.

If instead $\theta_* + n\alpha \in I_+$, $|n| \leq L^{(2)}$ and $n \neq n_0$, and suppose $|f(\theta_* + n\alpha) - f(\theta_*)| \leq 3\sprev{\rho}$.  Then $\theta_* + n_0\alpha \in I_+$ and
\begin{align*}
3\sprev{\rho} \geq |f(\theta_*) - f(\theta_*+n\alpha)| &= |f(\theta_* + n_0\alpha) - f(\theta_* + n\alpha)| \geq \nu\|n\alpha\|_\mathbb{T} \geq \frac{C\nu}{(L^{(2)})^\tau},
\end{align*}
a contradiction as above.  Again, the claim for $\theta \in B_{\sprev{\rho}/8D_0}(\theta_*) \cap I \cap (I-n\alpha)$ follows from the Mean Value Theorem and the uniform bound $|\partial_\theta f| \leq D_0$.  The lower bound for the derivatives follows from Lemmas \ref{l:tfbound} and \ref{l:DRnubd}, and that $B_{\sprev{\rho}/8D_0}(\theta_*) \cap I \subset I^{\textnormal{DR}}_{n_0,-}$.
\end{proof}

\subsection{Domain adjustment}

We now show that a function $g$ well-approximating a function $f$ satisfying Assumption \ref{as:C2fn} can have its domain slightly modified to satisfy Assumption \ref{as:C2fn}.

Let $I_\pm \subset \mathbb{T}$ be two closed intervals with disjoint interiors, let $I := I^f_- \cup I^f_+$, and consider a piecewise $C^1$ function
\begin{align*}
f : I \to J, \quad f_\pm := f|_{I^f_\pm}
\end{align*}
such that each function $f_\pm$ maps onto $J$.

Let $\delta>0$ satisfy $16\delta < |J|$ and $\frac{16\delta}{\nu} < |I_\pm|$.  Our first lemma handles the interval adjustment in the absence of critical points for $f$:

\begin{lem}\label{l:Asm3StabDR}
Suppose $\pm\partial_\theta f_\pm \geq \nu > 0$ uniformly on $I_\pm$, and suppose $g \in C^2(I, \mathbb{R})$ satisfies the following stability conditions:
\begin{enumerate}
\item The function $g$ is Morse on $I$
\begin{align*}
d \leq |\partial_\theta g| + |\partial_\theta^2 g| \leq D
\end{align*}
with $I_\pm^g := \{\theta \in I : \pm\partial_\theta g \geq 0\}$ two closed intervals with disjoint interiors.
\item There is a constant $0 < \tilde{\nu} < d/2$ such that $|\partial_\theta g| \geq \tilde{\nu}$ on the boundary points of $I$.
\item %
If $\theta \in I$ is such that $g(\theta) \notin g(I^g_-) \cap g(I^g_+)$, or if $\theta$ is a boundary point of $I$, then 
\begin{align*}
|f(\theta) - g(\theta)| \leq 2\delta.
\end{align*}
\end{enumerate}
Then there exists a subset $\tilde I^g = \tilde I_-^g \cup \tilde I_+^g \subset I$, $\tilde I_\pm^g$ being closed intervals with disjoint interiors, such that
\begin{align*}
|I \setminus \tilde I^g| \leq \frac{32\delta}{\nu}
\end{align*} 
on which $g$ satisfies Assumption \ref{as:C2fn}.
\end{lem}
\begin{proof}
Let $J^g_r := \min\{\sup g(I_g^+),\sup g(I_g^-)\}$ 
and $J^g_l := \max\{\inf g(I_g^+),\inf g(I_g^-)\}$. 

We first show that $J^g_r > J^g_l$. 
Note that $\max\{g(\sup I^g_+),g(\inf I^g_-)\}=\sup g$; likewise, $\min\{g(\inf I^g_+),g(\sup I^g_-)\}=\inf g$. 
Define $g_\pm := g|_{I^g_\pm}$; note that we have either $g_+^{-1}(J^g_r)=\sup I^g_+$ or $g_-^{-1}(J_g^r)=\inf I^g_-$.
Either $\inf I^g_- = \sup I^g_+$ is a critical point of $g$, in which case $J^g_r = \sup g$, or both $\inf I^g_-$ and $\sup I^g_+$ are (distinct) boundary points of $I$, in which case $f(\inf I^g_-)=f(\sup I^g_+) \in \{\inf J, \sup J\}$, and thus $|f(\inf I^g_-) - J^g_r|, |f(\inf I^g_-) - \sup g|\leq 2\delta$. Thus in either case we have $|J^g_r - \sup g|\leq 4\delta$, and it follows that 
\begin{equation}\label{eq:Jgrsup}
|J^g_r - \sup J|\leq 6\delta.
\end{equation}
By analogous reasoning, 
\begin{equation}\label{eq:Jglinf}
|J^g_l - \inf J|\leq 6\delta. 
\end{equation}
Since, by assumption, $|J|>12\delta$, we must have $J^g_r>J^g_l$.

We define $J^g := [J^g_l,J^g_r]$ and $\tilde I^g_\pm := g_\pm^{-1}(J^g)$ (so $\tilde I^g = g^{-1}(J^g)$). Since we have shown $J^g_r>J^g_l$, $\tilde I^g_\pm$ are nonempty intervals, and $g_\pm$ maps $\tilde I^g_\pm$ onto $J^g$. 
We now measure the set
$I \setminus \tilde I^g$
by considering the sets $I_\pm \setminus \tilde I^g_\pm$. 
We first note that by applying the Mean Value Theorem to $f$ on $\tilde I^g_\pm$, we conclude that $\tilde I^g_\pm \cap I_\pm \neq \emptyset$. Thus each of the sets $I_\pm \setminus \tilde I^g_\pm$ consists of at most two intervals $[\inf I_\pm,\inf \tilde I^g_\pm)$ and $(\sup \tilde I^g_\pm, \sup I_\pm]$. 
Consider first $[\inf I_-,\inf \tilde I^g_-)$.%
Such an interval belongs to $I_- \setminus \tilde I^g_-$ if and only if $\inf \tilde I^g_- \in I_-$. 
If $J^g_r$ is not a critical point of $g$, we have $|f(\inf \tilde I^g_-) - J^g_r| = |(f-g)\circ g_-^{-1}(J^g_r)| \leq 2\delta$; combining this with \eqref{eq:Jgrsup} gives $|f(\inf \tilde I^g_-) - \sup J|\leq 8\delta$. Thus 
\begin{align*}
    |[\inf I_-,\inf \tilde I^g_-)| &= |\inf \tilde I^g_- - \inf I_-| \\ 
    &= |f_-^{-1}(f(\inf \tilde I^g_-))-f_-^{-1}(\sup J)| \\ 
    &\leq \frac{1}{\nu}|f(\inf \tilde I^g_-) - \sup J| \leq \frac{8\delta}{\nu}.
\end{align*}
Analogously, $|(\sup \tilde I^g_+, \sup I_+]| \leq \frac{8\delta}{\nu}$ if $J^g_r$ is not a critical point of $g$, and $|(\sup \tilde I^g_-, \sup I_-]| \leq \frac{8\delta}{\nu}$, $|[\inf I_+,\inf \tilde I^g_+)|\leq \frac{8\delta}{\nu}$ if $J^g_l$ is not a critical point of $g$.

If $g$ contains no critical points, we conclude that $|I_\pm \setminus \tilde I^g_\pm|\leq \frac{16\delta}{\nu}$, and $|I \setminus \tilde I^g| \leq |I_- \setminus \tilde I^g_-|+|I_+\setminus \tilde I^g_+| \leq \frac{32\delta}{\nu}$.

If $g$ attains its minimum, but not its maximum, at a critical point, then $\tilde I^g = [\inf \tilde I^g_-, \sup \tilde I^g_+]$, and likewise $I = [\inf I_-, \sup I_+]$; thus, 
\begin{align*}
|I \setminus \tilde I^g| &\leq |[\inf I_-, \inf \tilde I^g_-)| + |(\sup \tilde I^g_+, \sup I_+]| \\
&\leq \frac{16\delta}{\nu}
\end{align*}
by the above computation. The argument is analogous if $g$ attains its maximum, but not its minimum, at a critical point. If $g$ attains both its maximum and its minimum at critical points, then trivially $\tilde I^g = I$.

It remains to verify a lower bound for $|\partial_\theta g|$ on the boundary points of $\tilde I^g$. 
We first note that we can apply Lemma \ref{l:monotonederiv} to the unrestricted function $g$ on $I$, since the condition that $g_\pm$ map onto the same image is not used in Lemma \ref{l:monotonederiv}. 
Thus, if $|\partial_\theta g|<\tilde\nu$ on a boundary point of $\tilde I^g$ (which is also a boundary point of $\tilde I^g_\pm)$, it must belong to a connected component of $I^g_\pm$ containing a critical point of $g$ (which is a boundary point of $I^g_\pm$), and $|\partial_\theta^2 g|\geq d/2$ throughout that component. By construction, the critical point belongs to $\tilde I^g_\pm$, and is thus the other boundary point of that interval; thus $|\partial_\theta^2 g|\geq d/2$ on the entirety of $\tilde I^g_\pm$. Then, on the boundary point of $\tilde I^g$, 
\begin{equation*}
    |\partial_\theta g| \geq \frac{d}{2}|\tilde I^g_\pm| \geq \frac{d}{2}\left(|I_\pm| - \frac{16\delta}{\nu}\right),
\end{equation*}
so we conclude $|\partial_\theta g| \geq \min\left\{\tilde\nu,\frac{d}{2}\left(|I_\pm| - \frac{16\delta}{\nu}\right)\right\}$ on the boundary of $\tilde I^g$.
\end{proof}

If $f : I \to J$ satisfies Assumption \ref{as:C2fn} and $g : I \to \mathbb{R}$ is $C^2$ close to $f$, we likewise can find these adjusted intervals:
\begin{lem} \label{l:Asm3StabSR}
Suppose $g \in C^2(I,\mathbb{R})$ is $C^2$ close to $f$, i.e.
\begin{align*}
\|\partial_\theta^k (f-g)\| \leq C\frac{\delta}{\rho^k}, \quad 0 \leq k \leq 2.
\end{align*}
Then there exists a subset $\tilde{I}^g = \tilde{I}^g_- \cup \tilde{I}^g_+ \subset I$, $\tilde{I}_\pm^g$ being closed intervals with disjoint interiors, such that 
\begin{align*}
|I \setminus \tilde{I}^g| \leq \frac{32\delta}{\nu}
\end{align*}
and on which $g$ satisfies Assumption 3.
\end{lem}

\begin{proof}
By the well-approximation of $g$ to $f$, we immediately get that $g$ is Morse on $I$, e.g.
\begin{align*}
\tilde{d} = d - C\left(\frac{\delta}{\rho} + \frac{\delta}{\rho^2}\right) < |\partial_\theta g| + |\partial_\theta^2 g| \leq D + C\left(\frac{\delta}{\rho} + \frac{\delta}{\rho^2}\right) = \tilde{D}.
\end{align*}
Furthermore, on the boundary points of $I$, one has
\begin{align*}
|\partial_\theta g| \geq \nu - C\frac{\delta}{\rho} > \nu - 2C\left(\frac{\delta}{\rho} + \frac{\delta}{\rho^2}\right) =: \tilde{\nu}
\end{align*}
and $\tilde{\nu} < \tilde{d}/2$.  By the Morse condition on $g$ and the assumption on $f$, we can decompose $I = I^g_- \cup I^g_+$ into intervals such that
\begin{align*}
\pm \partial_\theta g|_{I^g_\pm} \geq 0.
\end{align*}
We denote by $g_\pm := g|_{I^g_\pm}$.  By Lemma \ref{l:monotonederiv}, critical points of $g_\pm$ must lie on the boundary of $I^g_\pm$ and each connected component of $I^g_{\pm,<\tilde{\nu}}$ must contain a critical point, and, conversely, any critical point of $g$ must lie in a connected component of $I^g_{\pm,<\tilde{\nu}}$.

If $g$ has no critical point in $I$, then neither can $f$, and we are in the situation outlined in Lemma \ref{l:Asm3StabDR} (with $\nu$ replaced by $\nu/4$, e.g.).  If $g$ has a critical point $\theta_c$, then by the Morse condition it is an extremum of $g$.  Suppose the critical point is a minimum; then, defining $J^g$ as in the proof of Lemma \ref{l:Asm3StabDR}, we have
\begin{align*}
g_+^{-1}(J^g) &=: \tilde{I}_+^g = [\theta_c, \tilde{\theta}_+], \\
g_-^{-1}(J^g) &=: \tilde{I}_-^g = [\tilde{\theta}_-, \theta_c].
\end{align*}
By monotonicity of $g_\pm$, it must be that $\tilde{I}_\pm^g = I_\pm^g$ for at least one of $+$ or $-$; suppose $\tilde{I}_+^g = I_+^g$.  Since $\sup f_+ = \sup f_-$ and $\|f -g\| \leq C\delta$, we get that $|\sup g_+ - \sup g_-| = |g(\tilde{\theta}_+) - J_r^g| \leq C\delta$.

To measure $|I_- \setminus \tilde{I}_-^g|$, it suffices to compute $|\tilde{\theta}_- - \theta_-|$, where $\theta_\pm = \inf I_\pm^g$.  We have
\begin{align*}
|g(\theta_-) - g(\tilde{\theta}_-)| = |g(\theta_-) - f(\theta_+)| \leq C\delta + |f(\theta_-) - f(\theta_+)| = C\delta;
\end{align*}
furthermore, $\theta_-$ and $\tilde{\theta}_-$ live on the same monotonicity interval of $g$, and so
\begin{align*}
|\theta_- - \tilde{\theta}_-| \leq \sqrt{\frac{12C\delta}{\tilde{d}}}.
\end{align*}
Since our only critical point is a minimum, it must be that $\theta_- = \inf I_-^f$, and thus for any $\theta \in [\theta_-,\tilde{\theta}_-]$, we have
\begin{align*}
|g'(\theta)| &\geq |f'(\theta)| - \frac{C\delta}{\rho}  \\
&\geq \nu - \left(\frac{C\delta}{\rho} + D\sqrt{\frac{12C\delta}{\tilde{d}}}\right) \geq \tilde{\nu}.  \qedhere
\end{align*}
\end{proof}

We conclude by estimating the difference of inverses of close functions satisfying Assumption \ref{as:C2fn}.
\begin{lem} \label{l:difinv}
Let $f$ be a function satisfying Assumption \ref{as:C2fn} with boundary derivative constant $\nu$ and containing no critical point. 
Let $g$ be another function satisfing Assumption \ref{as:C2fn} with $\dom g_\pm \subset \dom f_\pm$ and $\|f-g\|_\infty \leq 2\delta$. 
Then for $E_* \in \Ima f \cap \Ima g$, $|T_f(E_*) - T_g(E_*)| \leq \frac{4\delta}{\nu}$.
\end{lem}
\begin{proof}
Since $f$ contains no critical point, by Lemma \ref{l:monotonederiv}, $|\partial_\theta f_\pm(\theta)| \geq \nu$ for all $\theta \in I$, and thus $|\partial_\theta f_\pm^{-1}(E)| \geq \nu^{-1}$ for all $E \in J$. Then
\begin{align*}
    |f_\pm^{-1}(E_*) - g_\pm^{-1}(E_*)| &= 
    |f_\pm^{-1} \circ g \circ g_\pm^{-1}(E_*) - f_\pm^{-1} \circ f \circ g_\pm^{-1}(E_*)| \\&\leq \nu^{-1}|g \circ g_\pm^{-1}(E_*) - f \circ g_\pm^{-1}(E_*)| \\&\leq 2\delta\nu^{-1}.
\end{align*}
Thus, 
\begin{align*}
|T_f(E_*) - T_g(E_*)| \leq |f_+^{-1}(E_*) - g_+^{-1}(E_*)| + |f_-^{-1}(E_*) - g_-^{-1}(E_*)| &\leq 4\delta\nu^{-1}.  \qedhere
\end{align*}
\end{proof}

\newpage
\section{Multiscale spectral analysis}\label{sec:induction}

We wish to understand how the Rellich functions of $H^\Lambda$ inherit the cosine-like properties from the sampling function $v$.  Given $(\theta_*, E_*) \in \mathbb{T} \times \mathbb{R}$, we claim that, by properly choosing our notions of resonance, we can inductively construct increasing intervals $\Lambda_s = \Lambda_s(\theta_*, E_*)$ in the integer lattice, centered at 0, such that any $E_*$-resonant Rellich functions of $H^\Lambda(\theta_*)$ are locally cosine-like with eigenfunctions localized near 0.  Furthermore, these intervals will be stable in the parameters $\theta_*$ and $E_*$.

In broad strokes, the induction will proceed as follows: we suppose that we inductively have been given a cosine-like Rellich function $\mathbf{E}_s : \mathbf{I}_s \to \mathbf{J}_s$ and that certain ``nonresonant'' intervals for $\mathbf{E}_s$ exhibit off-diagonal Green's function decay for energies $E_*$ near the codomain of $\mathbf{E}_s$.  Under these assumptions, we apply Section 4 to $\mathbf{E}_s$ to classify energy regions $J_{s}^{(j_{s+1})} \subset \mathbf{J}_s$ as being simple- or double-resonant (indicated by $j_{s+1} \in \{1,2\}$) and prevent recurrence to those energy regions for long times.  The resonant sites for $\mathbf{E}_s$ will be so well-separated that we can find an even integer $L_{s+1}$ for each energy region well-separated from resonances of all the ancestors of $\mathbf{E}_s$. %
Thus, the intervals $\Lambda_{s+1}$ of length $L_{s+1}$ will have long shoulder intervals with Green's function decay by our inductive assumption.  

Having found the integers $L_{s+1}$, we can then apply Sections 2 and 3 to construct Rellich children $\mathbf{E}_{s+1}$ which are likewise cosine-like.  Finally, for energies which are not resonant with any child $\mathbf{E}_{s+1}$ for long intervals $\Lambda$, we can use these constructed intervals to build coverings of $\Lambda$ like those described in Appendix A, proving the inductive nonresonance hypothesis.  
Here we crucially use the uniform local separation between double-resonant Rellich children established in Proposition \ref{pr:unifmLocalSep}; since the separation is much larger than the scale-$s+1$ resonance parameter, energies near the codomain of one of a pair of double-resonant Rellich functions do not allow resonance with the other.

The perturbative upper bound $\varepsilon_0(v,\alpha)$ on the interaction $\varepsilon$ will be polynomially small, depending on $v$ and $\alpha$, in the initial length scale $L$; the size of $L$ effectively will be dictated by avoiding substantial cumulative loss in the Green's function decay parameters $\gamma_k$, cf. Lemma \ref{l:lenrelns} below.

\begin{figure}[h]
    \centering
\begin{tikzpicture}[x=0.6pt,y=0.6pt,yscale=-1,xscale=1]

\draw  [dash pattern={on 0.84pt off 2.51pt}]  (372,11) -- (341,49) ;
\draw    (322,63) -- (110,134) ;
\draw    (462.89,118) -- (437,106) ;
\draw  [color={rgb, 255:red, 128; green, 128; blue, 128 }  ,draw opacity=1 ][dash pattern={on 4.5pt off 4.5pt}] (371,148.5) .. controls (371,125.58) and (433.46,107) .. (510.5,107) .. controls (587.54,107) and (650,125.58) .. (650,148.5) .. controls (650,171.42) and (587.54,190) .. (510.5,190) .. controls (433.46,190) and (371,171.42) .. (371,148.5) -- cycle ;
\draw    (535,134) -- (462.89,118) ;
\draw    (462.89,118) -- (462.89,134) ;
\draw    (348,74) -- (407,134) ;
\draw    (360,63) -- (602,134) ;
\draw  [color={rgb, 255:red, 128; green, 128; blue, 128 }  ,draw opacity=1 ][dash pattern={on 4.5pt off 4.5pt}] (38,148.5) .. controls (38,125.58) and (100.46,107) .. (177.5,107) .. controls (254.54,107) and (317,125.58) .. (317,148.5) .. controls (317,171.42) and (254.54,190) .. (177.5,190) .. controls (100.46,190) and (38,171.42) .. (38,148.5) -- cycle ;
\draw    (185,134) -- (323,69) ;
\draw    (271,133) -- (331,75) ;
\draw    (356,67) -- (408,90) ;

\draw (331,52.4) node [anchor=north west][inner sep=0.75pt]    {\footnotesize $\mathbf{E}_{s}$};
\draw (69,82.4) node [anchor=north west][inner sep=0.75pt]    {\footnotesize $\mathcal{E}^{( 1)}(\mathbf{E}_{s})$};
\draw (555,83.4) node [anchor=north west][inner sep=0.75pt]    {\footnotesize $\mathcal{E}^{( 2)}(\mathbf{E}_{s})$};
\draw (157,140.4) node [anchor=north west][inner sep=0.75pt]    {\footnotesize $\mathbf{E}_{s+1,i}^{( 1)}$};
\draw (435.22,139.4) node [anchor=north west][inner sep=0.75pt]    {\footnotesize $\mathbf{E}_{s+1,n_{s} ,\vee }^{( 2)}$};
\draw (523.22,138.4) node [anchor=north west][inner sep=0.75pt]    {\footnotesize $\mathbf{E}_{s+1,n_{s} ,\wedge }^{( 2)}$};
\draw (396.22,87.4) node [anchor=north west][inner sep=0.75pt]    {\footnotesize $n_{s} \in \mathcal{N}_{s}$};
\draw (398,139.4) node [anchor=north west][inner sep=0.75pt]    {\footnotesize $\cdots $};
\draw (591,139.4) node [anchor=north west][inner sep=0.75pt]    {\footnotesize $\cdots $};
\draw (255,139.4) node [anchor=north west][inner sep=0.75pt]    {\footnotesize $\cdots $};
\draw (89,139.4) node [anchor=north west][inner sep=0.75pt]    {\footnotesize $\cdots $};
\draw (115,200) node [anchor=north west][inner sep=0.75pt]    {\footnotesize $L_{s+1} \sim M_{s+1}^{( 1)}$};
\draw (460,200) node [anchor=north west][inner sep=0.75pt]    {\footnotesize $L_{s+1} \sim M_{s+1}^{( 2)}$};

\end{tikzpicture}

    \caption{The inductive expansion of our Rellich tree at a particular node $\mathbf{E}_s$.}
    \label{f:treeexpn}
\end{figure}
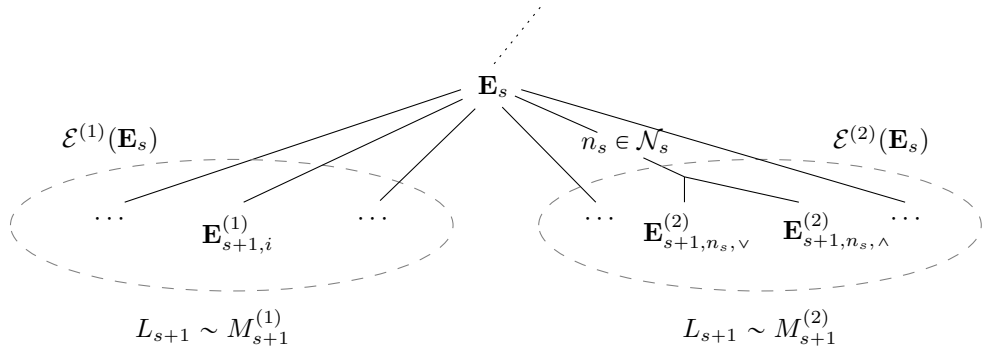

\subsection{Preparing the induction}

\subsubsection{The initial scale}

Suppose $\mathbf{E}_0 := v$, $v :\mathbb{T}\to [-1,1]$, is cosine-like in the sense of the main theorem, with Morse constants $$d_0 \leq |\partial_\theta v(\theta)| + |\partial_\theta^2 v(\theta)| \leq D_0, \quad \theta \in \mathbb{T},$$ and that $\alpha \in DC_{C,\tau}$ is Diophantine.  Letting $L$ be a large even integer, we consider the initial length scales $$L_1^{(1)} := L, \quad L_1^{(2)} := L^2,$$ and initialize the following parameters:
$$\nu_0 = \frac{C_\alpha d_{0}}{14 (8L)^\tau}, \quad \rho_{-1} = \frac{d_0}{24D_0^2}\frac{1}{(8L)^{2\tau}}, \quad \bar{\rho}_0 = \frac{d_0}{15000 D_0^2}\rho_{-1}^2,$$ $$\rho_0 = \bar{\rho}_0^3, \quad \delta_0 := \varepsilon < \varepsilon_0 := \rho_0^4, \quad \gamma_{-1}= \gamma_0 = \frac14|\log\varepsilon|, \quad \ell_{-1}= \ell_0 = 1.$$
With these parameters, we recall a classical Neumann series argument ensuring off-diagonal Green's function decay on nonresonant intervals:
\begin{lem}\label{l:initGrnDec}
	Let $\Lambda \subset \mathbb{Z}$ and suppose $|v(\theta_* + m\alpha) - E_*| \geq \rho_0$ for $m \in \Lambda$.  Then for $m, n \in \Lambda$, $|\theta- \theta_*| < \rho_0/8D_0$, and $|E - E_*| < \rho_0/2$,
	\begin{align*}
		|R^\Lambda_{\theta,E}(m,n)| \leq \frac{8}{\rho_0}\left(\frac{8\varepsilon}{\rho_0}\right)^{|m-n|}.
	\end{align*}
	In particular, for $\gamma_0 = \frac{1}{4}|
	\log\varepsilon|$ and $|m-n|\geq 1$,
	\begin{align}\label{eq:InitDec}
		\log|R^\Lambda_{\theta,E}(m,n)| \leq -\gamma_0|m-n|,
	\end{align}
	i.e. the Green's function decay property for $(\ell_0, \gamma_0)$ holds.
\end{lem}
\begin{proof}
	For $m \in \Lambda$, $|\theta - \theta_*| < \rho_0/8D_0$, and $|E - E_*| < \rho_0/2$, one has
	\begin{align*}
		|v(\theta + m\alpha) - E| &\geq |v(\theta_* + m\alpha) - E_*| - |E-E_*| - D_0|\theta - \theta_*| > \rho_0/4.
	\end{align*}
	Thus $|v(\theta + m\alpha) - E| \geq \rho_0/4 > 4\varepsilon$ for all $m \in \Lambda$, and $\|(V(\theta) - E)^{-1}\varepsilon\Delta\|_{\mathbb{C}^\Lambda} < 1/2$; hence, by the Neumann series, one has for $m, n \in \Lambda$
	\begin{align*}
		|R^\Lambda_{\theta,E}(m,n)| &= |\langle \delta_m, (H^\Lambda(\theta) - E)^{-1} \delta_n \rangle| \\
		&\leq \frac{4}{\rho_0} \sum_{k \geq |m-n|} \left(\frac{8\varepsilon}{\rho_0}\right)^k \\
		&\leq \frac{8}{\rho_0}\left(\frac{8\varepsilon}{\rho_0}\right)^{|m-n|}.
	\end{align*}
	Provided $|m-n| \geq 1$, we have
	\begin{align*}|R^\Lambda_{\theta,E}(m,n)| \leq \frac{8}{\rho_0}\left(\frac{8\varepsilon}{\rho_0}\right)^{|m-n|} &\leq \rho_0^{2|m-n|-1} \leq \rho_0^{|m-n|}.\qedhere\end{align*}
\end{proof}
\noindent This lemma, combined with our assumptions on $v$, will be the foundation of our induction.

\subsubsection{Inductive definitions}

Fix a scale $s \geq 0$.  If $s \geq 1$, we suppose that, for $0 \leq k \leq s-1$, we have collections 
\begin{align*}
	\mathcal{E}_k^{(j)} &= \bigcup_{\mathbf{E} \in \mathcal{E}_{k-1}} \mathcal{E}^{(j)}(\mathbf{E}), \quad j \in \{1,2\}, \\
	\mathcal{E}_k &= \mathcal{E}_k^{(1)} \cup \mathcal{E}_k^{(2)}
\end{align*} 
of Rellich functions $\mathbf{E}_k : \mathbf{I}_k \to \mathbf{J}_k$ of certain Dirichlet restrictions $H^{\Lambda_k}$, $\Lambda_k = [-L_k/2, L_k/2]$, where $L_k = L_k(\mathbf{E}_k)$ are even integers on scale $k$
\begin{align*}
	\frac{1}{2}M_k^{(j)} \leq L_k \leq M_k^{(j)}, \quad \mathbf{E}_k \in \mathcal{E}_k^{(j)},
\end{align*}
where we define the length scales
$$M_k^{(j)} := (L^{4^{k-1}})^j, \quad  k \geq 1,\;  j \in \{1,2\}.$$
We take as convention that $M_0^{(1)} = M_0^{(2)} = 1$.  
For each $\mathbf{E}_k \in \mathcal{E}_k$,  $k \geq 1$, we define the corresponding parameters:
\vspace{0.1in}
\begin{enumerate}
	\item[] $d_k = \frac{\nu_{k-1}}{12}$, $D_k = 2D_0(1+\frac{D_0}{\sigma_k})$, \hfill (Morse)
	\item[]  $\nu_k = \bar{\nu}_{k-1}/2$, $\bar\nu_k := \frac{3d_k}{8D_0}\bar{\rho}_{k}$ \hfill (Derivative control)
	\item[] $\rho_k = \varepsilon^{L_k^{2/3}}$, $\bar{\rho}_k = \frac{d_k}{15000 D_0^2}\rho_{k-1}^2$, $\sigma_k =  \varepsilon^{72M_k^{(1)}}$ \hfill (Resonance) %
	\item[] $\delta_{k} = \varepsilon^{L_k/8}$ \hfill (Eigenfunction interaction)
	\item[] $\ell_k = L_k^{5/6}$, $\gamma_k = \gamma_0(1- 64\sum_{i=1}^k\frac{|\log(\rho_i/8)|}{\ell_i})$ \hfill  (Green's function decay)
\end{enumerate}
\vspace{0.1in}
We likewise define $\salt{\mathbf{J}}(\mathbf{E}_k)$ and $\salt{\salt{\mathbf{J}}}(\mathbf{E}_k)$ as in equation \eqref{eq:Jcheck} with $\rho = \rho_k$ and $\sprev{\rho} = \rho_{k-1}$, cf. Figure \ref{f:Jcheck}. As matters of convention, we define $\salt{\mathbf{J}}(\mathbf{E}_{-1}) = \salt{\salt{\mathbf{J}}}(\mathbf{E}_{-1}) = \mathbb{R}$ and $\mathcal{E}_0 = \mathcal{E}(\mathbf{E}_{-1}) := \{v\}$.

We have the following relations among these parameters: 
\begin{lem}\label{l:lenrelns}
	Let $\mathbf{E}_k \in \mathcal{E}_k$ with $L_k$ and constants defined as above.  For $L$ sufficiently large (depending only on $C_\alpha$, $\tau$, $d_0$, and $D_0$), we have the following for all $k \geq 0$:
	\begin{enumerate}
		\item $\gamma_k \geq \frac{1}{2}\gamma_0 \geq \log 7$.
		\item $10M_k^{(2)} < M_{k+1}^{(1)}$, $24L_{k+1}^{2/3} < 8L_{k+1}^{5/6} < L_{k+1}$, $\ell_{k+1} \gg 16|\log\varepsilon||\log \rho_{k+1}|$.
		\item $\frac{C_\alpha}{2(8M_{k+1}^{(1)})} \geq \frac{7\nu_k}{d_k}$, $\frac{3d_k}{8D_0}\bar{\rho}_k < \nu_k$, and $\frac{9\rho_k}{8\nu_k}  \ll \bar{\rho}_k \ll \frac{\nu_k^2\rho_{k-1}}{100D_0^3}$.
		\item $\delta_k < \rho_k^3/2 \ll \rho_{k-1}^2\nu_k^{5/2}$.
		\item If $j_k = 2$, then $2\rho_k < \frac{\nu_{k-1}\sigma_k}{2D_0+\nu_{k-1}}$.
		\end{enumerate}
\end{lem}
\begin{proof}
    We proceed with each item in turn:
	\begin{enumerate}
		\item For $1 \leq i \leq k$, we have by definition that
		\begin{align*}
			\frac{|\log(\rho_i/8)|}{\ell_i} &\leq L_i^{-1/6}|\log\varepsilon| \lesssim \tau L^{-\frac{1}{6}4^{i-1}}\log L \lesssim \tau L^{-\frac{1}{7}(4^{i-1})}
		\end{align*}
		Since this sequence is subgeometric, we may find $L$ sufficiently large such that 
		\begin{align*}
			64\sum_{i=1}^k\frac{|\log(\rho_i/8)|}{\ell_i} \leq \frac{1}{2}
		\end{align*}
		independently of $k$.  By possibly making $L$ larger, one can insist that $$\frac12\gamma_0 = \frac18\log|\varepsilon| \sim \tau\log|L| > \log 7.$$
		\item This is immediate from the fact that, for $k \geq 1$, $$M_{k+1}^{(1)} \geq L M_k^{(2)},$$ alongside the observation that $|\log\varepsilon|$ is comparable to $\log L$.
		\item 
		One can check by definitions that, for $L$ sufficiently large (depending on $\tau$, $d_0$, and $D_0$), one has for $k \geq 0$ that:
		\begin{align*}
			\nu_{k}^{5/4} < d_{k+1} < \nu_{k} < \frac{d_{k}}{2} \\
			\rho_{k-1}^{9/4} < \bar{\rho}_k < \rho_{k-1}^{2} \\
			\rho_{k-1}^{9/4} < \nu_{k+1}, \bar{\nu}_k < \rho_{k-1}^2 \\
			\rho_k \ll \rho_{k-1}^6.
		\end{align*}
		The inequalities follow for sufficiently large $L$. 
		
		\item 
		Given the above inequalities, one has $$\rho_{k-1}^2\nu_k^{5/2} \gg \rho_{k-1}^2\rho_{k-2}^{45/8} \gg \rho_{k-1}^2\rho_{k-2}^6 \gg \rho_{k-1}^3 \gg \rho_k^3/2,$$
		and $L_k > 24L_k^{2/3}$.
		
		\item 
		Since $j_k = 2$, we have $L_k^{2/3} \geq (M_k^{(2)}/2)^{2/3} \geq \frac{1}{2}(M_k^{(1)})^{4/3} \gg 144M_k^{(1)}$ provided $L$ is sufficiently large,
		and thus $\log\rho_{k} \ll 2\log\sigma_{k}$. The inequality follows because $\log\nu_{k-1} \gg \log\rho_{k-1} \gg \log\sigma_k$.
		\qedhere
	\end{enumerate}
\end{proof}

It will be convenient to fix language describing ``nonresonant'' regions at each scale.  Fix $(\theta_*,E_*) \in \mathbb T \times \mathbb R$ and a Rellich function $\mathbf E_k \in \mathcal E_k$. We say that a set $S \subset \mathbb Z$ is \emph{$k$-nonresonant} (relative to $(\theta_*,E_*)$ and $\mathbf E_k$) if $$\theta_*+m\alpha \notin \mathbf I_k \,\text{ or }\, |\mathbf E_k(\theta_*+m\alpha)-E_*| \geq \rho_k \quad \forall m \in S.$$ Note that, with this language, Lemma \ref{l:initGrnDec} shows that $0$-nonresonant intervals have $(\ell_0, \gamma_0)$ Green's function decay.

To show Green's function decay at future scales $k \geq 1$, we must avoid resonant sites for all ancestors of $\mathbf{E}_k$; we codify this condition as $k$-regularity.  Namely, we say that a point $m \in \mathbb Z$ is \emph{$k$-left-regular} if, for each ancestor $\mathbf E_i$ ($0 \leq i \leq k-1$) of $\mathbf E_k$, 
$[m,m+\frac34L_{i+1}]$ is $i$-nonresonant. Analogously, we say that $m$ is \emph{$k$-right-regular} if, for each ancestor $\mathbf E_i$ ($0 \leq i \leq k-1$) of $\mathbf E_k$, 
$[m-\frac34L_{i+1},m]$ is $i$-nonresonant, and we say that $m$ is \emph{$k$-regular} if it is both $k$-left-regular and $k$-right-regular.

Finally, we define the slight enlargement of the set of $k$-resonant points by $$\mathcal{S}_k(\theta_*, E_*) := \left\{m \in \mathbb{Z} : \theta_*+m\alpha \in \mathbf{I}_k, \; |\mathbf{E}_k(\theta_*+ m\alpha) - E_*| \leq \frac{25}{24}\rho_k\right\}.$$

\subsection{The inductive proposition}

Subject to these definitions, we suppose that the following inductive proposition holds for $0 \leq k \leq s-1$:

\begin{prop}[Induction, scale $k$]\label{pr:induction}
	Suppose the hypotheses below hold:
	\begin{hypsub}[Cosine-like Rellich functions, scale $k$]\label{prsub:c2rell}
		Each $\mathbf{E}_k \in \mathcal{E}_k$ satisfies Assumption \ref{as:C2fn} with parameters $\nu = \nu_k, d = d_k,$ and $D = D_k$.  
	\end{hypsub}
	\begin{hypsub}[Resonant orbits, scale $k$]\label{pr:resorbs}  
		If $k \geq 1$, let $\mathbf{E}_k \in \mathcal{E}^{(j_k)}(\mathbf{E}_{k-1})$, $\mathbf{E}_k : \mathbf{I}_k \to \mathbf{J}_k$, denote $$\mathcal{N}_k := \mathcal{N}(8M_{k+1}^{(1)},\alpha, \mathbf{E}_k),$$ and let $\theta_* \in \mathbb{T}$.  We can characterize aspects of the resonant orbits depending on the resonance type $j_k$:
		\begin{enumerate}
			\item $j_k = 1$:
			
			We have $$\min\{|n| : n \in \mathcal{N}_{k}\} > 8M_{k}^{(1)} > L_{k}.$$  Furthermore, suppose either $E_* \in \salt{\mathbf{J}}(\mathbf{E}_k)$ or the closest point in $\mathbf{J}_k$ to $E_*$ is a critical value of $\mathbf{E}_k$.  Then, for any $m \in \mathcal{S}_k(\theta_*, E_*)$, $$B_{\rho_{k-1}/24D_0}(\theta_*+ m\alpha) \subset \mathbf{I}_k.$$
			
			\item $j_k = 2$ \textit{(and thus $\mathbf{E}_k = \mathbf{E}_{k,n_{k-1},\bullet}^{(2)}$)}:
			
			We have $$\mathcal{N}_k = \emptyset.$$  Furthermore, suppose $E_* \in [\inf \salt{\mathbf{J}}(\mathbf{E}_{k,n_{k-1},\wedge}^{(2)}), \sup \salt{\mathbf{J}}(\mathbf{E}_{k,n_{k-1},\vee}^{(2)})].$  Then, for any $m \in \mathcal{S}_k(\theta_*, E_*)$, for some $p \in \{m, m-n_{k-1}\}$, $$B_{\rho_{k-1}/24D_0}(\theta_* + p\alpha) \subset \mathbf{I}_{k,n_{k-1},\vee} \cup \mathbf{I}_{k,n_{k-1},\wedge}.$$  If $\theta_* + p\alpha \in \mathbf{I}_{k,n_{k-1},\wedge}^{(2)} \setminus \mathbf{I}_{k,n_{k-1},\vee}^{(2)}$, then $|\mathbf{E}_{k,n_{k-1},\wedge}^{(2)}(\theta_*+p\alpha) - E_*| < \frac{9}{8}\rho_{k-1}$; similarly, if $\theta_* + p\alpha \in \mathbf{I}_{k,n_{k-1},\vee}^{(2)} \setminus \mathbf{I}_{k,n_{k-1},\wedge}^{(2)}$, then $|\mathbf{E}_{k,n_{k-1},\vee}^{(2)}(\theta_*+p\alpha) - E_*| < \frac{9}{8}\rho_{k-1}$.
		\end{enumerate}
	\end{hypsub}
	\begin{hypsub}[Nonresonance, scale $k$]\label{prsub:NR}
	    Let $\mathbf{E}_{k-1} \in \mathcal{E}_{k-1}$, let $\theta_* \in \mathbb{T}$, and suppose $E_* \in \salt{\mathbf{J}}(\mathbf{E}_{k-1})$ is such that
	    \begin{itemize}
		\item If $\mathbf{E}_{k-1}$ does not attain its maximum at a critical point, then $E_* \leq \sup \bigcup_{\mathbf{E} \in \mathcal{E}(\mathbf{E}_{k-1})}\salt{\mathbf{J}}(\mathbf{E})$.
		\item If $\mathbf{E}_{k-1}$ does not attain its minimum at a critical point, then $E_* \geq \inf \bigcup_{\mathbf{E} \in \mathcal{E}(\mathbf{E}_{k-1})}\salt{\mathbf{J}}(\mathbf{E})$.
	\end{itemize}
		Let $\mathbf{E}_k \in \mathcal{E}(\mathbf{E}_{k-1})$ be a Rellich curve minimizing $\dist(E_*, \salt{\mathbf{J}}(\mathbf{E}))$ among $\mathbf{E} \in \mathcal{E}(\mathbf{E}_{k-1})$.
		Suppose $\Lambda = [a,b] \subset \mathbb{Z}$ is $k$-nonresonant and that $a$ and $b$ are $k$-left- and $k$-right-regular, respectively.  Then for $|E - E_*| < \rho_k/2$ and $\theta \in B_{\rho_k/8D_0}(\theta_*)$,
		\begin{align*}
			\log|R^{\Lambda}_{\theta,E}(m,n)| \leq -\gamma_{k}|m-n|, \quad |m-n| \geq \ell_{k},
		\end{align*}
		and
		\begin{align*}
			\|R^{\Lambda}_{\theta,E}\| \leq 4\rho^{-1}_k.
		\end{align*}
	
		Furthermore, if $\salt{\Lambda}$ is $k$-nonresonant with 
		$|\salt{\Lambda}| \geq 2L_k+3M_k^{(1)}$, then
		there exists a $k$-nonresonant subinterval $\Lambda = [a,b] \subset \salt{\Lambda}$ with $|\Lambda|\geq|\salt{\Lambda}|-(2L_k+3M_k^{(1)})$ such that $a$ and $b$ are $k$-left- and $k$-right-regular, respectively. 
		
	\end{hypsub}

Subject to the above hypotheses, we have the following:	Let $\mathbf{E}_k \in \mathcal{E}_{k}$ with $\mathbf{E}_k : \mathbf{I}_k \to \mathbf{J}_k$.  There exists a collection $\mathcal{I}_{k+1}(\mathbf{E}_k) = \mathcal{I}_{k+1}^{(1)} \cup \mathcal{I}_{k+1}^{(2)}$ of intervals $I_{k+1}$  and corresponding even integers $L_{k+1}$ so that, denoting $\Lambda_{k+1} = [-L_{k+1}/2, L_{k+1}/2]$, the following hold:
	\begin{enumerate}
		\item[](Simple resonance): If $I_{k+1} \in \mathcal{I}_{k+1}^{(1)}$, $\theta_*\in I_{k+1}$, and $E_* = \mathbf{E}_k(\theta_*)$, then $H^{\Lambda_{k+1}}(\theta_*)$ satisfies Assumption \ref{as:SR} with \begin{gather*}
			\snext\Lambda = \Lambda_{k+1}, \; {\Lambda} = \Lambda_k, \\
			\rho = \rho_k,\; \delta = \delta_k,\; \gamma = \gamma_k,\; \ell = \ell_k.
		\end{gather*}
		\item[](Double resonance): If $I_{k+1} \in \mathcal{I}_{k+1}^{(2)}$, there exists a unique $n$ with $L_k < |n| \leq 8M_{k+1}^{(1)}$ and a unique point $\theta_{**} = \theta_{k,n,-} \in I_{k+1}$ such that $E_{**} = \mathbf{E}_{k}(\theta_{**}) = \mathbf{E}_k(\theta_{**}+n\alpha)$.  Furthermore, $H^{\Lambda_{k+1}}(\theta_{**})$ satisfies Assumption \ref{as:DR} with the following assignments:
		\begin{gather*}
			\snext\Lambda = \Lambda_{k+1},\; {\Lambda}_- = \Lambda_k,\; {\Lambda}_+ = \Lambda_k + n, \\
			\sprev{\rho} = \frac23\rho_{k-1},\; \delta = \delta_k,\; \gamma = \gamma_k,\; \ell = \ell_k,\; \nu = \nu_k.
		\end{gather*}
		Moreover, we can choose the additional parameters introduced in Section \ref{sec:DR} as follows:
		\begin{gather*}
		    \snext\sigma = \sigma_{k+1},\; 
		    \eta = \frac{5\rho_k}{4\nu_k},\; 
		    \rho = \rho_k.
		\end{gather*}
	\end{enumerate}
	In particular, there exists a family $\mathcal{E}(\mathbf{E}_k) = \mathcal{E}^{(1)}(\mathbf{E}_k) \cup \mathcal{E}^{(2)}(\mathbf{E}_k)$ of Rellich children $\mathbf{E}_{k+1} : \mathbf{I}_{k+1} \to \mathbf{J}_{k+1}$  such that each $\mathbf{E}_{k+1}$ is a Rellich function of $H^{\Lambda_{k+1}}$ satisfying Assumption \ref{as:C2fn}, such that Hypotheses \ref{prsub:c2rell}, \ref{pr:resorbs}, and \ref{prsub:NR} hold at scale $k+1$.
	
	Additionally, we have the following:
	\begin{enumerate}
		\item The even integers $L_{k+1}$ can be chosen such that $\pm L_{k+1}/2$ and $\pm(L_{k+1}/2+1)$ are $k$-regular for any $\theta_* \in \mathbf{I}_{k+1}$ for any $E_* \in B_{\rho_k}(\mathbf{E}_{k}(\theta_*))$, and such that $L_{k+1}/2+1$ (respectively, $-(L_{k+1}/2+1)$) is $k+1$-left-regular (respectively, $k+1$-right-regular) for any $E_* \in B_{\rho_{k}}(\mathbf{E}_{k+1}(\theta_*))$ 
		\item The intervals $\mathbf{I}_{k+1}$ and their relevant translates cover $\mathbf{I}_k$, in the sense that if $\theta \in \mathbf{I}_k$ and $\theta \notin \mathbf{I}_{k+1}$ for all $\mathbf{E}_{k+1} \in \mathcal{E}(\mathbf{E}_k)$, then there exists a unique $n$ with $8M_{k}^{(1)} < |n| \leq 8M_{k+1}^{(1)}$ and a Rellich curve $\mathbf{E}_{k+1} = \mathbf{E}_{k+1,n,\bullet} \in \mathcal{E}^{(2)}(\mathbf{E}_k)$ such that $\theta \in \mathbf{I}_{k+1} + n\alpha$.
		
		\item If $E_* = E_{n_{k+1}}(\mathbf E_{k+1})$ for $n_{k+1} \in \mathcal N_{k+1}$,  then $\overline{B}_{\rho_{k}}(E_*) \subset \salt{\mathbf{J}}(\mathbf{E}_k)$.  Furthermore, we have 
		\begin{align*}
			\bar{\rho}_{k}/2 \leq |\mathbf{J}_{k+1}| \leq 2\bar{\rho}_{k}.
		\end{align*}
	\end{enumerate}
\end{prop}

Subject to these hypotheses at scale $s-1$, we verify that the statement holds at scale $s$ with appropriately defined constants:
\begin{thm} \label{t:induction}
	There exists $\varepsilon_0 = \varepsilon_0(\alpha, v)$ such that, for $\varepsilon \leq \varepsilon_0$, Proposition \ref{pr:induction} holds for all $k \geq 0$.
\end{thm}

For the remainder of this section, we suppose that the proposition holds for $0 \leq k \leq s-1$, that Hypotheses \ref{prsub:c2rell}, \ref{pr:resorbs}, and \ref{prsub:NR} hold for $0 \leq k \leq s$, and prove the proposition for $k = s$.  We note that, in light of Lemma \ref{l:initGrnDec}, the proof below also verifies the initial case $s=0$.

\subsection{The family of descendants of $\mathbf{E}_s$}

If $s = 0$, let $\mathbf{E}_s = v$; otherwise, let $\mathbf{E}_s \in \mathcal{E}^{(j_s)}(\mathbf{E}_{s-1})$, $\mathbf{E}_s : \mathbf{I}_s \to \mathbf{J}_s$, $j_s \in \{1,2\}$, be a Rellich function of $H^{\Lambda_s}$, $\Lambda_s = [-L_s/2,L_s/2]$ constructed as in Proposition \ref{pr:induction} at scale $s-1$.  In the case $j_s = 2$ (and thus $\mathbf{E}_s = \mathbf{E}_{s,n_{s-1},\bullet}^{(2)}$), we denote $$\mathbf{I}_{s,n_{s-1},\cup} := \mathbf{I}_{s,n_{s-1},\wedge}^{(2)}\cup \mathbf{I}_{s,n_{s-1},\vee}^{(2)}.$$

\subsubsection{Construction of the next length scale}
By our inductive hypotheses, $\mathbf{E}_s$ satisfies Assumption \ref{as:C2fn}, with Morse constants $d_s, D_s$ and boundary derivative lower bound $\nu_s$.  By Lemma \ref{l:lenrelns}, Section \ref{sec:invfn} applies to $\mathbf{E}_s$ with these values, and, recalling that $$\bar{\rho}_s = \frac{d_s}{15000D_0^2}\rho_{s-1}^2,$$ we have the following:
\begin{lem}\label{l:constructionsummary}
	$\salt{\salt{\mathbf{J}}}(\mathbf{E}_s) \cap \mathbf{J}_s$ can be covered by closed intervals $J_{s,i}^{(j)}$, $j \in \{1,2\}$, such that
	\begin{enumerate}
		\item $J_{s,n_{s}}^{(2)} = \overline{B}_{\bar{\rho}_s}(E_{n_s})$, where $\theta_* = \theta_{s,n_{s},-}$ is the unique point in $\mathbf{I}_s$ such that $$E_{n_{s}} := \mathbf{E}_s(\theta_*) = \mathbf{E}_s(\theta_*+n_{s}\alpha)$$
		for any $n_{s} \in \mathcal{N}_{s} := \mathcal{N}(8M_{s+1}^{(1)},\alpha,\mathbf{E}_s)$ (cf. Section 4).
		\item $\bar{\rho}_s \leq |J_{s,i}^{(j)}| \leq 2\bar{\rho}_s$
		\item $J_{s,i} \cap J_{s,i'} \neq \emptyset \implies |J_{s,i}\cap J_{s,i'}| = 3\rho_s$
		\item For any $E \in J_{s,i}^{(j)}$, $\overline{B}_{\rho_s}(E) \subset \salt{\mathbf{J}}(\mathbf{E}_s)$.  Furthermore, if $j=2$ and $E_* = E_{n_s}$, then $\overline{B}_{\rho_{s-1}}(E_*) \subset \salt{\mathbf J}(\mathbf E_{s-1})$.
		\item The total number of such intervals $J_{s,i}^{(j)}$ does not exceed $5\lceil |\mathbf{J}_s|\bar{\rho}_s^{-1}\rceil$.
	\end{enumerate}
	These intervals are defined such that, on the corresponding preimages
	\begin{align*}
		I_{s,i,\pm}^{(j)} := 	(\mathbf{E}_s)_\pm^{-1}(J_{s,i}^{(j)}),
		\quad I_{s,i}^{(j)} := I_{s,i,-}^{(j)} \cup I_{s,i,+}^{(j)},
	\end{align*}
	one has either
\begin{enumerate}
		\item[] (Simple resonance): If $\theta_* \in I_{s,i}^{(1)}$, then for any $n$ with $ 0 \neq |n| \leq 8M_{s+1}^{(1)}$ such that $\theta_* + n\alpha \in \mathbf{I}_s$,
		\begin{align*}
			|\mathbf{E}_s(\theta_*) - \mathbf{E}_s(\theta_*+n\alpha)| > 3\rho_s.
		\end{align*}
		\item[] (Double resonance): If $\theta_* = \theta_{s,n_s,-} \in I_{s,n_s,-}^{(2)}$, then for $\theta \in B_{\rho_{s-1}/8D_0}(\theta_*)$ and any $n$ with $n \notin \{0,n_s\}, |n| \leq 8M_{s+1}^{(2)}$ such that $\theta + n\alpha \in \mathbf{I}_s$,
		\begin{align*}
			|\mathbf{E}_s(\theta) - \mathbf{E}_s(\theta+n\alpha)| > 2\rho_{s-1}
		\end{align*}
		and
		\begin{align*}
			\min\{-\partial_\theta \mathbf{E}_s(\theta), \partial_\theta \mathbf{E}_s(\theta+n_s\alpha)\} \geq \nu_s.
		\end{align*}
	\end{enumerate}
	Finally, recalling that 
	\begin{equation*}
		\bar\nu_s = \frac{3d_s}{8D_0}\bar{\rho}_{s} < \nu_s,
	\end{equation*}
	each function $\mathbf E_s|_{I_{s,i}^{(j)}}: I_{s,i}^{(j)} \to J_{s,i}^{(j)}$ satisfies Assumption \ref{as:C2fn} with Morse constants $d_s,D_s$ and boundary derivative constant $\bar\nu_s$. 
\end{lem}
\begin{proof}
    This follows immediately from Propositions \ref{pr:C2cover} and \ref{pr:SRDRsep}, except the second half of item 4,
    which follows from the inductive proposition.
\end{proof}

\begin{figure}
    \centering
    
\begin{tikzpicture}[x=0.75pt,y=0.75pt,yscale=-1,xscale=1]

\draw [line width=2.25]    (152,31) .. controls (206.5,299.47) and (332.5,276.72) .. (595,31) ;
\draw [color={rgb, 255:red, 128; green, 128; blue, 128 }  ,draw opacity=1 ][line width=1.5]    (151,304.02) -- (594,304.02) ;
\draw [shift={(594,304.02)}, rotate = 180] [color={rgb, 255:red, 128; green, 128; blue, 128 }  ,draw opacity=1 ][line width=1.5]    (0,6.71) -- (0,-6.71)   ;
\draw [shift={(151,304.02)}, rotate = 180] [color={rgb, 255:red, 128; green, 128; blue, 128 }  ,draw opacity=1 ][line width=1.5]    (0,6.71) -- (0,-6.71)   ;
\draw [color={rgb, 255:red, 128; green, 128; blue, 128 }  ,draw opacity=1 ][line width=1.5]    (607,74.13) -- (607,226.12) ;
\draw [shift={(607,226.12)}, rotate = 270] [color={rgb, 255:red, 128; green, 128; blue, 128 }  ,draw opacity=1 ][line width=1.5]    (0,6.71) -- (0,-6.71)   ;
\draw [shift={(607,74.13)}, rotate = 270] [color={rgb, 255:red, 128; green, 128; blue, 128 }  ,draw opacity=1 ][line width=1.5]    (0,6.71) -- (0,-6.71)   ;
\draw [color={rgb, 255:red, 128; green, 128; blue, 128 }  ,draw opacity=1 ] [dash pattern={on 0.84pt off 2.51pt}]  (102,123.83) -- (614.5,123.83) ;
\draw [line width=2.25]    (102,85.6) -- (102,162.05) ;
\draw [shift={(102,162.05)}, rotate = 270] [color={rgb, 255:red, 0; green, 0; blue, 0 }  ][line width=2.25]    (0,7.83) -- (0,-7.83)   ;
\draw [shift={(102,85.6)}, rotate = 270] [color={rgb, 255:red, 0; green, 0; blue, 0 }  ][line width=2.25]    (0,7.83) -- (0,-7.83)   ;
\draw   (126,124.14) .. controls (130.67,124.14) and (133,121.81) .. (133,117.14) -- (133,117.14) .. controls (133,110.47) and (135.33,107.14) .. (140,107.14) .. controls (135.33,107.14) and (133,103.81) .. (133,97.14)(133,100.14) -- (133,93.83) .. controls (133,89.16) and (130.67,86.83) .. (126,86.83) ;
\draw [color={rgb, 255:red, 128; green, 128; blue, 128 }  ,draw opacity=1 ] [dash pattern={on 0.84pt off 2.51pt}]  (178,124.74) -- (178,314.04) ;
\draw   (485.5,121.1) .. controls (485.5,116.43) and (483.17,114.1) .. (478.5,114.1) -- (285.62,114.1) .. controls (278.95,114.1) and (275.62,111.77) .. (275.62,107.1) .. controls (275.62,111.77) and (272.29,114.1) .. (265.62,114.1)(268.62,114.1) -- (190.5,114.1) .. controls (185.83,114.1) and (183.5,116.43) .. (183.5,121.1) ;
\draw [color={rgb, 255:red, 128; green, 128; blue, 128 }  ,draw opacity=1 ] [dash pattern={on 0.84pt off 2.51pt}]  (486,122.01) -- (486,312.22) ;
\draw [color={rgb, 255:red, 155; green, 155; blue, 155 }  ,draw opacity=1 ][line width=0.75]  [dash pattern={on 4.5pt off 4.5pt}]  (101,85.6) -- (534,85.6) ;
\draw [color={rgb, 255:red, 128; green, 128; blue, 128 }  ,draw opacity=1 ][line width=0.75]  [dash pattern={on 4.5pt off 4.5pt}]  (101,162.05) -- (438,162.05) ;
\draw [color={rgb, 255:red, 128; green, 128; blue, 128 }  ,draw opacity=1 ] [dash pattern={on 4.5pt off 4.5pt}]  (438,162.05) -- (438,304.93) ;
\draw [color={rgb, 255:red, 128; green, 128; blue, 128 }  ,draw opacity=1 ] [dash pattern={on 4.5pt off 4.5pt}]  (534,85.6) -- (534,304.93) ;
\draw [line width=2.25]    (438,304.93) -- (534,304.93) ;
\draw [shift={(534,304.93)}, rotate = 180] [color={rgb, 255:red, 0; green, 0; blue, 0 }  ][line width=2.25]    (0,7.83) -- (0,-7.83)   ;
\draw [shift={(438,304.93)}, rotate = 180] [color={rgb, 255:red, 0; green, 0; blue, 0 }  ][line width=2.25]    (0,7.83) -- (0,-7.83)   ;
\draw  [dash pattern={on 0.84pt off 2.51pt}]  (486,337.15) -- (486,352.62) ;
\draw [color={rgb, 255:red, 128; green, 128; blue, 128 }  ,draw opacity=1 ] [dash pattern={on 4.5pt off 4.5pt}]  (165,85.6) -- (165,304.93) ;
\draw [color={rgb, 255:red, 128; green, 128; blue, 128 }  ,draw opacity=1 ] [dash pattern={on 4.5pt off 4.5pt}]  (196,162.05) -- (196,304.93) ;
\draw [line width=2.25]    (165,304.93) -- (196,304.93) ;
\draw [shift={(196,304.93)}, rotate = 180] [color={rgb, 255:red, 0; green, 0; blue, 0 }  ][line width=2.25]    (0,7.83) -- (0,-7.83)   ;
\draw [shift={(165,304.93)}, rotate = 180] [color={rgb, 255:red, 0; green, 0; blue, 0 }  ][line width=2.25]    (0,7.83) -- (0,-7.83)   ;
\draw  [dash pattern={on 0.84pt off 2.51pt}]  (178,338.06) -- (178,351.71) ;
\draw [color={rgb, 255:red, 155; green, 155; blue, 155 }  ,draw opacity=1 ][line width=1.5]    (116,151.59) -- (116,198.46) ;
\draw [shift={(116,198.46)}, rotate = 270] [color={rgb, 255:red, 155; green, 155; blue, 155 }  ,draw opacity=1 ][line width=1.5]    (0,6.71) -- (0,-6.71)   ;
\draw [shift={(116,151.59)}, rotate = 270] [color={rgb, 255:red, 155; green, 155; blue, 155 }  ,draw opacity=1 ][line width=1.5]    (0,6.71) -- (0,-6.71)   ;
\draw [color={rgb, 255:red, 155; green, 155; blue, 155 }  ,draw opacity=1 ][line width=1.5]    (102,189.81) -- (102.5,226) ;
\draw [shift={(102.5,226)}, rotate = 269.21] [color={rgb, 255:red, 155; green, 155; blue, 155 }  ,draw opacity=1 ][line width=1.5]    (0,6.71) -- (0,-6.71)   ;
\draw [shift={(102,189.81)}, rotate = 269.21] [color={rgb, 255:red, 155; green, 155; blue, 155 }  ,draw opacity=1 ][line width=1.5]    (0,6.71) -- (0,-6.71)   ;
\draw [color={rgb, 255:red, 155; green, 155; blue, 155 }  ,draw opacity=1 ][line width=1.5]    (115,50.57) -- (115,97.44) ;
\draw [shift={(115,97.44)}, rotate = 270] [color={rgb, 255:red, 155; green, 155; blue, 155 }  ,draw opacity=1 ][line width=1.5]    (0,6.71) -- (0,-6.71)   ;
\draw [shift={(115,50.57)}, rotate = 270] [color={rgb, 255:red, 155; green, 155; blue, 155 }  ,draw opacity=1 ][line width=1.5]    (0,6.71) -- (0,-6.71)   ;
\draw   (89,151.51) .. controls (87.63,151.51) and (86.94,152.2) .. (86.94,153.58) -- (86.94,153.58) .. controls (86.94,155.54) and (86.25,156.52) .. (84.88,156.52) .. controls (86.25,156.52) and (86.94,157.5) .. (86.94,159.46)(86.94,158.58) -- (86.94,159.46) .. controls (86.94,160.84) and (87.63,161.53) .. (89,161.53) ;

\draw (304,313.63) node [anchor=north west][inner sep=0.75pt]    {$\mathbf{J}_{s}$};
\draw (620,118) node [anchor=north west][inner sep=0.75pt]    {$E_{n_{s}}(\mathbf{E}_{s})$};
\draw (322,229.9) node [anchor=north west][inner sep=0.75pt]    {$\mathbf{E}_{s}$};
\draw (66,110.5) node [anchor=north west][inner sep=0.75pt]    {$J_{s,n_{s}}^{( 2)}$};
\draw (571,240) node [anchor=north west][inner sep=0.75pt]  [color={rgb, 255:red, 0; green, 0; blue, 0 }  ,opacity=1 ]  {$\salt{\mathbf{J}}(\mathbf{E}_{s}) \cap \mathbf{J}_{s}$};
\draw (143,98) node [anchor=north west][inner sep=0.75pt]    {$\overline{\rho }_{s}$};
\draw (260,92) node [anchor=north west][inner sep=0.75pt]    {$n_{s} \alpha $};
\draw (162,351.52) node [anchor=north west][inner sep=0.75pt]    {$\theta _{s,n_{s} ,-}$};
\draw (466,352.43) node [anchor=north west][inner sep=0.75pt]    {$\theta _{s,n_{s} ,+}$};
\draw (465,312.85) node [anchor=north west][inner sep=0.75pt]    {$I_{s,n_{s} ,+}$};
\draw (164,313.85) node [anchor=north west][inner sep=0.75pt]    {$I_{s,n_{s} ,-}$};
\draw (65,193.15) node [anchor=north west][inner sep=0.75pt]  [color={rgb, 255:red, 128; green, 128; blue, 128 }  ,opacity=1 ]  {$J_{s,i}^{( 1)}$};
\draw (55,146.56) node [anchor=north west][inner sep=0.75pt]    {$3\rho _{s}$};

\end{tikzpicture}

    \caption{A cartoon illustration of the covering constructed in Lemma \ref{l:constructionsummary}.}
    \label{f:covering}
\end{figure}
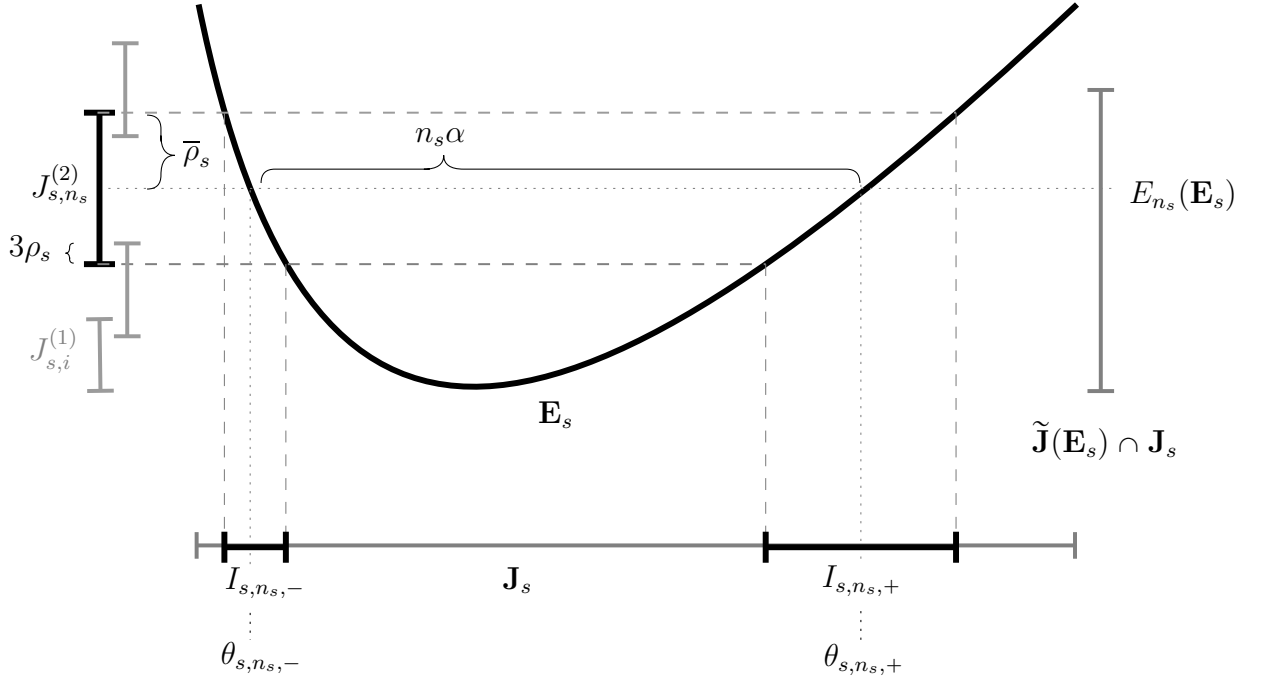

Fix an interval $J_{s,i}^{(j)}$ ($j \in \{1,2\}$), 
and note that  %
$|I_{s,i,\pm}^{(j)}| \leq \rho_{s-1}/24D_0$ by Lemma \ref{l:morsefnbd}, the definition of $\bar{\rho}_s$, and item 2 of Lemma \ref{l:constructionsummary}.  
Consider all `ancestors' $\mathbf E_{k} : \mathbf I_k \to \mathbf J_k$ ($0\leq k \leq s-1$) of $\mathbf E_s$. 
For each such $k$, we define two sets $\mathcal S_{k,\pm}$ as follows:
$$\mathcal S_{k,\pm} := \bigcup_{\theta \in I_{s,i,\pm}^{(j)}} \mathcal{S}_{k}(\theta, \mathbf{E}_s(\theta)).$$
The integers in $\mathcal S_{k,\pm}$ are well-separated uniformly in $\theta$ at each scale $0\leq k\leq s-1$ in the following sense:
\begin{lem}\label{l:Ssep} Let $0 \leq k \leq s-1$, and suppose $m \in \mathcal S_{k,\pm}$. 
	\begin{enumerate}
		\item If $j_{k+1} = 1$, then for any $n \neq m$ with $|n-m| \leq 8M_{k+1}^{(1)}$, $n \notin \mathcal S_{k,\pm}$.
		\item If $j_{k+1} = 2$ (and thus $\mathbf{E}_{k+1} = \mathbf{E}_{k+1,n_{k},\bullet}$ for some $M_k^{(2)} < |n_{k}| \leq 8 M_{k+1}^{(1)}$), then for some $p \in \{m,m-n_{k}\}$ and for any $n \notin \{p,p+ n_{k}\}$ with $|n-p| \leq 8M_{k+1}^{(2)}$, $n \notin \mathcal S_{k,\pm}$.
	\end{enumerate}
\end{lem}
\begin{proof}
	Since $m \in \mathcal S_{k,\pm}$, there is some $\theta_* \in I_{s,i,\pm}^{(j)}$ such that $m \in \mathcal{S}_{k}(\theta_*,\mathbf{E}_s(\theta_*))$; note that $J_{s,i}^{(j)} \subset \salt{\mathbf J}(\mathbf{E}_s) \subset \salt{\mathbf J}(\mathbf{E}_k)$, so $E_* = \mathbf{E}_s(\theta_*)$ satisfies the assumptions of the inductive hypothesis \ref{pr:resorbs}.
	As noted above, 
	we also have that $|I_{s,i,\pm}^{(j)}| \leq \rho_{s-1}/24D_0 \leq \rho_{k}/24D_0$.
	
	We have two cases: 
	\begin{enumerate}
		\item \textit{$j_{k+1} = 1$:} 
		
		Take $p = m$.  By the length bound on $I_{s,i,\pm}^{(j)}$, we have $$I_{s,i,\pm}^{(j)}+p\alpha \subset B_{\rho_{k}/24D_0}(\theta_*+ p\alpha) \subset \mathbf{I}_s.$$  
		
		\item \text{$j_{k+1} = 2$:} By Hypothesis \ref{pr:resorbs} and the length bound on $I_{s,i,\pm}^{(j)}$, for some $p \in \{m,m-n_{k+1}\}$, we have $$I_{s,i,\pm}^{(j)} + p\alpha \subset B_{\rho_{k}/24D_0}(\theta_*+ p\alpha) \subset \mathbf{I}_{k+1,\,\cup} \subset B_{\rho_k/8D_0}(\theta_{k+1,n_{k+1},-}).$$
	\end{enumerate}
	In both cases, by Lemma \ref{l:constructionsummary}, for any $\theta \in I_{s,i,\pm}^{(j)}$ and $n \neq p$ with $|n-p|\leq 8M_{k+1}^{(j_{k+1})}$ such that $\theta+n\alpha \in \mathbf{I}_k$, we have
	\begin{align*}
		|\mathbf E_k(\theta+n\alpha) - \mathbf E_s(\theta)| \geq &|\mathbf E_k(\theta+n\alpha)-\mathbf E_k(\theta + p\alpha)| - |\mathbf E_k(\theta+p\alpha)-\mathbf E_k(\theta_*+p\alpha)| \\&- |\mathbf E_k(\theta_*+p\alpha)-\mathbf E_s(\theta_*)| - |\mathbf E_s(\theta_*)-\mathbf E_s(\theta)| \\
		&\geq 3\rho_k - \frac{1}{24} \rho_{s-1} - \frac{25}{24}\rho_k - \frac{1}{24}\rho_{s-1} \geq \frac{25}{24}\rho_k,
	\end{align*}
	and so $n \notin \mathcal{S}_{k,\pm}$, as claimed.	
\end{proof}

This separation will, in turn, allow us to construct even integers $L_{s+1}$ so that $\pm L_{s+1}/2$ are $s$-regular:

\begin{lem}\label{l:Lconst}
	There is an even integer $L_{s+1}$ with $M_{s+1}^{(j)}/2 \leq L_{s+1} \leq M_{s+1}^{(j)}$ so that, for all $\theta_* \in I_{s,i,\pm}^{(j)}$ for all $E_* \in B_{\rho_s}(\mathbf{E}_s(\theta_*))$, $\pm L_{s+1}/2$ and $\pm(L_{s+1}/2+1)$ are $s$-regular relative to $(\theta_*, E_*)$.  %
\end{lem}

\begin{proof}
	Let $a_s := M_{s+1}^{(j)}/2$.  We define a nonincreasing sequence of integers $a_s, a_{s-1}, \dots, a_0$ as follows.  For each $1 \leq k \leq s$, we have by the previous lemma that $$\# \left\{ a \in [a_k - 8L_k, a_k] : \dist(a, \mathcal{S}_{k-1,\pm}) < \frac{4}{5}L_k \right\} < \frac{9}{5}L_k,$$
	where we have crucially used here that $5M_{k-1}^{(2)} < M_k^{(1)}/2$.  The same counting argument holds replacing $[a_k - 8L_k, a_k]$ by $[-a_k, -a_k + 8L_k]$; since $4(9/5) < 8$, we can find some $a_{k-1} \in [a_k-8L_k,a_k]$ such that $$\dist(\pm a_{k-1}, \mathcal{S}_{k-1,\pm}) \geq \frac{4}{5}L_k$$
	for all four choices of signs.
	
	We define $L_{s+1} = 2a_0$.  By construction, we have $$M^{(j)}_{s+1} \geq L_{s+1} \geq M^{(j)}_{s+1} - 2\sum_{k=1}^s 8L_k \geq M^{(j)}_{s+1}/2,$$ and $$\dist(\pm L_{s+1}/2, \mathcal S_{k-1,\pm}) \geq \frac45L_k - 8\sum_{i=1}^{k-1} L_i > \frac34L_k+1$$ for all $1\leq k\leq s$ and all four choices of signs.
\end{proof}

\subsubsection{Decay of resonant eigenpairs}

For the fixed interval $J_{s,i}^{(j)}$ we denote by $$\Lambda_{s+1} := [-L_{s+1}/2, L_{s+1}/2]$$ and recall that $$\ell_s = L_s^{5/6} \geq 16L_s^{2/3}|\log\varepsilon|^2 = 16|\log\varepsilon||\log\rho_s|.$$  From our inductive assumptions, we immediately find Green's function decay away from $\{0\}$ (and the double-resonant site $\{n_s\}$, if it exists):

\begin{lem}\label{c:GFdecInd}
	Fix $\theta_*\in \mathbf{I}_s$ and find $\Lambda_{s+1}$ as above.  Then we have the following:
	\begin{enumerate}
		\item If $\theta_*\in I_{s,i,\pm}^{(1)}$, then for $|E - \mathbf{E}_s(\theta_*)| < \frac{3}{2}\rho_s$ and $\theta \in B_{\rho_s/8D_0}(\theta_*)$, $$\log|R_{\theta,E}^{\Lambda_{s+1} \setminus \Lambda_s}(m,n)| \leq -\gamma_s|m-n|, \quad |m-n| \geq \ell_s$$ and $$\|R_{\theta,E}^{\Lambda_{s+1}\setminus \Lambda_s}\| \leq 4 \rho_s^{-1}.$$
		Furthermore, there are intervals $\Lambda_l \supset [-\frac78L_s, -\frac18L_s]$ and $\Lambda_r \supset [\frac18L_s, \frac78L_s]$ such that for $|E - \mathbf{E}_s(\theta_*)| < \frac{3}{2}\rho_s$ and $\theta \in B_{\rho_{s-1}/8D_0}(\theta_*)$,
		\begin{align*}
			\log|R^{\Lambda_{l/r}}_{\theta,E}(m,n)| \leq -\gamma_{s-1}|m-n|, \quad |m-n| \geq \ell_{s-1}.
		\end{align*}
		\item If $\theta_* = \theta_{s,n_s,-}$, then for $|E - \mathbf{E}_s(\theta_*)| < \frac{3}{2}\rho_{s-1}$ and $\theta \in B_{\rho_{s-1}/8D_0}(\theta_*)$, $$\log|R_{\theta,E}^{\Lambda_{s+1} \setminus (\Lambda_s \cup \Lambda_s + n_s)}(m,n)| \leq -\gamma_{s-1}|m-n|, \quad |m-n| \geq \ell_{s-1}$$ and $$\|R_{\theta,E}^{\Lambda_{s+1}\setminus (\Lambda_s \cup \Lambda_s + n_s)}\| \leq 4 \rho_{s-1}^{-1}.$$
		Furthermore, letting $c_l = \min\{0,n_s\}$ and $c_r = \max\{0,n_s\}$, there are intervals $\Lambda_l \supset [c_l - \frac78L_s,c_l-\frac18L_s]$, $\Lambda_c \supset [c_l+\frac18L_s,c_r-\frac18L_s]$, and $\Lambda_r \supset [c_r+\frac18L_s,c_r+\frac78L_s]$ such that, for $|E - \mathbf{E}_s(\theta_*)| < \frac{3}{2}\rho_{s-1}$ and $\theta \in B_{\rho_{s-1}/8D_0}(\theta_*)$,
		\begin{align*}
			\log|R^{\Lambda_{l/c/r}}_{\theta,E}(m,n)| \leq -\gamma_{s-1}|m-n|, \quad |m-n| \geq \ell_{s-1}.
		\end{align*}
	\end{enumerate}
	In either case, the %
	intervals $\Lambda_{s} \cap \Lambda_{l/c/r}$ %
	also satisfy the Green's function decay property for $(\ell_{s-1},\gamma_{s-1})$, 
	as do %
	the intervals $(\Lambda_{s} + n_s) \cap \Lambda_{l/c/r}$.
\end{lem}
\begin{proof}
	Noting the $s$-regularity relative to $(\theta_*, E_*)$ for any $E_* \in B_{\rho_s}(\mathbf{E}_s(\theta_*))$, these statements are precisely the nonresonance hypothesis  \ref{prsub:NR} applied to various intervals.  In particular, the maximal connected components of $\Lambda_{s+1} \setminus \Lambda_s$ are inductively $s$-nonresonant with $s$-directionally regular endpoints. 	Similarly, the maximal connected components of  $\Lambda_{s+1} \setminus (\Lambda_{s} \cup \Lambda_{s} + n_s)$ are $s-1$-nonresonant, as are the maximal components of $[-L_s,L_s] \setminus \Lambda_{s-1}$ (in case $\mathbf E_s \in \mathcal E^{(1)}(\mathbf E_{s-1})$) or $[-L_s,L_s] \setminus (\Lambda_{s-1} \cup \Lambda_{s-1} + n_{s-1})$ (in case $\mathbf E_s \in \mathcal E^{(2)}(\mathbf E_{s-1})$).

	The existence of the subintervals $\Lambda_{l/c/r}$ follows from the nonresonance hypothesis after noting that $$\frac18L_s - M_{s-1}^{(2)} - 1 \geq \frac{1}{16} (M_{s-1}^{(2)})^2 - M_{s-1}^{(2)} - 1 \gg 5L_{s-1}$$ for the left and right intervals, and $$c_r-c_l-2(1+M_{s-1}^{(2)}) \geq L_s - 2(1+M_{s-1}^{(2)}) \gg 5L_{s-1}$$ for the center interval in the double-resonant case.

	Finally, the decay estimates on the intervals $\Lambda_s \cap \Lambda_{l/c/r}$ likewise follow from the inductive nonresonance hypothesis and the $s-1$-directional regularity of $L_{s}/2$.  The decay estimates on $(\Lambda_s + n_s) \cap \Lambda_{l/c/r}$ follow from the corresponding regularities of $n_s \pm L_s/2$, which follow from the observation that $H^{\Lambda_s}(\theta_* + n_s\alpha) = H^{\Lambda_s+ n_s}(\theta_*)$.

\end{proof}

Noting that $$\frac{3}{8}L_s > L_s^{5/6} = \ell_s \gg \ell_{s-1},$$ we apply the above decay to conclude stable eigenvector decay for resonant eigenvectors of $H^{\Lambda_{s+1}}$:

\begin{lem} \label{l:asm25}
	Suppose $\theta_* \in \mathbf{I}_s$, and let $E_* = \mathbf{E}_s(\theta_*)$.
	\begin{enumerate}
		\item If $\theta_* \in I_{s,i,\pm}^{(1)}$, let $\rho = \rho_s$ and let $\mathcal{P}$ be the partition of $\Lambda_{s+1}$ into $\Lambda_s$ and its complement.
		\item If $\theta_* = \theta_{s,n_s,-}$, let $\rho = \rho_{s-1}$ and let $\mathcal{P}$ be the partition of $\Lambda_{s+1}$ into $\Lambda_s \cup \Lambda_s+n_s$ and its complement.
	\end{enumerate}
	Then, for $\theta \in B_{\rho/8D_0}(\theta_*)$, for any eigenpair $(E,\psi)$ of $H^{\Lambda_{s+1}}(\theta)$ with $|E - E_*| < \frac32\rho$, the unit eigenvector $\psi$ has $\mathcal{P}$-boundary values no larger than $\delta_s/\varepsilon$:
	$$\|\Gamma_\mathcal{P}^{\Lambda_{s+1}}\psi\| \leq 4 \delta_s.$$
	Furthermore, the resonant unit eigenvector $\psi_s$ of $H^{\Lambda_s}$ corresponding to $\mathbf{E}_s$  (as well as the associated unit eigenvector of $H^{\Lambda_s+n_s}$ in case 2) has stably-small $\mathcal{P}$-boundary values:
	\begin{align*}
		\|\Gamma_\mathcal{P}^{\Lambda_{s+1}}\psi_s\| &\leq 2\delta_s \\
		\|\Gamma_\mathcal{P}^{\Lambda_{s+1}}(\partial_\theta\psi_s)\| &\leq 24D_0\frac{\delta_s}{\rho}
	\end{align*}
\end{lem}

\begin{proof}
	The estimate on $\|\Gamma_\mathcal{P}^{\Lambda_{s+1}}\psi\|$ follows in all cases from Lemma \ref{c:GFdecInd} and the Poisson formula \eqref{eq:poisson}; for example, in the case $\theta_* \in I_s^{(1)}$, denoting $\Lambda_l = [a_l,b_l]$, one has %
	\begin{align*}
		|\psi(-L_s/2)| &\leq \varepsilon \left(|R_{\theta,E}^{\Lambda_l}(-L_s/2,a_l)| + |R_{\theta,E}^{\Lambda_l}(-L_s/2,b_l)|\right) \\
		&\leq 2\varepsilon e^{-3\gamma_{s-1}L_s/8} \leq \delta_s.
	\end{align*}
	The other cases follow analogously.
	
	The estimate on $\|\Gamma_\mathcal{P}^{\Lambda_{s+1}}\psi_s\|$ is made similarly, instead applying the decay on the intervals $\Lambda_s \cap \Lambda_{l/c/r}$.
	
	Let $\chi_s$ denote the linear projection onto the coordinates in the interval $\Lambda_s$, and $\salt{\chi}_s$ denote the linear projection onto the interval $$\{m : |m| \leq L_s/4\} \subset \Lambda_s.$$
	Let $\Gamma_s := \Gamma_\mathcal{P}^{\Lambda_{s+1}}\chi_s$, and note that $$\Gamma_s(\partial_\theta\psi_s) = \Gamma_\mathcal{P}^{\Lambda_{s+1}}(\partial_\theta\psi_s).$$  	Recall that $$\partial_\theta\psi_s = -R_\perp(H^{\Lambda_s}, \mathbf{E}_s) V'\psi_s;$$ we break our estimate of $\|\Gamma_s(\partial_\theta\psi_s)\|$ into pieces corresponding to $\salt{\chi}_s$ and its relative complement $(\chi_s - \salt{\chi}_s)$.

	By again applying the Poisson formula on the intervals $\Lambda_s \cap \Lambda_{l/c/r}$ from Lemma \ref{c:GFdecInd}, we have
	\begin{align*}
		\|(\chi_s - \salt{\chi}_s)V'\psi_s\|^2 &\leq \sum_{L_s/4 \leq |m| \leq L_s/2} D_0^2|\psi_s(m)|^2 \\
		&\leq 8D_0^2\sum_{j > L_s/8} \varepsilon^2e^{-2\gamma_{s-1}j} \\
		&\leq (3D_0\delta_s)^2.
	\end{align*}
	Now let $\sprev{\mathcal{P}}$ be the partition of $\Lambda_s$ into $\Lambda_s \cap (\Lambda_l \cup \Lambda_c \cup \Lambda_r)$ and its complement, and recall that $E_* = \mathbf{E}_s(\theta_*)$.  Letting $R_\perp = R_\perp( H^{\Lambda_s}, \mathbf{E}_s), \sprev{R} = R_{\sprev{\mathcal{P}}}^{\Lambda_s}(E_*),$ and $P_s = \psi_s\psi_s^{\top}$, we expand $$\Gamma_sR_\perp\salt{\chi}_s = \Gamma_s\sprev{R}(I - P_s - \Gamma_{\sprev{\mathcal{P}}}^{\Lambda_s}R_\perp)\salt{\chi}_s$$ by the resolvent identity.  Since the intervals in  $\Lambda_s \cap (\Lambda_l \cup \Lambda_c \cup \Lambda_r)$ are $s-1$-nonresonant, we get $$\max\{\|\Gamma_s\sprev{R}\salt{\chi}_s\|, \|\Gamma_s\sprev{R}\Gamma_{\sprev{\mathcal{P}}}^{\Lambda_s}\|\} \leq 2\varepsilon e^{-\gamma_{s-1}L_s/4}.$$  By expanding $P_s = \salt{\chi}_sP_s + (\chi_s - \salt{\chi}_s)P_s$, we likewise estimate
	\begin{align*}
		\|\Gamma_s\sprev{R}P_s\salt{\chi}_{s}\| &\leq \|\Gamma_s \sprev{R}\salt{\chi}_s\|\|P_s\salt{\chi}_s\| + \|\Gamma_s\sprev{R}\|\|(\chi_s - \salt{\chi}_s)P_s\|\|\salt{\chi}_s\| \\
		&\leq 2\varepsilon e^{-\gamma_{s-1}L_s/4}+\frac{8\varepsilon}{\rho_{s-1}}\cdot 3\varepsilon e^{-\gamma_{s-1}L_s/8} \\
		&\leq 6\frac{\delta_s}{\rho_{s-1}}.
	\end{align*}
	It remains to estimate $\|R_\perp\|$, but this follows inductively.  Specifically, by Hypothesis \ref{pr:resorbs}, if $\theta_* = \theta_{s,n_s,-}$, it must be that $j_s = 1$ and so $\|R_\perp\| \leq 4\rho_{s-1}^{-1}$ by \eqref{eq:rperpbd}.  
	If $\theta_* \in I_s^{(1)}$, we have the estimate $\|R_\perp\| \leq 4 \rho_s^{-1}$, either by \eqref{eq:rperpbd} if $j_s=1$, or by \eqref{eq:rperpperpbd} and Proposition \ref{pr:unifmLocalSep} if $j_s=2$, since the separation guaranteed by Proposition \ref{pr:unifmLocalSep} is bigger than $2\rho_s$.  In either case, $\|R_\perp\| \leq 4\rho^{-1}$ by our definition of $\rho$.  Thus, we compute
	\begin{align*}
		\|\Gamma_s(\partial_\theta\psi_s)\| &= \|\Gamma_sR_\perp V'\psi_s\| \\
		&\leq \|\Gamma_sR_\perp\| \|(\chi_s-\salt{\chi}_s)V'\psi_s\| + \|\Gamma_sR_\perp\salt{\chi}_s\|\|V'\psi_s\| \\
		&\leq \|\Gamma_sR_\perp\| \|(\chi_s-\salt{\chi}_s)V'\psi_s\| + D_0\left(\|\Gamma_s\sprev{R}\salt{\chi}_s\| + \|\Gamma_s\sprev{R}P_s\salt{\chi}_s\| + \|\Gamma_s\sprev{R}\Gamma_{\sprev{\mathcal{P}}}^{\Lambda_s}\|\|R_\perp\|\|\salt{\chi}_s\|\right) \\
		&\leq 8\varepsilon\rho^{-1}\cdot 3D_0\delta_s + D_0\left(2\delta_s + 6\delta_s\rho_{s-1}^{-1} + 8\delta_s\rho^{-1}\right) \\
		&\leq 24D_0\frac{\delta_s}{\rho},
	\end{align*}
	as claimed.  %
	The estimate for the unit eigenvector of $H^{\Lambda_s + n_s}$ follows analogously using the intervals $(\Lambda_s + n_s) \cap \Lambda_{l/c/r}$ instead of the intervals $\Lambda_s \cap \Lambda_{l/c/r}$.
\end{proof}

\subsubsection{Constructing the descendants $\mathbf{E}_{s+1}$}

With these estimates, we are ready to invoke the content of Sections 2 and 3:
\begin{lem}\label{l:res}
	We have the following:
	\begin{enumerate}
		\item (Simple resonance): If $\theta_* \in I_{s}^{(1)}$ and $E_* = \mathbf{E}_s(\theta_*)$, then $H^{\Lambda_{s+1}}(\theta_*)$ satisfies Assumption \ref{as:SR} with the following assignments:
		\begin{align*} 
			\snext\Lambda = \Lambda_{s+1}&, \; \Lambda = \Lambda_s, \\
			\rho = \rho_s, \; 
			\delta = \delta_s&, \;
			\gamma = \gamma_s, \; \ell = \ell_s.
		\end{align*}
		\item (Double resonance): If $\theta_{**} = \theta_{s,n_s,-}$ and $E_{**} = \mathbf{E}_s(\theta_*)$, then $H^{\Lambda_{s+1}}(\theta_{**})$ satisfies Assumption \ref{as:DR} with the following assignments:
		\begin{align*} 
			\snext\Lambda = \Lambda_{s+1}, \; {\Lambda}_- &= \Lambda_s, \; {\Lambda}_+ = \Lambda_s + n_s, \\ \sprev{\rho} = \frac23\rho_{s-1},\; \delta = \delta_s&, \; \gamma = \gamma_s, \; \ell=\ell_s, \; \nu = {\nu}_s.
		\end{align*}
	\end{enumerate}
\end{lem}
\begin{proof}
	We verify the relevant assumptions in each case:
	\begin{enumerate}
		\item The values $\delta$, $\rho$, and $\gamma$ satisfy \eqref{eq:eps1} and \eqref{eq:gam1} by definition.  Taking ${\psi} = \psi_s$, item 1 of Assumption \ref{as:SR} is immediate by definition of $E_*$. Items 2 and 5 follow immediately from Lemma \ref{l:asm25}, and item 3 follows immediately from Lemma \ref{c:GFdecInd}.  Finally, item 4 follows from the inductive bound $\|R^{\Lambda_s}_{\perp,\theta,E}\| \leq 4\rho_s^{-1}$ (following from either Proposition \ref{pr:AL1} or \ref{pr:AL2}, depending on the value of $j_s$) and the bound $\|R^{\Lambda_{s+1}\setminus\Lambda_s}_{\theta,E}\| \leq 4\rho_s^{-1}$ from Lemma \ref{c:GFdecInd}.
		\item Again, $\delta$, $\sprev{\rho}$, and $\gamma$ satisfy \eqref{eq:eps2}, \eqref{eq:gam2}, and \eqref{eq:rho2} by definition.  As in case 1, item 1 of Assumption \ref{as:DR} is immediate by definition of $E_{**}$, items 2 and 5 follow immediately from Lemma \ref{l:asm25}, and item 3 follows immediately from Lemma \ref{c:GFdecInd}.  Since $n_s \in \mathcal{N}_s \neq \emptyset$, we must have $j_s = 1$, and so inductively $\|R_{\perp,\theta,E}^{\Lambda_s}\| \leq 4\rho_{s-1}^{-1} < 4\sprev{\rho}^{-1}$ for $|E - E_{**}| \leq \rho_{s-1} = \frac32\sprev{\rho}$ (and similarly for $\|R_{\perp,\theta,E}^{\Lambda_s+n_s}\|$) by Proposition \ref{pr:AL1}.  Item 4 follows from this observation and the bound 	$\|R^{\Lambda_{s+1}\setminus(\Lambda_s \cup \Lambda_s+n)}_{\theta,E}\| \leq 4\rho_{s-1}^{-1} < 4\sprev{\rho}^{-1}$ from Lemma \ref{c:GFdecInd}.  Finally, item 6 follows from Lemma \ref{l:constructionsummary}. \qedhere
	\end{enumerate}
\end{proof}

As a consequence of the above, for each phase interval $I_s \in \mathcal{I}$, we have found an interval $\Lambda_{s+1} = \Lambda_{s+1}(I_s)$ such that any resonant Rellich function is Morse with at most one critical point and whose eigenvector is localized.  Specifically, if $I_s = I_{s,i}^{(1)}$ we have found a unique Rellich function $$\mathbf{E}_{s+1,i}^{(1)}: I_{s,i}^{(1)} \to J_{s+1,i}^{(1)}$$ of $H^{\Lambda_{s+1}}$, 
where we define $J_{s+1,i}^{(1)}$ to be the image of $I_{s,i}^{(1)}$ under $\mathbf E_{s+1,i}^{(1)}$.  
Similarly, for each $I_{s,n,-}^{(2)}$, $n \in \mathcal{N}_s$, we have found two Rellich functions $$\mathbf{E}_{s+1,n,\bullet}^{(2)} : I_{s,n,-}^{(2)} \to J_{s+1,n}^{(2)}, \quad \bullet \in \{\vee,\wedge\}$$ of $H^{\Lambda_{s+1}}$,
where we define $J_{s+1,n}^{(2)}$ to be the smallest interval containing the images of $I_{s,n,-}^{(2)}$ under both Rellich functions.  
We note that, since $\varepsilon < \frac13$ and $\ell_s,L_s<M_{s+1}^{(2)}$, we may choose $\snext\sigma=\sigma_{s+1}$ in Theorem \ref{t:DRRelFns}, as this choice satisfies \eqref{eq:sepconst}.
To finish constructing the children of $\mathbf{E}_s$, we must modify the domains of these new Rellich functions so as to again satisfy Assumption \ref{as:C2fn}.

In the simple-resonant case, by Lemma \ref{l:constructionsummary} and Proposition \ref{pr:interscaleapprox1}, the function $\mathbf{E}_{s+1,i}^{(1)}$ satisfies the assumptions of Lemma \ref{l:Asm3StabSR} relative to the function $\mathbf{E}_{s,i}^{(1)}$ with $\delta = \delta_s$ and $\rho = \rho_s$; thus, we may find an interval $\mathbf{I}_{s+1,i}^{(1)} \subset I^{(1)}_{s,i}$ and its image $\mathbf{J}_{s+1,i}^{(1)} := \mathbf{E}_{s+1,i}^{(1)}(\mathbf{I}_{s+1,i}^{(1)})$ with $$|I_{s,i}^{(1)} \setminus \mathbf{I}_{s+1,i}^{(1)}| \leq \frac{32\delta_s}{\bar\nu_s} \ll \frac{\rho_s}{24D_0}.$$
We abuse notation and henceforth denote by $\mathbf{E}_{s+1,i}^{(1)}$ the restriction of the  Rellich function of $H^{\Lambda_{s+1}}$ to the interval $\mathbf{I}_{s+1,i}^{(1)}$.

The double-resonant case requires a bit more care.  If $I_s = I_{s,n_s,-}^{(2)}$, $n_s \in \mathcal{N}_s$, we recall the regions $$J^{(2)}_{s,n_s} := \overline{B}_{\bar{\rho}_{s}}(E_{n_s}), \quad I_{s,n_s,\pm}^{(2)} := \mathbf{E}_{s,\pm}^{-1}(J^{(2)}_{s,n_s}).$$  Note that one has $$B_{\frac{\bar{\rho}_{s}}{D_0}}(\theta_{s,n_s,\pm}) \subset I_{s,n_s,\pm}^{(2)} \subset B_{\frac{\bar{\rho}_s}{\nu_s}}(\theta_{s,n_s,\pm})$$ since $\nu_s \leq |\partial_\theta \mathbf{E}_s| \leq D_0$ on $I_{s,n_s,\pm}^{(2)}$.

Denote by
\begin{align*}
	\tilde{\mathbf{E}}_{s,n_s,\vee}^{(2)}(\theta) &:= \max\{\mathbf{E}_{s,-}(\theta), \mathbf{E}_{s,+}(\theta+n_s\alpha)\}, \quad \theta \in \tilde{I}_{s,n_s,\vee}^{(2)}:=[\inf I_{s,n_s,-}^{(2)}, \sup I_{s,n_s,+}^{(2)}-n_s\alpha], \\
	\tilde{\mathbf{E}}_{s,n_s,\wedge}^{(2)}(\theta) &:= \min\{\mathbf{E}_{s,-}(\theta), \mathbf{E}_{s,+}(\theta+n_s\alpha)\}, \quad \theta \in \tilde{I}_{s,n_s,\wedge}^{(2)}:=[\inf I_{s,n_s,+}^{(2)} - n_s\alpha, \sup I_{s,n_s,-}^{(2)}].
\end{align*}
We denote also $$I_{s,n_s,\cap}^{(2)}:= I_{s,n_s,-}^{(2)} \cap (I_{s,n_s,+}^{(2)}-n_s\alpha) \subset \tilde{I}_{s,n_s,\bullet}^{(2)}, \quad \bullet \in \{\vee,\wedge\}.$$
Since $\theta_{s,n_s,+} = \theta_{s,n_s,-}+n_s\alpha$, we have that $$\overline{B}_{\frac{\bar{\rho}_s}{D_0}}(\theta_{s,n_s,-}) \subset I_{s,n_s,\cap}^{(2)}.$$  By construction, $$| \tilde{\mathbf{E}}_{s,n_s,\bullet}^{(2)}(\tilde{I}_{s,n_s,\bullet}^{(2)})| = \bar{\rho}_s, \quad \bullet \in \{\vee,\wedge\},$$ and thus, by Lemma \ref{l:morsefnbd},
\begin{align*}
	|\{\theta \in \tilde{I}_{s,n_s,\bullet}^{(2)} : \tilde{\mathbf{E}}_{s,n_s,\bullet}^{(2)}(\theta) = \mathbf{E}_s(\theta+p\alpha)\}| \geq \frac{\bar{\rho}_s}{D_0}, \quad \bullet \in \{\vee,\wedge\},\; p \in \{0,n_s\}.
\end{align*}
The functions $\tilde{\mathbf{E}}_{s,n_s,\bullet}^{(2)}$ satisfy Assumption 3; by Lemma \ref{l:constructionsummary} and Theorem \ref{t:DRRelFns}, the Rellich functions $\mathbf{E}_{s+1,n_s,\bullet}^{(2)}$ satisfy the assumptions of Lemma \ref{l:Asm3StabDR} relative to the functions $\tilde{\mathbf{E}}_{s,n_s,\bullet}^{(2)}$ with $\delta = C\delta_s$; thus, we may find an interval $\mathbf{I}_{s+1,n_s,\bullet}^{(2)} \subset \tilde{I}_{s,n_s,\bullet}^{(2)}$ and its image $\mathbf{J}_{s+1,n_s,\bullet}^{(2)} := \mathbf{E}_{s+1,n_s,\bullet}^{(2)}(\mathbf{I}_{s+1,n_s,\bullet}^{(2)})$ with $$|\tilde{I}_{s,n_s,\bullet}^{(2)} \setminus \mathbf{I}_{s+1,n_s,\bullet}^{(2)}| \leq \frac{32C\delta_s}{\bar\nu_s} \ll \frac{\rho_s}{24D_0}.$$  
In particular, we note that $|\mathbf I^{(2)}_{s+1,n_s,\cap}| \geq \frac{\overline\rho_s}{D_0}-\frac{\rho_s}{24D_0} \gg \frac{5\rho_s}{4\nu_s}$, where $\mathbf I^{(2)}_{s+1,n_s,\cap}:=\mathbf I^{(2)}_{s+1,n_s,\vee} \cap \mathbf I^{(2)}_{s+1,n_s,\wedge}$, so we may choose $\eta = \frac{5\rho_s}{4\nu_s}$ and $\rho=\rho_s$ in Theorem \ref{t:DRRelFns}, with $B_\eta(\theta_{s,n_s,-}) \subset \mathbf I^{(2)}_{s+1,n_s,\cap}$.
We abuse notation and henceforth denote by $\mathbf{E}_{s+1,n_s,\bullet}^{(2)}$ the restriction of the corresponding Rellich function of $H^{\Lambda_{s+1}}$ to the interval $\mathbf{I}_{s+1,n_s,\bullet}^{(2)}$, $\bullet \in \{\vee,\wedge\}$.

Denote the collection of all descendants of $\mathbf{E}_s$ constructed above by $\mathcal{E}(\mathbf{E}_s)$.
\begin{lem}\label{l:jrelations}
	Let $\mathbf{E}_{s+1} \in \mathcal{E}(\mathbf{E}_s)$ be constructed as above.
	\begin{enumerate}
		\item $\mathbf{E}_{s+1} : \mathbf{I}_{s+1} \to \mathbf{J}_{s+1}$ satisfies Assumption \ref{as:C2fn} with Morse constants 
			\begin{align*}
				d_{s+1} &= \nu_s/12, \\
				D_{s+1} &= 2D_0(1+D_0\sigma_{s+1}^{-1})
			\end{align*}
			and boundary derivative constant
			\begin{align*}
				\nu_{s+1} &= \min\left\{\frac{d_{s+1}}{14(8M_{s+2}^{(1)})^\tau}, \frac{\bar{\nu}_s}{2}\right\} = \frac{\bar{\nu}_s}{2}
			\end{align*}
		\item $\mathbf{J}_{s+1} \cup \salt{\mathbf{J}}(\mathbf{E}_{s+1}) \subset \salt{\mathbf{J}}(\mathbf{E}_{s})$
		\item If $\mathbf{E}_{s}$ does not attain its maximum at a critical point, then $$\sup \salt{\salt{\mathbf{J}}}(\mathbf{E}_{s}) \leq \sup \bigcup_{\mathbf{E} \in \mathcal{E}(\mathbf{E}_s)}\salt{\salt{\mathbf{J}}}(\mathbf{E}).$$
		Similarly, if $\mathbf{E}_s$ does not attain its minimum at a critical point, then $$\inf \salt{\salt{\mathbf{J}}}(\mathbf{E}_s) \geq \inf \bigcup_{\mathbf{E} \in \mathcal{E}(\mathbf{E}_s)} \salt{\salt{\mathbf J}}(\mathbf{E}).$$
		\item If $\mathbf{E}_{s+1}$ does not attain its supremum $\sup \mathbf{J}_{s+1}$ at a critical point and $\sup \mathbf{J}_{s+1} \neq \sup_{\mathbf{E}' \in \mathcal{E}(\mathbf{E}_s)}\sup \mathbf{J}'$, then there is some $\mathbf{E}_{s+1}' \in \mathcal{E}(\mathbf{E}_{s+1})$ such that $$B_{\frac{11\rho_s}{4}}(\sup \mathbf{J}_{s+1}) \subset \mathbf{J}_{s+1}'.$$  The same statement holds with $\inf$ in place of $\sup$.
	\end{enumerate}
\end{lem}

\begin{proof}
	Item 1 follows immediately from Lemma \ref{l:res}; from the same result, we likewise have 
	\begin{equation*}
		|\sup J_{s+1,i}^{(j)} - \sup J_{s,i}^{(j)}|, |\inf J_{s+1,i}^{(j)} - \inf J_{s,i}^{(j)}| \leq C\delta_s \ll \rho_s/2
	\end{equation*}
	for $j \in \{1,2\}$ and all $i$ for which these intervals are defined. The process of ``trimming'' the function by Lemma \ref{l:Asm3StabDR} or \ref{l:Asm3StabSR} does not change the supremum or infimum; thus, we still have 
	\begin{align*}
		|\sup \mathbf J_{s+1,i}^{(1)} - \sup J_{s,i}^{(1)}|, |\inf \mathbf J_{s+1,i}^{(1)} - \inf J_{s,i}^{(1)}| &\leq C\delta_s \ll \rho_s/2, \\
		|\sup \mathbf J_{s+1,n_s,\vee}^{(2)} - \sup J_{s,n_s}^{(2)}|, |\inf \mathbf J_{s+1,n_s,\wedge}^{(2)} - \inf J_{s,n_s}^{(2)}| &\leq C\delta_s \ll \rho_s/2.
	\end{align*}
	In particular, since each $\mathbf E_{s+1,n,\vee}^{(2)}$ attains its infimum at a critical point, and each $\mathbf E_{s+1,n,\wedge}^{(2)}$ attains it supremum at a critical point, this estimate applies at every non-critical value extremum of any $\mathbf E_{s+1} \in \mathcal E(\mathbf E_s)$.
	
	The next three items follow. Specifically, by construction, $$\sup \salt{\mathbf J}(\mathbf E_s) - \sup J_{s+1,i}^{(j)} \geq \rho_s;$$ and by definition, $$\sup \salt{\mathbf J}(\mathbf E_{s+1}) - \sup \mathbf J_{s+1} \leq \rho_{s+1}.$$ These estimates, combined with the analogous estimates with $\inf$ and with the estimates established above, give item 2. For item 3, we use the fact that by construction, $$\sup \salt{\mathbf J}(\mathbf E_s) - \sup_{i,j}\sup J_{s,i}^{(j)} \leq 2\bar{\rho}_s+\rho_s \ll \rho_{s-1}/24,$$ and the fact that, by definition, for each $\mathbf E \in \mathcal E(\mathbf E_s)$, $\mathbf{E} : \mathbf{I} \to \mathbf{J}$, we have $$\sup \mathbf{J} - \sup \salt{\salt{\mathbf{J}}}(\mathbf E) \leq \frac54\rho_s \ll \rho_{s-1}/24.$$ Since, by definition, $\sup\salt{\mathbf J}(\mathbf E_s) - \sup\salt{\salt{\mathbf J}}(\mathbf E_s) = \rho_{s-1}/8$, and the analogous estimates with $\inf$ hold, item 3 follows from the above estimates. Item 4 follows from the above estimates and the fact that, by construction, each pair of adjacent intervals $J_{s,i}^{(j)}$ overlap by $3\rho_s$. 
\end{proof}

Finally, we verify the directional regularity of our new endpoints:
\begin{lem}\label{l:dirreg}
	For all $\theta_* \in \mathbf{I}_{s+1}$ for all $E_* \in B_{\rho_s}(\mathbf{E}_{s+1}(\theta_*))$, $L_{s+1}/2+1$ is $s+1$-left-regular and $-(L_{s+1}/2+1)$ is $s+1$-right-regular relative to $(\theta_*,E_*)$.
\end{lem}

\begin{proof}
	We will show the point $L_{s+1}/2 + 1$ is $s+1$-left-regular; the other case is analogous.  Since $L_{s+1}/2+1$ is $s$-regular by Lemma \ref{l:Lconst}, it suffices to check that $[L_{s+1}/2+1, 5L_{s+1}/4+1]$ is $s$-nonresonant.  	
	
	In the case $j_{s+1} = 2$, we have $\theta_* \in \mathbf I^{(2)}_{s+1,n_s,\cup}$, an interval of length at most $\rho_{s-1}/12D_0$; thus $\dist(\theta_*,\mathbf I_{s,-}) \leq \rho_{s-1}/12D_0$, and $\dist(\theta_*+n_s\alpha,\mathbf I_{s,+}) \leq \rho_{s-1}/12D_0$. Then, by the Diophantine condition, $\dist(\theta_*+m\alpha, \mathbf I_s) > \rho_{s-1}/12D_0$ for all $m \in [L_{s+1}/2+1, 5L_{s+1}/4+1]$, which implies that interval is  $s$-nonresonant.
	
	Consider instead the case $j_{s+1} = 1$. 
	We have $$|\mathbf{E}_{s+1}(\theta_*) - \mathbf{E}_{s}(\theta_*)| \leq C\delta_{s} \ll \rho_{s}/2,$$ and for any $m \in [L_{s+1}/2+1,5L_{s+1}/4+1]$, $$0 < |m| \leq (5L_{s+1}/4+1)< 8M_{s+1}^{(1)};$$ by Lemma \ref{l:constructionsummary}, it follows that, for any such $m$, if $\theta_*+m\alpha \in \mathbf I_s$,
	\begin{align*}
		|\mathbf{E}_{s}(\theta_*+ m\alpha) - E_*| &\geq |\mathbf{E}_{s}(\theta_*) - \mathbf{E}_{s}(\theta_*+ m\alpha)| - |E_* - \mathbf{E}_{s}(\theta_*)| \\
		 &\geq |\mathbf{E}_{s}(\theta_*) - \mathbf{E}_{s}(\theta_*+ m\alpha)| - 2\rho_{s} > \rho_{s}. \qedhere
	\end{align*}
	
\end{proof}

\subsection{Resonant orbits}

We note that double resonance of any descendant $\mathbf{E}_{s+1}$ can  happen only at lengths greater than $L_{s+1}$: 

\begin{lem}\label{l:centerdist}
	If $j_{s+1} = 1$, then
	\begin{align*}
		\min\{|n| : n \in \mathcal{N}_{s+1}\} > 8M_{s+1}^{(1)} \geq L_{s+1}.
	\end{align*}
	If $j_{s+1} = 2$, then $\mathcal{N}_{s+1} = \emptyset$.
\end{lem}

\begin{proof}
	In the case $j_{s+1} = 2$, the interval $\mathbf{I}_{s+1}$ is a single interval of length at most $\rho_{s-1}/24D_0$, and $\mathcal{N}_{s+1} = \emptyset$ by the Diophantine condition.
	
	Consider instead the case $j_{s+1} = 1$, and suppose for the sake of contradiction that there exists $n \in \mathcal{N}_{s+1}$ with $|n| \leq 8M_{s+1}^{(1)}$.  Then by definition there exists $\theta_{s+1,n,-}  \in \mathbf{I}_{s+1,-}$ such that $\theta_{s+1,n,-}+ n\alpha \in \mathbf{I}_{s+1,+}$ and $\mathbf{E}_{s+1}(\theta_{s+1,n,-}) = \mathbf{E}_{s+1}(\theta_{s+1,n,-}+n\alpha)$.
	
	Define the function $g(\theta) = \mathbf{E}_{s}(\theta) - \mathbf{E}_{s}(\theta+n\alpha)$.  Since $j_{s+1} = 1$, we have that $|g(\theta_{s+1,n,-})| \leq 2C\delta_{s}$; furthermore, we have $$|\partial_\theta g(\theta)| \geq 2\left(\nu_{s} - C\frac{\delta_{s}}{\rho_{s}} \right) \geq \nu_s$$ for $\theta \in I_{s+1,n,-}^{(2)}$ by Lemma \ref{l:constructionsummary}.  Thus, there exists some point $\theta_{s,n,-}$ with $$|\theta_{s,n,-} - \theta_{s+1,n,-}| < 2C\frac{\delta_{s}}{\nu_s} \ll \frac{3}{4D_0}\bar{\rho}_{s}$$ such that $|g(\theta_{s,n,-})| = 0$; since $|n| \leq 8M_{s+1}^{(1)}$, $n \in \mathcal{N}_s$.  But $B_{3\bar{\rho}_{s-1}/4D_0}(\theta_{s-1,n,-}) \cap I_{s-1,i,-}^{(1)} = \emptyset$ for any $i$ and any $n \in \mathcal{N}_s$; thus, $j_{s+1} = 2$, a contradiction.
\end{proof}

This allows us to verify a remaining item of the inductive proposition:
\begin{lem} \label{l:item3}
If $E_* = E_{n_{s+1}}(\mathbf E_{s+1})$ for $n_{s+1} \in \mathcal N_{s+1}$,  then $\overline{B}_{\rho_{s}}(E_*) \subset \salt{\mathbf{J}}(\mathbf{E}_s)$.
\end{lem}
\begin{proof}
Let $n_{s+1} \in \mathcal N_{s+1}$ and $E_* = E_{n_{s+1}}$.
We have two cases:
\begin{itemize}
    \item If $\mathbf E_{s+1}$ does not attain its minimum at a critical point, since $E_* \in \salt{\mathbf J}(\mathbf E_{s+1})$, $E_* - \rho_{s} \in \mathbf J_{s+1} \subset \salt{\mathbf J}(\mathbf E_{s})$, by Lemma \ref{l:jrelations}.
    \item If $\mathbf E_{s+1}$ attains its minimum at a critical point, note that as $\mathcal N_{s+1}$ is nonempty, $j_{s+1}=1$ by Lemma \ref{l:centerdist}.  Thus, by Lemma \ref{l:res}, $\mathbf E_{s+1}$ is $C^2$-close to $\mathbf E_s$, so $\mathbf E_{s}$ also attains its minimum at a critical point, with  $|\inf\mathbf J_{s} - \inf\mathbf J_{s+1}| \leq 2\delta_{s}$.  By Lemma \ref{l:DRensepfromcrit}, $E_* - \inf \mathbf J_{s+1} \gg 2\delta_{s}$; thus 
    \begin{align*}
        E_* - \rho_{s} &\geq \inf \mathbf J_{s+1} + 2\delta_{s} - \rho_{s} \\ &\geq \inf \mathbf J_{s} - \rho_{s} \\&= \inf\salt{\mathbf J}(\mathbf E_{s}).
    \end{align*}
\end{itemize}
In either case, $E_* - \rho_{s} \geq \inf\salt{\mathbf J}(\mathbf E_{s})$; analogously, $E_* + \rho_{s} \leq \sup\salt{\mathbf J}(\mathbf E_{s})$.
\end{proof}

We likewise demonstrate that $s$-resonant points $m$ correspond to points $\theta_*+ m\alpha$ in the domain of the relevant Rellich function (part of Hypothesis \ref{pr:resorbs}):

\begin{lem}\label{l:Scond}
	Let $\theta_*\in \mathbb{T}$ and $\mathbf{E}_{s+1} \in \mathcal{E}(\mathbf{E}_s)$, and suppose that $E_*$ satisfies the following:
	\begin{enumerate}
		\item If $j_{s+1} = 1$, either $E_* \in \salt{\mathbf{J}}(\mathbf{E}_{s+1})$, or the closest point in $\mathbf{J}_{s+1}$ to $E_*$ is a critical value of $\mathbf{E}_{s+1}$.
		\item If $j_{s+1} = 2$ (and so $\mathbf{E}_{s+1} = \mathbf{E}_{s+1,n_s,\bullet}^{(2)}$), then $E_* \in [\inf \salt{\mathbf{J}}(\mathbf{E}_{s+1,n_s,\wedge}^{(2)}), \sup \salt{\mathbf{J}}(\mathbf{E}_{s+1,n_s,\vee}^{(2)})]$
	\end{enumerate}
	Let $m \in \mathcal{S}_{s}(\theta_*,E_*)$.  Then:
	\begin{enumerate}
		\item If $j_{s+1} = 1$, then $B_{\rho_s/24D_0}(\theta_*+m\alpha)\subset \mathbf{I}_{s+1}$.
		\item If $j_{s+1} = 2$ (and so $\mathbf{E}_{s+1} = \mathbf{E}_{s+1,n_s,\bullet}^{(2)}$), then for some $p \in \{m, m-n_s\}$, $B_{\rho_s/24D_0}(\theta_*+p\alpha) \subset \mathbf{I}_{s+1,n_s,\cup}^{(2)}$.  Furthermore, if $\theta_* + p\alpha \in \mathbf{I}_{s+1,n_s,\bullet} \setminus \mathbf{I}_{s+1,n_s,\cap}$, then $$|\mathbf{E}_{s+1,n_s,\bullet}^{(2)}(\theta_*+ p\alpha) - E_*| < \frac{9}{8}\rho_s.$$
	\end{enumerate}
\end{lem}

\begin{proof}
	The proof proceeds by case analysis:
	\begin{enumerate}
		\item $j_{s+1} = 1$:
		\begin{enumerate}
			\item \textit{$\mathbf{I}_{s+1}$ is a union of two disjoint intervals:} In this case, $\mathbf{E}_{s+1}$ has no critical points, so $E_* \in \salt{\mathbf{J}}(\mathbf{E}_{s+1}) = [\inf \mathbf{J}_{s+1} + \frac98\rho_s, \sup \mathbf{J}_{s+1} - \frac98\rho_s]$.  Thus, $$\mathbf{E}_{s}(\theta_*+m\alpha) \in [\inf \mathbf{J}_{s+1}+\frac{1}{12}\rho_s, \sup \mathbf{J}_{s+1}-\frac{1}{12}\rho_s],$$ and so $$\mathbf{E}_s\left(B_{\frac{\rho_s}{24D_0}}(\theta_*+m\alpha)\right) \subset [\inf \mathbf{J}_{s+1}+\frac{1}{24}\rho_s, \sup \mathbf{J}_{s+1}-\frac{1}{24}\rho_s].$$  On the other hand, one has  $$\mathbf{E}_{s,\pm}(\mathbf{I}_{s+1,\pm}) \supset [\inf \mathbf{J}_{s+1}+\frac{1}{24}\rho_s, \sup \mathbf{J}_{s+1}-\frac{1}{24}\rho_s]$$ since $j_{s+1} = 1$, and in particular
			$$\mathbf{I}_{s+1} \supset \mathbf{E}_{s}^{-1}\left([\inf \mathbf{J}_{s+1}+\frac{1}{24}\rho_s, \sup \mathbf{J}_{s+1}-\frac{1}{24}\rho_s]\right) \supset B_{\frac{\rho_s}{24D_0}}(\theta_*+m\alpha).$$
			
			\item \textit{$\mathbf{I}_{s+1}$ is a single interval:} In this case, $\mathbf{E}_{s+1}$ has a unique critical point in $\mathbf{I}_{s+1}$.  Suppose that this critical point is a minimum (the other case is analogous).  Since $j_{s+1} = 1$, $\mathbf{E}_s$ likewise achieves its minimum at a critical point, and by definition of $\salt{\mathbf{J}}(\mathbf{E}_{s+1})$ and our assumption on $E_*$ we have $$E_* \leq \sup \mathbf{J}_{s+1} - \frac98\rho_s.$$  As above, we find $$\mathbf{E}_s\left(B_{\frac{\rho_s}{24D_0}}(\theta_*+m\alpha)\right) \subset [\inf \mathbf{J}_s, \sup \mathbf{J}_{s+1} - \frac{1}{24}\rho_s].$$  On the other hand, one has that $$\mathbf{E}_{s,\pm}(\mathbf{I}_{s+1,\pm}) \supset [\inf \mathbf{J}_{s}, \sup \mathbf{J}_{s+1}-\frac{1}{24}\rho_s]$$ since $j_{s+1} = 1$, and in particular
			$$\mathbf{I}_{s+1} \supset \mathbf{E}_{s}^{-1}\left([\inf \mathbf{J}_{s}, \sup \mathbf{J}_{s+1}-\frac{1}{24}\rho_s]\right) \supset B_{\frac{\rho_s}{24D_0}}(\theta_*+m\alpha).$$
		\end{enumerate}
	\item \textit{$j_{s+1} = 2$ and $\mathbf{E}_{s+1} = \mathbf{E}_{s+1,n_s,\bullet}^{(2)}$}:  By assumption, we have $$E_* \in [\inf \salt{\mathbf{J}}(\mathbf{E}_{s+1,n_s,\wedge}^{(2)}), \sup \salt{\mathbf{J}}(\mathbf{E}_{s+1,n_s,\vee}^{(2)})],$$ from which it follows that $$\mathbf{E}_{s}\left(B_{\frac{\rho_s}{24D_0}}(\theta_*+m\alpha)\right) \supset [\inf \mathbf{J}_{s+1,\wedge}+\frac{1}{24}\rho_s, \sup \mathbf{J}_{s+1,\vee}-\frac{1}{24}\rho_s]$$ in analog to the above.  We have two subcases:
	\begin{enumerate}
		\item $\theta_*+m\alpha \in \mathbf{I}_{s,-}$: Let $I = [\inf \mathbf{I}_{s+1,\vee}, \sup \mathbf{I}_{s+1,\wedge}] \subset \mathbf{I}_{s+1,n_s,\cup}$.  Since $j_{s+1} = 2$, we have $$\mathbf{E}_{s,-}(I) \supset [\inf \mathbf{J}_{s+1,\wedge} + \frac{1}{24}\rho_s, \sup \mathbf{J}_{s+1,\vee} - \frac{1}{24}\rho_s],$$ from which it follows that $$I \supset \mathbf{E}_{s,-}^{-1}\left([\inf \mathbf{J}_{s+1,\wedge} + \frac{1}{24}\rho_s, \sup \mathbf{J}_{s+1,\vee} - \frac{1}{24}\rho_s]\right) \supset B_{\frac{\rho_s}{24D_0}}(\theta_*+m\alpha).$$
		Furthermore, for $\theta \in (\mathbf{I}_{s+1,n_s,\vee} \setminus \mathbf{I}_{s+1,n_s,\wedge}) \cap I$, we have by Lemma \ref{l:DRedgeSR} that $$|\mathbf{E}_{s+1,n_s,\vee}^{(2)}(\theta) - \mathbf{E}_s(\theta)| < C\delta_s \ll \rho_s/12,$$
		and similarly for $\vee$ and $\wedge$ swapped.  The lemma thus holds in this case taking $p = m$.
		
		\item $\theta_*+m\alpha \in \mathbf{I}_{s,+}$: Let $I = [\inf \mathbf{I}_{s+1,\wedge}, \sup \mathbf{I}_{s+1,\vee}] \subset \mathbf{I}_{s+1,n_s,\cup}$.  Since $j_{s+1} = 2$, we have $$\mathbf{E}_{s,+}(I+n_s\alpha) \supset [\inf \mathbf{J}_{s+1,\wedge} + \frac{1}{24}\rho_s, \sup \mathbf{J}_{s+1,\vee} - \frac{1}{24}\rho_s],$$ from which it follows that $$I +n_s\alpha \supset \mathbf{E}_{s,+}^{-1}\left([\inf \mathbf{J}_{s+1,\wedge} + \frac{1}{24}\rho_s, \sup \mathbf{J}_{s+1,\vee} - \frac{1}{24}\rho_s]\right) \supset B_{\frac{\rho_s}{24D_0}}(\theta_*+m\alpha),$$ i.e. for $p = m-n_s$, $$\mathbf{I}_{s+1,n_s,\cup} \supset I \supset B_{\frac{\rho_s}{24D_0}}(\theta_*+p\alpha).$$  Furthermore, for $\theta \in (\mathbf{I}_{s+1,n_s,\vee} \setminus \mathbf{I}_{s+1,n_s,\wedge}) \cap I$, we have by Lemma \ref{l:DRedgeSR} that $$|\mathbf{E}_{s+1,n_s,\vee}^{(2)}(\theta) - \mathbf{E}_s(\theta+n_s\alpha)| < C\delta_s \ll \rho_s/12,$$
		and similarly for $\vee$ and $\wedge$ swapped.  The lemma thus holds in this case taking $p = m-n_s$. \qedhere
	\end{enumerate} 
	\end{enumerate}
\end{proof}

\subsection{Nonresonance}

We now prove that the nonresonance hypothesis \ref{prsub:NR} holds at scale $s+1$.  
We begin by establishing Green's function decay for $s+1$-nonresonant intervals provided their left and right endpoints are $(s+1)$-left- and $(s+1)$-right-regular, respectively.

\begin{lem}\label{l:NRreg}
	Let $\theta_* \in \mathbb{T}$ and $E_* \in \salt{\mathbf{J}}(\mathbf{E}_{s})$ be such that 
	\begin{itemize}
		\item If $\mathbf{E}_s$ does not attain its maximum at a critical point, then $E_* \leq \sup \bigcup_{\mathbf{E} \in \mathcal{E}(\mathbf{E}_s)}\salt{\mathbf{J}}(\mathbf{E})$.
		\item If $\mathbf{E}_s$ does not attain its minimum at a critical point, then $E_* \geq \inf \bigcup_{\mathbf{E} \in \mathcal{E}(\mathbf{E}_s)}\salt{\mathbf{J}}(\mathbf{E})$.
	\end{itemize}
	Let $\mathbf{E}_{s+1} \in \mathcal{E}(\mathbf{E}_s)$ be a Rellich curve minimizing $\dist(E_*, \salt{\mathbf{J}}(\mathbf{E}))$ among all $\mathbf{E} \in \mathcal{E}(\mathbf{E}_s)$, and let $\snext\Lambda = [a,b] \subset \mathbb{Z}$ be an $(s+1)$-nonresonant interval with $a$ being $(s+1)$-left-regular and $b$ being $(s+1)$-right-regular.  Then for $E \in B_{\rho_{s+1}/2}(E_*)$ and $\theta \in B_{\rho_{s+1}/8D_0}(\theta_*)$,
	$$\log|R^{\snext\Lambda}_{\theta,E}(m,n)| \leq - \gamma_{s+1}|m-n|, \quad |m-n| \geq \ell_{s+1},$$ and $$\|R^{\snext\Lambda}_{\theta,E}\| \leq 4\rho_{s+1}^{-1}.$$
\end{lem}

\begin{proof}
	Recall the set $\mathcal{S}_s(\theta_*, E_*)$ (defined relative to the parent curve $\mathbf{E}_s$), and define $\mathcal{S}_s^{\snext\Lambda} := \mathcal{S}_s \cap \snext\Lambda$.  If $\mathcal{S}_s^{\snext\Lambda}$ is empty, then $\snext\Lambda$ is $s$-nonresonant and the lemma follows from the inductive hypothesis.  We thus assume $\mathcal{S}_s^{\snext\Lambda}$ is nonempty.
	
	In order to apply Lemmas \ref{l:NRspecsep} and \ref{l:NRmultdecay}, we wish to find a finite family $\{\Lambda_{s+1} + p\}_{p \in \overline{\mathcal{S}}_s^{\snext\Lambda}}$ of translates of $\Lambda_{s+1}$ such that each $m \in \mathcal{S}_s^{\snext\Lambda}$ is near the center of $\Lambda_{s+1}+p$ for some $p \in \overline{\mathcal{S}}_s^{\snext\Lambda}$ and $\|R^{\Lambda_{s+1}+p}_{\theta_*,E_*}\| \leq \rho_{s+1}^{-1}$ for all $p \in \overline{\mathcal{S}}_s^{\snext\Lambda}$.  We do so by case analysis:
	\begin{enumerate}
		\item $j_{s+1} = 1$:
		
		For any $m \in \mathcal{S}_s^{\snext\Lambda}$, $\theta_*+ m\alpha \in \mathbf{I}_{s+1}$ by Lemma \ref{l:Scond}, and since $|\mathbf{E}_s(\theta_*+m\alpha) - E_*| < 25\rho_s/24$ and $\snext\Lambda$ is $(s+1)$-nonresonant, we have $$\frac98 \rho_s > |\mathbf{E}_{s+1}(\theta_*+m\alpha) - E_*| \geq \rho_{s+1}.$$  By the eigenvalue separation estimates in Section \ref{sec:SR} coming from Lemma \ref{l:res}, it follows that $|\lambda - E_*| \geq \rho_{s+1}$ for any eigenvalue $\lambda$ of $H^{\Lambda_{s+1}+m}(\theta_*)$, and so $\|R_{\theta_*,E_*}^{\Lambda_{s+1}+m}\| \leq \rho_{s+1}^{-1}$.  In this case, we define $\overline{\mathcal{S}}_s^{\snext\Lambda} = \mathcal{S}_s^{\snext\Lambda}$.
		
		\item $j_{s+1} = 2$:
		
		Let $\mathbf{E}_{s+1} = \mathbf{E}_{s+1,n_s,\vee}^{(2)}$ (the $\wedge$ case is analogous), and fix $m \in \mathcal{S}_s^{\snext\Lambda}$.  By Lemma \ref{l:Scond}, $\theta_* + p\alpha \in \mathbf{I}_{s+1,n_s,\cup}^{(2)}$ for some $p = p_m \in \{m, m-n_s\}$.  By the regularity assumption on the endpoints of $\snext\Lambda$, we must have that $p \in \snext\Lambda$.  We have subcases:  
		\begin{enumerate}
			\item $\theta_* + p\alpha \in \mathbf{I}_{s+1,n_s,\cap}^{(2)}$: We have $|\mathbf{E}_{s+1}(\theta_*+ p\alpha) - E_*| \geq \rho_{s+1}$.
			To separate $E_*$ from $\mathbf E^{(2)}_{s+1,n_s,\wedge}(\theta_*+p\alpha)$, we crucially use the uniform separation between the images of $\mathbf E^{(2)}_{s+1,n_s,\vee}$ and $\mathbf E^{(2)}_{s+1,n_s,\wedge}$ guaranteed by Proposition \ref{pr:unifmLocalSep}. This separation is greater than $2\rho_{s+1}$, which implies that $\salt{\mathbf J}(\mathbf E^{(2)}_{s+1,n_s,\vee})$ and $\salt{\mathbf J}(\mathbf E^{(2)}_{s+1,n_s,\wedge})$ are disjoint; since, by assumption, %
			$$\dist(E_*, \salt{\mathbf{J}}(\mathbf{E}_{s+1})) \leq \dist(E_*, \salt{\mathbf{J}}(\mathbf{E}_{s+1,n_s,\wedge}^{(2)})),$$ we have $$E_* \geq %
			\sup \salt{\mathbf{J}}(\mathbf{E}_{s+1,n_s,\wedge}^{(2)})=
			\sup \mathbf{J}_\wedge + \rho_{s+1},$$ and so $$|\mathbf{E}_{s+1,n_s,\wedge}^{(2)}(\theta_* + p\alpha) - E_*| \geq \rho_{s+1}.$$  
			It remains to separate $E_*$ from other eigenvalues of $H^{\Lambda_{s+1}+p}$.  Because $$\dist\left(E_*, \mathbf{E}_s(\mathbf{I}_{s+1,n_s,\cap}^{(2)} \cup (\mathbf{I}_{s+1,n_s,\cap}^{(2)}+n_s\alpha))\right) < \frac{25}{24}\rho_s \ll \frac{1}{12}\rho_s$$ and by construction $$|\mathbf{E}_s(\mathbf{I}_{s+1,n_s,\cap}^{(2)} \cup (\mathbf{I}_{s+1,n_s,\cap}^{(2)}+n_s\alpha))| \ll \frac{1}{4}\rho_{s-1},$$ it follows that $$|E_* - \mathbf{E}_{s}(\theta_{s,n_s,-})| < \frac{1}{3}\rho_{s-1} = \frac{1}{2}\left(\frac{2}{3}\rho_{s-1}\right).$$ Thus we may apply the estimate on $\|R_\perp(E_*; H^{\Lambda_{s+1}+p}, \{\mathbf{E}_{s+1,\vee}\}\cup\{\mathbf{E}_{s+1,\wedge}\})\|$ coming from Section \ref{sec:DR} and Lemma \ref{l:res}.  It follows that $\|R_{\theta_*,E_*}^{\Lambda_{s+1}+p}\| \leq \rho_{s+1}^{-1}$.
			
			\item $\theta_* + p\alpha \in \mathbf{I}_{s+1,n_s,\vee}^{(2)} \setminus \mathbf{I}_{s+1,n_s,\wedge}^{(2)}$ \textit{(The case with $\vee$ and $\wedge$ swapped is analogous.):} Lemma \ref{l:Scond} and the eigenvalue separation estimate from Lemma \ref{l:DRedgeSR} ensure $|\lambda - E_*| \geq \rho_{s+1}$ for any eigenvalue $\lambda \neq \mathbf{E}_{s+1,n_s,\wedge}^{(2)}$ of $H^{\Lambda_{s+1}+p}(\theta_*)$.  As in the $\mathbf{I}_{s+1,n_s,\cap}$ case, $$|\mathbf{E}_{s+1,n_s,\wedge}^{(2)}(\theta_*+p\alpha) -E_*| \geq \rho_{s+1}.$$  It follows that $\|R_{\theta_*,E_*}^{\Lambda_{s+1}+p}\| \leq \rho_{s+1}^{-1}$.  
		\end{enumerate}
		We denote by $\overline{\mathcal{S}}_s^{\snext\Lambda}$ the set of all integers $p$ coming from $m \in \mathcal{S}_s^{\snext\Lambda}$ as above.  
	\end{enumerate}
	For each interval $\Lambda_{s+1}+p$, $p \in \overline{\mathcal{S}}_s^{\snext\Lambda}$, it follows immediately from $\|R_{\theta_*,E_*}^{\Lambda_{s+1}+p}\| \leq \rho_{s+1}^{-1}$ that, for $E \in B_{3\rho_{s+1}/4}(E_*)$ and $\theta \in B_{\rho_{s+1}/8D_0}(\theta_*)$, that $$\|R_{\theta,E}^{\Lambda_{s+1}+p}\| \leq 8\rho_{s+1}^{-1}.$$  By the $s+1$-regularity of the endpoints of $\snext\Lambda = [a,b]$, for any $m \in \mathcal{S}_s^{\snext\Lambda}$, we have $$\min\{|m-a|,|m-b|\} \geq 3L_{s+1}/4;$$ since $|p-m| \leq 8M_{s+1}^{(1)} \ll L_{s+1}/4$, we have that $\Lambda_{s+1}+p \subset \Lambda$ for all $p \in \overline{\mathcal{S}}_s^{\snext\Lambda}$.  Moreover, by Lemma \ref{l:constructionsummary}, the intervals $\{\Lambda_{s+1}+p\}_{p \in \overline{\mathcal{S}}_s^{\snext\Lambda}}$ are non-overlapping.  Enumerating $\overline{\mathcal{S}}_s^{\snext\Lambda} = \{p_i\}_{i=1}^{|\overline{\mathcal{S}}_s^{\snext\Lambda}|}$ in increasing order, we denote by $\salt\Lambda_i := [p_i, p_i + n_s]$ (with the convention that $n_s = 0$ if $j_{s+1}=1$) and define the intervals $\salt\Lambda_{i,l/r}$ to be the maximal intervals to the left/right of $\salt\Lambda_i$, i.e.
	\begin{align*}
		\salt\Lambda_{i,l} &= [\max \salt\Lambda_{i-1}+1,\min \salt\Lambda_i - 1] =: [\salt a_{i,l},\salt b_{i,l}], \\
		\salt\Lambda_{i,r} &= [\max \salt\Lambda_{i}+1,\min \salt\Lambda_{i+1} - 1] =: [\salt a_{i,r},\salt b_{i,r}].
	\end{align*}
	We note that $\salt\Lambda_{i,r} = \salt\Lambda_{i+1,l}$, but we fix this notation to emphasize the relationship of these intervals to $\salt\Lambda_i$.
	
	Again by Lemma \ref{l:constructionsummary}, each interval $\salt\Lambda_{i,l/r}$ is $s$-nonresonant and has length at least $L_{s+1} \gg 2L_s+3M_s^{(1)}$, and thus each such interval satisfies the inductive hypothesis \ref{prsub:NR}.  In particular, there exist intervals ${\Lambda}_{i,l/r} = [{a}_{i,l/r}, {b}_{i,l/r}] \subset \salt\Lambda_{i,l/r}$ satisfying the Green's function decay property for $(\ell_s, \gamma_s)$ with $$\max\{|\salt a_{i,l/r} - {a}_{i,l/r}|,|\salt b_{i,l/r} - {b}_{i,l/r}|\} \leq 2L_s + 3M_s^{(1)} \ll \frac{1}{2}M_{s+1}^{(1)}.$$  Moreover, by the regularity of the endpoints of $\snext\Lambda$, we can choose $${a}_{1,l} = \salt{a}_{1,l} = a, \quad {b}_{|\overline{\mathcal{S}}^{\snext\Lambda}_s|,r} = \salt{b}_{|\overline{\mathcal{S}}^{\snext\Lambda}_s|,r} = b.$$
	
	Define ${\Lambda}_i = [{b}_{i,l}+1, {a}_{i,r}-1]$.  We have ${\Lambda}_i \subset \Lambda_{s+1} + p_i \subset \snext\Lambda$ and $\|R^{\Lambda_{s+1}+p_i}\| \leq 8\rho_{s+1}^{-1}$, so we may apply Lemma \ref{l:NRspecsep} to conclude that $\snext\Lambda$, partitioned by the $\Lambda_i$ and $\Lambda_{i,l/r}$ subintervals, satisfies Assumption \ref{as:NRmult} 
	with $({\ell},{\gamma})=(\ell_{s},\gamma_{s})$ and $\snext\ell = \ell_{s+1}$.  By Lemma \ref{l:NRmultdecay}, $\snext\Lambda$ satisfies the Green's function decay property for $(\ell_{s+1}, \gamma_{s+1})$ where $\gamma_{s+1} = \gamma_{s} - 16|\log\varepsilon\log(\rho_{s+1}/8)|/\ell_{s+1}$.

	Since we have shown the Green's function decay property for $|E-E_*|<\frac34\rho_{s+1}$, it follows that the resolvent is well-defined for $|E-E_*|<\frac34\rho_{s+1}$, i.e., $E \notin \spec H^{\snext\Lambda}(\theta)$. Thus, for $|E-E_*|\leq \rho_{s+1}/2$, $\|R^{\snext\Lambda}_{\theta,E}\|\leq4\rho_{s+1}^{-1}$.
\end{proof}

Finally, the following lemma shows that if the endpoints of $\snext\Lambda$ are not regular, we can adjust them (provided $\snext\Lambda$ is sufficiently long) to find a subinterval which does satisfy the conditions of Lemma \ref{l:NRreg}.
\begin{lem} \label{l:NRintadj}
	Let $\theta_*,E_*,\mathbf E_{s+1} \in \mathcal{E}^{(j_{s+1})}(\mathbf{E}_s)$ satisfy the relations stated in Lemma \ref{l:NRreg}, and let 
	$\snext\Lambda = [a,b]$ with 
	$|\snext\Lambda| \geq 2L_{s+1}+3M_{s+1}^{(1)}$. 
	Then there exists a subinterval $\tilde{\Lambda} = [\tilde{a},\tilde{b}] \subset \Lambda$ with $|\tilde\Lambda|\geq|\snext\Lambda|-(2L_{s+1}+3M_{s+1}^{(1)})$ such that $\tilde{a}$ is $s+1$-left-regular and $\tilde{b}$ is $s+1$-right-regular. 
\end{lem}

\begin{proof}
	In the following, if $j_{s+1} = 1$, we adopt the convention that $n_s=0$.  
	
	Let $m_a$ minimize $|m-a|$ among $m \in \mathcal{S}_{s}^{\snext\Lambda}$.  %
	If $\theta_*+m_a\alpha \in \mathbf I_s$ 
	and $|m_a-a|<L_s$, we replace $a$ with $a_{s-1} = \max\{m_a+1, m_a+n_0+1\}$; similarly, if $\theta_*+m_a\alpha \in \mathbf I_{s+1}+n_s\alpha$ and $|m_a-a|<L_{s+1}$, replace $a$ with $a_{s} = \max\{m_a+1, m_a-n_s+1\}$.  
	Otherwise $a_{s}=a$.  We perform analogous adjustments to $b$. 
	We repeat this process for each $s-1 \geq i \geq 1$, with $i$ replacing $s$, and set $\tilde a = a_0$ and $\tilde b = b_0$. 
	By the inductive proposition, these endpoints  satisfy the necessary regularity conditions.
\end{proof}

\subsection{Proof of Proposition \ref{pr:induction} and Theorem \ref{t:induction}}

\begin{proof}
    The lemmas above complete the proof of inductive Proposition \ref{pr:induction} at scale $s$.  The integers $L_{s+1}$ satisfy the Proposition by Lemmas \ref{l:Lconst} and \ref{l:dirreg}.  Any function in $\mathbf{E}_{s+1} \in \mathcal{E}(\mathbf{E}_s)$ is a Rellich function of some $H^{\Lambda_{s+1}}$ satisfying Assumption \ref{as:C2fn} by Lemma \ref{l:jrelations}; such a function satisfies inductive hypothesis \ref{prsub:c2rell} by Lemma \ref{l:res}, hypothesis \ref{pr:resorbs} by Lemmas \ref{l:centerdist} and \ref{l:Scond}, and hypothesis \ref{prsub:NR} by Lemmas \ref{l:NRreg} and \ref{l:NRintadj}.   %
    For any $\mathbf{E}_{s+1} \in \mathcal{E}(\mathbf{E}_{s})$, $\mathbf J_{s+1} \cup \salt{\mathbf J}(\mathbf{E}_{s+1}) \subset \salt{\mathbf J}(\mathbf E_s)$ by Lemma \ref{l:jrelations}; and if $E_* = E_{n_{s+1}}(\mathbf E_{s+1})$ for $n_{s+1} \in \mathcal N_{s+1}$, then $\overline{B}_{\rho_{s}}(E_*) \subset \salt{\mathbf{J}}(\mathbf{E}_s)$ by Lemma \ref{l:item3}.  Item 4 follows from Lemma \ref{l:constructionsummary}.
    
    Since our choice of $\varepsilon < \varepsilon_0 = \rho_0^4$ was arbitrary provided $L$ was sufficiently long as in Lemma \ref{l:lenrelns}, Theorem \ref{t:induction} holds.
\end{proof}

\newpage
\section{Localization and Cantor spectrum}

With the induction argument of Section \ref{sec:induction} complete, we have now constructed a tree of Rellich functions $\mathbf{E}_s$, cf. Figure \ref{f:Rellichtree}, with each child classified as being either simple- or double-resonant relative to its parent.  In this section we relate this tree to spectral information about the limiting operator $H$.  In particular, we will show the following:
\begin{enumerate}
    \item The spectral points of $H$ are precisely the limit points of the modified codomains $\salt{\salt{\mathbf{J}}}(\mathbf{E}_k)$ along any infinite path $\{\mathbf{E}_k\}_{k=0}^\infty$ in our Rellich tree.
    \item Every Rellich function $\mathbf{E}$ has some double-resonant descendant; 
    by the uniform local separation estimate from Proposition \ref{pr:unifmLocalSep},
    the gap that this double resonance opens is large enough to %
    remain open for all future scales, guaranteeing Cantor spectrum. 
    \item The set $\mathrm{B}$ of bad phases $\theta_*$ which encounter a scale-$k$ double resonance on an orbit of approximate size $M_{k+1}^{(2)}$ for infinitely many scales $k$ has zero measure.  The full-measure complement $\Theta$ of this bad set will have the property that, for $\theta \in \Theta$, any generalized eigenvalue $E(\theta)$ of $H(\theta)$ corresponds to a path $\{\mathbf{E}^{(j_{k})}_k\}_{k=0}^\infty$ which eventually consists only of simple resonances, i.e. there exists some $K$ such that $j_k = 1$ for all $k \geq K$.  In fact, we will show that this path can be chosen so that the corresponding center of localization $m_k$ eventually remains fixed; this will yield exponential localization of the corresponding generalized eigenfunction, hence Anderson localization.
\end{enumerate}

To this end, it will be useful to introduce some notation to refer to different parts of our Rellich tree.  Recall that $\mathcal{E}(\mathbf{E})$ denotes the collection of all immediate children of the Rellich curve $\mathbf{E}$, and that $\mathcal{E}_s$ denotes the collection of all scale-$s$ Rellich functions, i.e. the $s^{th}$ ``generation'' of the tree.  For a Rellich function $\mathbf{E} \in \mathcal{E}_k$, $k \leq s$, we denote by $\mathcal E_s(\mathbf E)$ the collection of all the scale-$s$ descendants of $\mathbf E$; i.e., for a scale-$k$ Rellich function $\mathbf E_k$, $\mathcal E_k(\mathbf E_k) = \{\mathbf E_k\}$, and for all $s\geq k$, if $\mathbf E' \in \mathcal E_s(\mathbf E_k)$ and $\mathbf E'' \in \mathcal E(\mathbf E')$, then $\mathbf E'' \in \mathcal E_{s+1}(\mathbf E_k)$.
With this notation we have $\mathcal E_k = \mathcal E_k(\mathbf E_0)$; we denote by $\mathcal E := \bigcup_{k\geq0} \mathcal E_k$ the entire Rellich tree.

\subsection{Characterization of the spectrum}

In this first subsection, we use paths through the tree of Rellich curves to characterize the spectrum of $H$.

First, note that there can be at most one energy common to the modified codomains $\salt{\mathbf{J}}(\mathbf{E}_k)$ of all Rellich functions in an infinite path:

\begin{lem} \label{l:Jto0}
Let $\{\mathbf E_k\}_{k=0}^\infty$ be a sequence of Rellich functions such that $\mathbf E_{k+1} \in \mathcal E(\mathbf E_k)$ for all $k\geq0$.  Then $\liminf_{k\to\infty}|{\salt{\mathbf J}}(\mathbf E_k)|=\liminf_{k\to\infty}|\salt{\salt{\mathbf J}}(\mathbf E_k)|=0$. In particular, $\bigcap_{k=0}^\infty \salt{\salt{\mathbf J}}(\mathbf E_k)$ contains at most a single point.
\end{lem}
\begin{proof}
By construction, $|\salt{\mathbf J}(\mathbf E_k)|,|\salt{\salt{\mathbf J}}(\mathbf E_k)| \leq 4\bar\rho_{k-1} \to 0$.
\end{proof}

Thus, each infinite path characterizes a single point; we will show these are spectral points. Define the sets 
\begin{equation*}
    \Sigma := \left\{E \in \mathbb R : E = \bigcap_{k=0}^\infty \salt{\salt{\mathbf J}}(\mathbf E_k),\; \mathbf E_{k+1} \in \mathcal E(\mathbf E_k)\, \forall k\geq0\right\}
\end{equation*}
and
\begin{equation*}
G := (\mathbb R \setminus \salt{\salt{\mathbf J}}(\mathbf{E}_0)) \cup \bigcup_{\mathbf E \in \mathcal{E}}\left(\salt{\salt{\mathbf J}}(\mathbf E) \setminus \bigcup_{\mathbf E' \in \mathcal{E}(\mathbf E)}\salt{\salt{\mathbf J}}(\mathbf E')\right).
\end{equation*}
We will show that $\Sigma$ characterizes the spectrum of $H(\theta)$, while $G$ characterizes the spectral gaps.  We first show that points in $\Sigma$ are indeed spectral:

\begin{lem}\label{l:specS}\begin{equation*}
\Sigma \subset \spec H(\theta).
\end{equation*}\end{lem}
\begin{proof}
Let $E \in \Sigma$. Then for each $k\geq0$, we can find a $(\theta_k,E_k) \in \mathbb T \times \mathbb R$, with $|E-E_k|\leq\rho_k$, and $\mathbf {E}_k \in \mathcal{E}$ such that $H^{\Lambda_k}\psi_k = \mathbf{E}_k(\theta_k)\psi_k = E_k\psi_k$. By irrationality of $\alpha$, we can find some $m_k \in \mathbb Z$ such that $\|(\theta+m_k\alpha)-\theta_k\|_\mathbb{T} <\rho_k/D_0$. Let $\mathcal P$ be the partition of $\mathbb Z$ into $\Lambda_k$ and its complement, and let $\psi_k^{m_k}$ denote $\psi_k$ shifted by $m_k$. Then 
\begin{align*}
\|(H(\theta) - E)\psi_k^{m_k}\| &= \|(H(\theta+m_k\alpha)-E)\psi_k\| \\
&\leq \|H(\theta+m_k\alpha)-H(\theta_k)\| + \|(H(\theta_k)-E_k)\psi_k\| + |E_k-E| \\
&\leq  \|V(\theta+m_k\alpha)-V(\theta_k)\| + \|\Gamma_{\mathcal P}\psi_k\| + |E_k-E| \\
&< 3\rho_k
\end{align*}
Thus $\{\psi_k^{m_k}\}_{k=0}^\infty$ forms a Weyl sequence for $E$, and $E \in \spec H(\theta)$.
\end{proof}

We now establish an important technical lemma demonstrating that energies
outside of the modified codomains of all children of a Rellich curve are uniformly nonresonant on sufficiently long intervals $\Lambda$:
\begin{lem} \label{l:longNR}
Let $E_* \in \salt{\salt{\mathbf J}}(\mathbf E_k) \setminus \bigcup_{\mathbf E \in \mathcal E(\mathbf E_k)} \salt{\salt{\mathbf J}}(\mathbf E)$. Then for any $\theta_* \in \mathbb{T}$, there are arbitrarily long intervals $\tilde{\Lambda}$ with the $(\ell_{k+1}, \gamma_{k+1})$ Green's function decay property.
\end{lem}
\begin{proof}
It suffices to show that the hypotheses of Lemma \ref{l:NRintadj} are met by $E_*$, any $\theta_* \in \mathbb T$, and any sufficiently long interval $\Lambda \subset \mathbb Z$.

Suppose $\mathbf E_k$ does not attain its maximum at a critical point.  We have $E_* \in \salt{\salt{\mathbf J}}(\mathbf E_k) \subset {\salt{\mathbf J}}(\mathbf E_k)$, and, by item 2 of Lemma \ref{l:jrelations}, $$E_* \leq \sup \salt{\salt{\mathbf J}}(\mathbf E_k) \leq \sup \bigcup_{\mathbf E \in\mathcal E(\mathbf E_k)} \salt{\salt{\mathbf J}}(\mathbf E) \leq \sup \bigcup_{\mathbf E \in\mathcal E(\mathbf E_k)} \salt{\mathbf J}(\mathbf E)$$ 
Similarly $E_* \geq \inf \bigcup_{\mathbf E \in\mathcal E(\mathbf E_k)} \salt{\mathbf J}(\mathbf E)$ if $\mathbf E_k$ does not attain its minimum at a critical point. 

Now suppose that $E_* \in \salt{\mathbf J}(\mathbf E) \setminus \salt{\salt{\mathbf J}}(\mathbf E)$ for some $\mathbf E \in \mathcal E(\mathbf E_k)$. By definition of $\salt{\mathbf J}$ and $\salt{\salt{\mathbf J}}$, this means that $E_*$ is near a non-critical-point extremum of $\mathbf E$, and by item 3 of Lemma \ref{l:jrelations}, it follows that $E_* \in \salt{\salt {\mathbf J}}(\mathbf E')$ for some other $\mathbf E' \in \mathcal E(\mathbf E_k)$, a contradiction. Thus we have $E_* \notin \salt{\mathbf J}(\mathbf E)$ for any $\mathbf E \in \mathcal E(\mathbf E_s)$.

Now let $\mathbf E \in \mathcal E(\mathbf E_k)$ minimize $\dist(E_*, \salt{\mathbf J}(\mathbf E))$ among all $\mathbf E \in \mathcal E(\mathbf E_k)$, let $\mathbf I$ be its domain, and let $\mathbf J$ be its image. By the above comments, this distance is nonzero; thus, $E_c := \arg\min_{E \in \mathbf J}|E_*-E|$ is an extremum of $\mathbf E$. But by item 3 of Lemma \ref{l:jrelations}, if this extremum were not a critical point, there would be some $\mathbf E' \in \mathcal E(\mathbf E_k)$ with $\dist(E_*, \salt{\mathbf J}(\mathbf E')) < \dist(E_*, \salt{\mathbf J}(\mathbf E))$, which is not the case; thus, $E_c$ is a critical point of $\mathbf E$.  Furthermore, since $E_* \notin \salt{\mathbf J}(\mathbf E)$, $\dist(E_*,\mathbf J)>\rho_{k+1}$. It follows that for any $m \in \mathbb Z$, either $|\mathbf E(\theta_* + m\alpha)-E_*|>\rho_{k+1}$, or $\theta_*+m\alpha \notin \mathbf I$. 
Thus, all the hypotheses of Lemma \ref{l:NRintadj} are met.
\end{proof}

As a consequence of this result, we can separate the energies in $G$ from generalized eigenvalues of $H$:

\begin{lem}\label{l:gapsG}
If $E_* \in G$, then there is some $\rho>0$ such that for $|E-E_*|<\rho$, $E$ is not a generalized eigenvalue of $H(\theta)$ corresponding to a generalized eigenvector $\psi$ growing at most quadratically ($|\psi(j)|<C(j^2+1)$).
\end{lem}
\begin{proof}
By Lemma \ref{l:longNR}, there is some scale $k\geq0$ %
for which we have $(\ell_k, \gamma_k)$ Green's function decay on arbitrary long intervals; i.e., for $|E-E_*|<\rho_k/2$, 
\begin{equation*}
\log|R^{\tilde{\Lambda}}_{E,\theta}(m,n)| \leq -\gamma_{k}|m-n|, \quad |m-n| \geq \ell_{k},
\end{equation*}
for arbitrarily long intervals $\tilde{\Lambda}$. Denote $\tilde\Lambda=[\tilde a,\tilde b]$. 

Assume $E$ is generalized eigenvalue of $H(\theta)$ corresponding to a generalized eigenvector $\psi$ growing at most quadratically ($|\psi(j)|<C(j^2+1)$). Fix $j \in \mathbb Z$, and choose $\tilde{\Lambda}\ni j$ large enough so $\min\{|\tilde a - j|,|\tilde b - j|\}\geq|\tilde\Lambda|/3\geq\ell_k$. By the Poisson formula \eqref{eq:poisson}, 
\begin{align*}
|\psi(j)| &\leq e^{-\gamma_k|\tilde\Lambda|/3}\left(|\psi(\tilde a-1)|+|\psi(\tilde b+1)|\right) \\
&\leq 2Ce^{-\gamma_k|\tilde\Lambda|/3}((|j|+|\tilde\Lambda|)^2+1)
\end{align*}
As $|\tilde\Lambda|\to\infty$, the right-hand side approaches $0$, so we must have $|\psi(j)|=0$. Since this is true for arbitrary $j$, $\psi$ must be identically zero, i.e., $E$ must not be a generalized eigenvalue.
\end{proof}

Combining the above results with the sets' definitions, we arrive at the characterization of $\spec H(\theta)$:

\begin{lem}\label{l:specchar}
\begin{equation*}
\spec H(\theta) = \mathbb R \setminus G = \Sigma.
\end{equation*}
\end{lem}
\begin{proof}
By Lemma \ref{l:specS}, $\Sigma \subset \spec H(\theta)$; and since generalized eigenvalues are dense in $\spec H(\theta)$ by Schnol's Lemma, Lemma \ref{l:gapsG} implies that $\spec H(\theta) \subset \mathbb R \setminus G$. It remains to show that $\mathbb R \setminus G \subset \Sigma$. Let $E \in \mathbb R \setminus G$. The definition of $G$ ensures that $E \in \salt{\salt{\mathbf J}}(\mathbf E_0)$, and if $E \in \salt{\salt{\mathbf J}}(\mathbf E_k)$, then $E \in \salt{\salt{\mathbf J}}(\mathbf E_{k+1})$ for some $\mathbf E_{k+1} \in \mathcal{E}(\mathbf E_k)$. But this means precisely that $E \in \Sigma$; thus $\mathbb R \setminus G \subset \Sigma$.
\end{proof}

\subsection{Cantor spectrum}

Having characterized the spectrum, we now demonstrate that $\Sigma$ is a Cantor set, i.e. it is a (closed) nowhere-dense set without isolated points.  To do so, we demonstrate that the spectral gaps $G$ always meet our modified codomains:

\begin{lem} \label{l:Jgap}
For any $\mathbf E \in \mathcal E$, ${\salt{\mathbf J}}(\mathbf E) \cap G$ is nonempty.
\end{lem}
\begin{proof}
We proceed by contradiction: suppose there is some $\mathbf E_k \in \mathcal E_k$ such that ${\salt{\mathbf J}}(\mathbf E_k) \cap G = \varnothing$. 

We first note that $\mathbf E_k$ cannot contain a critical point. Suppose it attained its minimum at a critical point; then we would have
\begin{align*}
\inf \salt{\salt{\mathbf J}}(\mathbf E_k) &= 
\inf \mathbf \Ima \mathbf E_k - \rho_k \\&< \inf_{\mathbf E \in \mathcal E(\mathbf E_k)} \inf \Ima \mathbf E - \rho_{k+1} \\&= \inf_{\mathbf E \in \mathcal E(\mathbf E_k)} \inf \salt{\salt{\mathbf J}}(\mathbf E).
\end{align*}
Then $\salt{\salt{\mathbf J}}(\mathbf E_k) \setminus \bigcup_{\mathbf E \in \mathcal E(\mathbf E_k)} \salt{\salt{\mathbf J}}(\mathbf E) \subset {\salt{\mathbf J}}(\mathbf E_k) \cap G$ is nonempty, contradicting our initial assumption. 
Analogously, if $\mathbf E_k$ attains its maximum at a critical point, we reach the same contradiction. Thus, $\mathbf E_k$ does not contain a critical point.

Let $n \in \mathcal N(\alpha, \mathbf E_k)$ such that $E_n \in \salt{\salt{\mathbf J}}(\mathbf E_k)$ and $\dist(E_n,\partial\salt{\salt{\mathbf J}}(\mathbf E_k))\geq \rho_{k-1}/12$; by irrationality of $\alpha$, such an $n$ exists. We will show that for some $\mathbf E_{k+1} \in \mathcal E(\mathbf E_k)$, $n \in \mathcal N(\alpha, \mathbf E_{k+1})$ and $E_n(\mathbf E_{k+1}) \in \salt{\salt{\mathbf J}}(\mathbf E_{k+1})$ with $\dist(E_n(\mathbf E_{k+1}),\partial\salt{\salt{\mathbf J}}(\mathbf E_{k+1}))\geq \rho_{k}/12$. 
First note that if $E_n \in J_{k,n_0}^{(2)}$ for some $n_0 \in \mathcal N_{k+1}$, then, by construction, there are a pair of double-resonant Rellich functions $\mathcal E(\mathbf E_k) \ni \mathbf E_{k+1,\bullet} : \mathbf I^{(2)}_{k+1,n,\bullet} \to \mathbf J^{(2)}_{k+1,n,\bullet}\,(\bullet\in\{\vee,\wedge\})$ with $\mathbf J^{(2)}_{k+1,n,\bullet} \subset \salt{\salt{\mathbf J}}(\mathbf E_k)$; since the uniform local separation between double-resonant Rellich pairs guaranteed by Proposition \ref{pr:unifmLocalSep} is larger than $2\rho_{k+1}$, %
there is some $E \in (\sup \mathbf J^{(2)}_{k+1,n,\wedge}+\rho_{k+1},\inf \mathbf J^{(2)}_{k+1,n,\vee}-\rho_{k+1})$, which must then satisfy $$E \in \salt{\salt{\mathbf J}}(\mathbf E_k) \setminus \bigcup_{\mathbf E \in \mathcal E(\mathbf E_k)} \salt{\salt{\mathbf J}}(\mathbf E) \subset \salt{\mathbf J}(\mathbf E_k) \cap G,$$ contradicting our assumption. 
Thus $E_n$ is contained in the images of only simple-resonant functions $\mathbf E_{k+1} \in \mathcal E(\mathbf E_k)$. 
By %
Lemma \ref{l:jrelations} item 4, 
we can find some such function with $B_{\rho_k/8}(E_n) \subset \salt{\salt{\mathbf J}}(\mathbf E_k)$. 
Let $E_\pm  = E_n \pm \rho_k/24$. Then, since $\partial_E T_{\mathbf E_k} \geq 2/D_0$ by the inverse function theorem, %
\begin{align*}
    T_{\mathbf E_{k+1}}(E_+) &\geq T_{\mathbf E_k}(E_+) - \frac{4\delta_s}{\nu_s} \\ &\geq T_{\mathbf E_k}(E_n) + \frac{\rho_k}{12D_0} - \frac{4\delta_s}{\nu_s} \\&\geq T_{\mathbf E_k}(E_n)
\end{align*}
where the first line follows from Lemma \ref{l:difinv}. Similarly,
\begin{align*}
    T_{\mathbf E_{k+1}}(E_-) &\leq T_{\mathbf E_k}(E_n).
\end{align*}
By continuity of $T_{\mathbf E_{k+1}}$, there is some $E_* \in [E_-,E_+]$ with $T_{\mathbf E_{k+1}}(E_*) = T_{\mathbf E_k}(E_n)$; by definition, this means $n \in \mathcal N(\alpha,\mathbf E_{k+1})$ with $E_n(\mathbf E_{k+1}) = E_*$. By construction, $E_* \in \salt{\salt{\mathbf J}}(\mathbf E_{k+1})$ with $\dist(E_*,\partial\salt{\salt{\mathbf J}}(\mathbf E_{k+1})) \geq \rho_{k+1}/12$.

But $\salt{\mathbf J}(\mathbf E_{k+1}) \subset \salt{\mathbf J}(\mathbf E_k)$ by Lemma \ref{l:jrelations}, 
so we can repeat the above procedure to find $\mathbf E_{k+2} \in \mathcal E(\mathbf E_{k+1})$ with $n \in \mathcal N(\alpha,\mathbf E_{k+2})$ satisfying the above conditions. 
As we continue to repeat this procedure, we will eventually reach some scale $s$ with $|n| \leq 8M_{s+1}^{(1)}$; since $n \in \mathcal N(\alpha, \mathbf E_s)$, we will have $n \in \mathcal N_{s}$, and we will construct a double-resonant interval $J_{s,n}^{(2)}$ with $E_n(\mathbf E_s) \in J_{s,n}^{(2)}$. As we have shown above, this leads to a contradiction with our original assumption.
\end{proof}

\begin{proof}[Proof of Cantor spectrum]
It suffices to show that for each $E \in \spec H(\theta)$, there are points arbitrarily close to $E$ which are not in $\spec H(\theta)$. Fix $E \in \spec H(\theta)$. By Lemmas \ref{l:Jto0} and \ref{l:specchar}, we can find some $\mathbf E \in \mathcal E$ with $\salt{\mathbf J}(\mathbf E) \supset \salt{\salt{\mathbf J}}(\mathbf E) \ni E$ and $|{\salt{\mathbf J}}(\mathbf E)|$ arbitrarily small. By Lemma \ref{l:Jgap}, $\salt{\salt{\mathbf J}}(\mathbf E) \ni E'$ for some $E' \in G$; by Lemma \ref{l:specchar}, $E' \notin \spec H(\theta)$.
\end{proof}

\subsection{Localization}

We first construct the full-measure set on which localization occurs. Using the definitions preceding  Lemma \ref{l:DRnubd}, we define the double-resonant regions $$J^{DR}_n(\overline\rho_{k}, 12M_{k+1}^{(2)}, \alpha, \mathbf E) := \overline{B}_{\overline\rho_{k}}(E_n(\mathbf{E})), \quad n \in \mathcal{N}(12M_{k+1}^{(2)},\alpha,\mathbf{E})$$ of each function $\mathbf E \in \mathcal E_k$, and we define $$J^{DR} = \bigcup_{n \in \mathcal N(12M_{k+1}^{(2)}, \alpha, \mathbf E)}J^{DR}_n.$$  
We now define and measure the sets
\begin{equation*}
    \mathrm{B}_k := \bigcup_{\mathbf{E} \in \mathcal{E}_k} \bigcup_{|n|\leq 12M_{k+1}^{(2)}} \left(\mathbf{E}^{-1}(J^{DR})+n\alpha\right)
\end{equation*}
on which two resonances at scale $k$ occur at nearby shifts, and define 
\begin{equation*}
\mathrm{B} := \bigcap_{n=0}^\infty \bigcup_{k=n}^\infty \mathrm{B}_k.
\end{equation*}
\begin{lem} \label{l:badsets}
$|\mathrm{B}|=0$. 
\end{lem}
\begin{proof}
By construction, for any $\mathbf{E} \in \mathcal E_k$, $|\mathcal E(\mathbf E)| \leq 10\lceil \bar\rho_k^{-1} \rceil \ll \rho_{k-1}^{-3}$. 
Thus, 
for $k\geq1$, 
\begin{equation}\label{eq:counting}
    |\mathcal{E}_k| \leq |\mathcal{E}_0|\prod_{i=0}^{k-1}\rho_{i-1}^{-3} \leq \rho_{k-2}^{-4}.
\end{equation}

Fix $\mathbf E \in \mathcal E_k$. By definition, each $|J_n^{DR}| = 2\overline\rho_{k}$; by Lemma \ref{l:morsefnbd}, $|\mathbf E^{-1}(J_n^{DR})| \leq 2\sqrt{\frac{12}{d_k}\overline\rho_{k}} \ll \rho_{k-1}^{1/2}$.
Then $|\mathbf{E}^{-1}(J^{DR})| = \left|\sum_{n \in \mathcal N(12M_{k+1}^{(2)}, \alpha, \mathbf E)}\mathbf{E}^{-1}(J^{DR}_n)\right| \leq 25M_{k+1}^{(2)}\rho_{k-1}^{1/2}$, and $\left|\bigcup_{|n|\leq 12M_{k+1}^{(2)}} \left(\mathbf{E}^{-1}(J^{DR})+n\alpha\right)\right| \leq 625(M_{k+1}^{(2)})^{2}\rho_{k-1}^{1/2}$. 

Combining this with the above estimate \eqref{eq:counting}, we have $|\mathrm{B}_k| \leq 625(M_{k+1}^{(2)})^{2}\rho_{k-2}^{-4}\rho_{k-1}^{1/2}$, and $\sum_{k=1}^\infty|\mathrm{B}_k| \leq 625\sum_{k=0}^\infty(M_{k+1}^{(2)})^{2}\rho_{k-2}^{-4}\rho_{k-1}^{1/2} < \infty$. 
Thus, by the Borel-Cantelli Lemma, $|\mathrm{B}|=0$.
\end{proof}

Thus $\Theta := \mathbb T \setminus \mathrm{B}$ is a set of full measure. For the remainder of this section, we fix $\theta \in \Theta$.

To show that $H(\theta)$ is Anderson localized, it suffices by Schnol's lemma \cite{Schnol1981} to show that every generalized eigenvector $\psi$ of $H(\theta)$ that grows at most quadratically ($|\psi(j)| < C(j^2+1)$) in fact decays exponentially. 
For the remainder of this section, 
we fix a generalized eigenvalue $E(\theta)$ and corresponding generalized eigenvector $\psi$. By Lemma \ref{l:specchar}, we can also fix a sequence $\{\mathbf E_k\}_{k=0}^\infty$, with $\mathbf E_{k+1} \in \mathcal{E}(\mathbf E_{k})$ for all $k\geq0$, such that
$E \in \bigcap_{k=0}^\infty \salt{\salt{\mathbf J}}(\mathbf E_k)$. %

To locate the center of localization for $\psi$, we first employ an argument similar to Lemma \ref{l:gapsG}:
\begin{lem} \label{l:evecseesres}
There is a $K \in \mathbb N$ such that for all $k \geq K$, there is an $m_k \in \mathbb Z$ satisfying $|m_k| \leq 3M_{k}^{(2)}$ such that $\theta+m_k\alpha \in \mathbf I_k$ and $|E-\mathbf E_k(\theta+m_k\alpha)|<\rho_k$.
\end{lem}
\begin{proof}
If not, there is an increasing integer sequence $\{k_i\}_{i=0}^\infty$ such that $[-3M_{k_i}^{(2)},3M_{k_i}^{(2)}]$ satisfies the hypotheses of Lemma \ref{l:NRintadj} at scale $k_i$. Thus, for each $i$, there is an interval $\Lambda_i := [a_i,b_i] \supset [-\frac23M_{k_i}^{(2)},\frac23M_{k_i}^{(2)}]$ which satisfies the Green's function decay property for $(\ell_{k_i},\gamma_{k_i})$. For any $|n|\leq\frac13M_{k_i}^{(2)}$, $|n-a_i|,|n-b_i|\geq\frac13M_{k_i}^{(2)}\geq\ell_{k_i}$; thus the Poisson formula \eqref{eq:poisson} gives 
\begin{align*}
|\psi(n)| &\leq e^{-\gamma_{k_i}M_{k_i}^{(2)}/3}\left(|\psi(a_i-1)|+|\psi(b_i+1)|\right) \\
&\leq 2Ce^{-\gamma_{0}M_{k_i}^{(2)}/24}((3M_{k_i}^{(2)}+1)^2+1)
\end{align*}
Fixing $n \in \mathbb Z$, we have $|n|\leq\frac13M_{k_i}^{(2)}$ as $i\to\infty$, and the right-hand side of the above inequality approaches $0$. Thus we must have $|\psi(n)|=0$. But $\psi$ cannot be identically zero.
\end{proof}

The next lemma confirms that $m_k$ is unique at scale $12M_{k+1}^{(2)}$, and is in fact (eventually) independent of $k$, for $\theta \in \Theta$.

\begin{lem} \label{l:uniqres}
There is a $K \in \mathbb N$ and $m_\infty \in \mathbb Z$, with $|m_\infty| \leq 3M_{K}^{(2)}$, such that for all $k\geq K$, 
\begin{equation} \label{eq:rescond}
\theta+m_\infty\alpha \in \mathbf I_k,\text{ and } |E-\mathbf E_k(\theta+m_\infty\alpha)|<\rho_k,
\end{equation}
and $[m_\infty-12M_{k+1}^{(2)},m_\infty-1] \cup [m_\infty+1,m_\infty+12M_{k+1}^{(2)}]$ is $k$-nonresonant (relative to $(\theta,E)$ and $\mathbf E_k$).
\end{lem}
\begin{proof}
By Lemma \ref{l:evecseesres}, for sufficiently large $k$, we can choose $|m_k| \leq 3M_{k}^{(2)}$ satisfying \eqref{eq:rescond} (with $m_k$ replacing $m_\infty$). If $\theta \in \Theta$, then $\mathbf E_k(\theta+m\alpha) \notin J^{DR}$ for all $|m|\leq 12M_{k+1}^{(2)}$, and thus $\theta+m_k\alpha$ is in the simple resonance case of Proposition \ref{pr:SRDRsep} (with $12M_{k+1}^{(2)}$ in the role of $L^{(1)}$), for sufficiently large $k$; thus, for $0 < |m-m_k|\leq 12M_{k+1}^{(2)}$, if $\theta+m\alpha\in \mathbf I_k$, then $$|E-\mathbf{E}(\theta+m\alpha)| \geq |\mathbf{E}(\theta+m_k\alpha)-\mathbf{E}(\theta+m\alpha)| - |E - \mathbf{E}(\theta+m_k\alpha)| > 2\rho_k-\rho_k=\rho_k;$$ i.e., $[m_k-12M_{k+1}^{(2)},m_k-1] \cup [m_k+1,m_k+12M_{k+1}^{(2)}]$ is $k$-nonresonant.

It remains only to show that $m_k$ must be independent of $k$. Suppose $m_{k+1} \neq m_k$.
Note that $|m_{k+1}| \leq 3M_{k+1}^{(2)}$, so $|m_{k+1}-m_k|\leq 12M_{k+1}^{(2)}$. Since $$\theta+m_{k+1}\alpha \in \mathbf I_{k+1}\subset \mathbf I_k$$ and $[m_k-12M_{k+1}^{(2)},m_k-1] \cup [m_k+1,m_k+12M_{k+1}^{(2)}]$ is $k$-nonresonant, we have $|E-\mathbf E_k(\theta+m_{k+1}\alpha)|>\rho_k$. 
Furthermore, since $$\theta +m_{k+1}\alpha \notin J^{DR}(\overline\rho_{k}, 12M_{k+1}^{(2)}, \alpha, \mathbf E_k) \supset J^{DR}(\overline\rho_{k}, 8M_{k+1}^{(1)}, \alpha, \mathbf E_k),$$ we must have $\mathbf E_{k+1} \in \mathcal E^{(1)}(\mathbf E_k)$; thus, by Proposition \ref{pr:interscaleapprox1}, $$|\mathbf E_{k+1}(\theta+m_{k+1}\alpha)-\mathbf E_k(\theta+m_{k+1}\alpha)| \leq 2\delta_k \ll \rho_k/2$$. Thus,
\begin{align*}
|E-\mathbf E_{k+1}(\theta+m_{k+1}\alpha)| &\geq |E-\mathbf E_k(\theta+m_{k+1}\alpha)| - |\mathbf E_{k+1}(\theta+m_{k+1}\alpha)-\mathbf E_k(\theta+m_{k+1}\alpha)| \\&> \rho_k - \rho_k/2 \gg \rho_{k+1}. 
\end{align*}
In particular, this means $\{m_{k+1}\}$ is $k+1$-nonresonant, %
which contradicts the construction of $m_{k+1}$; we must therefore have $m_{k+1}=m_k$, and consequently $m_k$ is in fact independent of $k$, and we may set $m_\infty = m_k$ for all sufficiently large $k$.
\end{proof}

\begin{proof}[Proof of Anderson localization]
As noted above, we need only show that $\psi$ decays exponentially. 
Let $K$ and $m_\infty$ be the integers defined in Lemma \ref{l:uniqres}. 
For $|n-m_\infty|> 8M_{K+1}^{(2)}$, let $k\geq K+1$ be such that $8M_k^{(2)} < |n-m_\infty| \leq 8M_{k+1}^{(2)}$. %
Consider the interval $$\Lambda=\left[n-\left\lceil \frac{|n-m_\infty|}2\right\rceil,n+\left\lceil \frac{|n-m_\infty|}2\right\rceil\right].$$ 
All $m \in \Lambda$ satisfy $0<|m-m_\infty| \leq 12M_{k+1}^{(2)}$. 
Thus, by Lemma \ref{l:uniqres}, $\Lambda$ satisfies the hypotheses of Lemma \ref{l:NRintadj} at scale $k$; 
therefore, there is some subinterval
\begin{align*}
    \tilde\Lambda := [\tilde a, \tilde b] &\supset \left[n-\left\lceil \frac{|n-m_\infty|}2\right\rceil+3M_k^{(2)},n+\left\lceil \frac{|n-m_\infty|}2\right\rceil-3M_k^{(2)}\right] \\ &\supset \left[n-\left\lceil \frac{|n-m_\infty|}9\right\rceil,n+\left\lceil \frac{|n-m_\infty|}9\right\rceil\right]
\end{align*} 
which satisfies the Green's function decay property for $(\ell_k,\gamma_k)$. 
Note that $\left\lceil \frac{|n-m_\infty|}9\right\rceil > \frac89M_k^{(2)} \geq \ell_k$; thus, by the Poisson formula \eqref{eq:poisson},
\begin{align*}
    |\psi(n)| &\leq e^{-|\log\varepsilon||n-m_\infty|/9}\left(|\psi(\tilde a-1)|+|\psi(\tilde b+1)|\right).
\end{align*}
We have 
\begin{align*}
\max\{|\tilde a -1|, |\tilde b + 1|\} &\leq |n| + \left\lceil\frac{|n-m_\infty|}2\right\rceil + 1 \\
&\leq |n-m_\infty| + \frac{|n-m_\infty|}2 + 2 + 3M_K^{(2)} \\
&\leq \frac32|n-m_\infty| + M_k^{(2)} \\
&< 2|n-m_\infty|.
\end{align*}
Therefore, 
\begin{align*}
    |\psi(n)| &\leq 4Ce^{-|\log\varepsilon||n-m_\infty|/9}(|n-m_\infty|^2+1),
\end{align*}
which decays exponentially away from $m_\infty$.
\end{proof}

\newpage

\appendix

\section{Bootstrapped Green's function decay}

In this appendix, we will show that, if one can partition an interval $\snext\Lambda$ into subintervals alternating between long intervals with Green's function decay and short intervals with resolvent bounds, one can iterate the resolvent identity to bootstrap the decay to all of $\snext\Lambda$.  Results of this variety are well-established in the literature (cf. e.g. \cite{FroSpe83}); we include the details relative to our specific application for the reader's convenience.

Let $v \in C^2(\mathbb{T}, [-1,1])$ with $\|\partial_\theta v\|_\infty+\|\partial_\theta^2 v\|_\infty \leq D_0$, $(\theta_*, E_*) \in \mathbb{T} \times \mathbb{R}$, $0 < \varepsilon < 1/7$, and $0<{\snext\rho}<1/2$.  Additionally, let $\ell \in \mathbb N$, $1 < {\gamma} \leq |\log\varepsilon|$, and $$M:=\max\{{\ell},|\log{\snext\rho}|\}.$$  Consider an integer $\snext\ell \in \mathbb{N}$ with $\snext\ell > 16|\log\varepsilon|M$, and define 
\begin{align*}
	\snext\gamma &= {\gamma} - 16|\log\varepsilon|M/\snext\ell .
\end{align*}

\begin{assm}\label{as:NRmult}
	Suppose that $\snext\Lambda \subset \mathbb{Z}$ admits a partition $\mathcal{P}$ satisfying:
	\begin{enumerate}
		\item Every other interval of $\mathcal{P}$ satisfies the Green's function decay property for $({\ell},{\gamma})$. We denote the intervening intervals by $\Lambda_i$ and use the notation $\Lambda_{i,l/r}$ to refer to the interval (on which the Green's function decay property for $({\ell},{\gamma})$ is satisfied) immediately to the left/right of $\Lambda_i$.
		
		Note that, with this convention, one will have, e.g., ${\Lambda}_{i,r} = {\Lambda}_{i+1,l}$.
		\item For each ${\Lambda}_i = [ a_i,  b_i]$, 
		the interval $\snext{\Lambda}_i :=\Lambda_{i,l} \cup \Lambda_i \cup \Lambda_{i,r}$ satisfies
		\begin{align}\label{eq:NRspecsepmult}
			\|{P}_{i}  R^{\snext{\Lambda}_i} \Gamma_{\mathcal P}\| \leq 2{\snext\rho}^{-1},
		\end{align}
		where ${P}_{i} $ denotes the projection onto coordinates $[ a_i-{\ell},  b_i+{\ell}] \cap \snext{\Lambda}_i$. 
		Furthermore, $R^{\snext{\Lambda}_i}$ is well-defined. 
		\item $|\Lambda_{i,l/r}|> 2\snext\ell$.
		\item $|\Lambda_i|\leq M$.
	\end{enumerate}
\end{assm}

We illustrate one segment of Assumption \ref{as:NRmult} in Figure \ref{f:NRmult}.  Under these assumptions, we get $(\snext\ell,\snext\gamma)$ decay on $\snext\Lambda$:

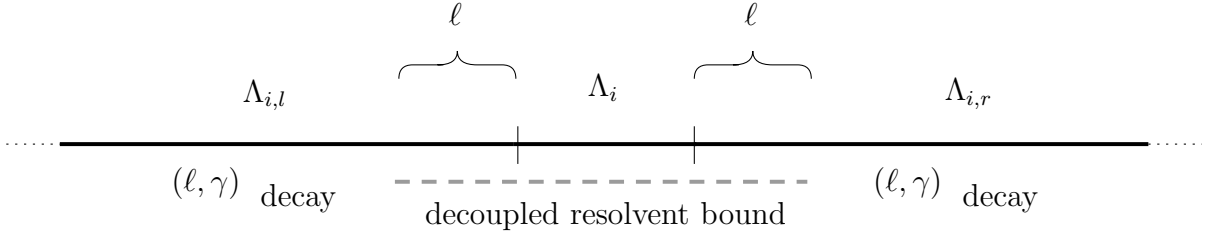
\begin{figure}
	\begin{tikzpicture}[x=0.75pt,y=0.75pt,yscale=-1,xscale=1]
		
		\draw [line width=1.5]    (51,93) -- (280,93) ;
		\draw [color={rgb, 255:red, 155; green, 155; blue, 155 }  ,draw opacity=1 ][line width=1.5]  [dash pattern={on 5.63pt off 4.5pt}]  (220,112) -- (428,112) ;
		\draw   (281,60) .. controls (281,55.33) and (278.67,53) .. (274,53) -- (261.5,53) .. controls (254.83,53) and (251.5,50.67) .. (251.5,46) .. controls (251.5,50.67) and (248.17,53) .. (241.5,53)(244.5,53) -- (229,53) .. controls (224.33,53) and (222,55.33) .. (222,60) ;
		\draw [line width=1.5]    (280,93) -- (372,93) ;
		\draw    (371,84) -- (371,103) ;
		\draw    (282,84) -- (282,103) ;
		\draw   (430,60) .. controls (430,55.33) and (427.67,53) .. (423,53) -- (410.5,53) .. controls (403.83,53) and (400.5,50.67) .. (400.5,46) .. controls (400.5,50.67) and (397.17,53) .. (390.5,53)(393.5,53) -- (378,53) .. controls (373.33,53) and (371,55.33) .. (371,60) ;
		\draw [line width=1.5]    (371,93) -- (600,93) ;
		\draw  [dash pattern={on 0.84pt off 2.51pt}]  (24,93) -- (51,93) ;
		\draw  [dash pattern={on 0.84pt off 2.51pt}]  (600,93) -- (627,93) ;
		
		\draw (142,59.4) node [anchor=north west][inner sep=0.75pt]    {${\Lambda }_{i,l}$};
		\draw (496,58.4) node [anchor=north west][inner sep=0.75pt]    {${\Lambda }_{i,r}$};
		\draw (246,20.4) node [anchor=north west][inner sep=0.75pt]    {${\ell }$};
		\draw (106,103.4) node [anchor=north west][inner sep=0.75pt]    {$\left({\ell } ,{\gamma }\right)$};
		\draw (151,112) node [anchor=north west][inner sep=0.75pt]   [align=left] {decay};
		\draw (460,103.4) node [anchor=north west][inner sep=0.75pt]    {$\left({\ell } ,{\gamma }\right)$};
		\draw (505,112) node [anchor=north west][inner sep=0.75pt]   [align=left] {decay};
		\draw (316,56.4) node [anchor=north west][inner sep=0.75pt]    {${\Lambda }_{i}$};
		\draw (234,121) node [anchor=north west][inner sep=0.75pt]   [align=left] {decoupled resolvent bound};
		\draw (395,20.4) node [anchor=north west][inner sep=0.75pt]    {${\ell }$};

	\end{tikzpicture}
	\caption{An illustration of one segment $\snext\Lambda_i$ of the partition $\mathcal{P}$ from Assumption \ref{as:NRmult}; note that the naming convention for the decay intervals is non-unique, i.e. $\Lambda_{i,r} = \Lambda_{i+1,l}$.}
	\label{f:NRmult}
\end{figure}
\begin{lem} \label{l:NRmultdecay}
	Suppose $\snext\Lambda$ satisfies Assumption \ref{as:NRmult}. Then $\snext\Lambda$ satisfies the Green's function decay property for $(\snext\ell,\snext\gamma)$. 
\end{lem}

Before proving Lemma \ref{l:NRmultdecay}, we will first handle the special case where there is precisely one interval $\Lambda_i$:
\begin{assm}\label{as:NR}
	Suppose that $\snext\Lambda$ is partitioned by  
	$\mathcal{P}$ into ${\Lambda}_l \cup {\Lambda} \cup {\Lambda}_r$, such that
	\begin{enumerate}
		\item ${\Lambda}_{l/r}$ satisfies the Green's function decay property for $({\ell},{\gamma})$.
		\item 
		$R^{\snext\Lambda}$ is well-defined, and,  
		denoting $\Lambda=[ a,  b]$,
		\begin{align}\label{eq:NRspecsep}
			\|{P}  R^{\snext\Lambda} \Gamma_{\mathcal P}\| \leq 2{\snext\rho}^{-1},
		\end{align}
		where ${P} $ denotes the projection onto coordinates $[ a-{\ell},  b+{\ell}] \cap \snext\Lambda$.
		\item $|\Lambda|\leq M$.
	\end{enumerate}
\end{assm}

Letting $$\salt\gamma:={\gamma}-6|\log\varepsilon|M/\snext\ell,$$ in this case we get slightly improved decay on $\snext{\Lambda}$:

\begin{lem}\label{l:NRdecay}
	Suppose $\snext\Lambda$ satisfies Assumption \ref{as:NR}. Then $\snext\Lambda$ satisfies the Green's function decay property for $(\snext\ell,\salt\gamma)$. 
\end{lem}
\begin{proof}
	Suppose $m,n \in \snext\Lambda$ and $|m-n| \geq \snext\ell \geq 3M \geq 2{\ell}+|\Lambda|$. 
	Then at most one of $m,n$ satisfies $\dist(\cdot,{\Lambda}) \leq {\ell}$.  In the case neither $m$ nor $n$ is within ${\ell}$ of ${\Lambda}$, we use the resolvent expansion 
	\begin{equation*}
		R^{\snext\Lambda} = R^{\snext\Lambda}_{\mathcal{P}} - R^{\snext\Lambda}_{\mathcal{P}}\Gamma_{\mathcal{P}} R^{\snext\Lambda} 
		= R^{\snext\Lambda}_{\mathcal{P}} - R^{\snext\Lambda}_{\mathcal{P}}\Gamma_{\mathcal{P}}R^{\snext\Lambda}_{\mathcal{P}} + R^{\snext\Lambda}_{\mathcal{P}}\Gamma_{\mathcal{P}}R^{\snext\Lambda}\Gamma_{\mathcal{P}}R^{\snext\Lambda}_{\mathcal{P}}.
	\end{equation*}
	Since $m,n$ do not belong to adjacent intervals in $\mathcal{P}$, 
	the middle term of $R^{\snext\Lambda}(m,n)$ in this expansion vanishes. By Green's function decay on ${\Lambda}_{l/r}$ and \eqref{eq:NRspecsep}, 
	\begin{align*}
		\log|R^{\snext\Lambda}(m,n)| &\leq \log\left(e^{-{\gamma}|m-n|} + 2\varepsilon{\snext\rho}^{-1}e^{-{\gamma}(|m-n|-|\Lambda|-1)}\right) \\
		&\leq \log\left(3{\snext\rho}^{-1}e^{-{\gamma}(|m-n|-|\Lambda|-2{\ell}-1)}\right)\\
		&= -\left({\gamma} - \frac{\log3+|\log{\snext\rho}|+{\gamma}|\Lambda|+2{\gamma}{\ell}+{\gamma}}{|m-n|}\right)|m-n|\\
		&\leq -\left({\gamma} - \frac{6|\log\varepsilon|M}{\ell}\right)|m-n|.
	\end{align*}
	
	On the other hand, if $\dist(n,\Lambda)\leq{\ell}$, we have $|m-n| \leq \dist(m,\Lambda)+|\Lambda|+2{\ell}$. 
	We use the resolvent expansion \begin{equation*}    R^{\snext\Lambda} = R^{\snext\Lambda}_{\mathcal{P}} - R^{\snext\Lambda}_{\mathcal{P}}\Gamma_{\mathcal{P}} R^{\snext\Lambda}.\end{equation*}
	By Green's function decay on ${\Lambda}_{l/r}$ and \eqref{eq:NRspecsep},
	\begin{equation*}
		\log|R^{\snext\Lambda}(m,n)| \leq \log\left(e^{-{\gamma}|m-n|}+2{\snext\rho}^{-1}e^{-{\gamma}(\dist(m,\Lambda)-1)}\right) \leq \log\left(3{\snext\rho}^{-1}e^{-{\gamma}(|m-n|-|\Lambda|-2{\ell}-1)}\right),
	\end{equation*}
	and the proof concludes as in the first case.
\end{proof}

The following lemma is useful to ensure that the decoupled resolvent bounds \eqref{eq:NRspecsepmult} or \eqref{eq:NRspecsep} are satisfied:
\begin{lem}\label{l:NRspecsep}
	Suppose $\snext\ell \in \mathbb{N}$ is an integer with $$\snext\ell\geq\max\left\{{\ell}, \frac{|\log{\snext\rho}|+1}{{\gamma}}\right\}.$$
	Suppose that $\snext\Lambda$ is partitioned by  
	$\mathcal{P}$ into ${\Lambda}_l \cup {\Lambda} \cup {\Lambda}_r$, such that
	\begin{enumerate}
		\item ${\Lambda}_{l/r}$ satisfies the Green's function decay property for $({\ell},{\gamma})$.
		\item There exists an interval $\snext\Lambda \supset \salt{\Lambda} = [\salt{a}, \salt{b}] \supset {\Lambda} = [{a}, {b}]$ such that
		\begin{align*}
			\min\{|\salt{a} - {a}|, |\salt{b} - {b}|\} &\geq \ell, \\
			\|R^{\salt{\Lambda}}\| \leq {\snext\rho}^{-1}.
		\end{align*}
	\end{enumerate}
	Then 
	$R^{\snext\Lambda}$ is well-defined, 
	\eqref{eq:NRspecsep} is satisfied, and $\snext\Lambda$ satisfies Assumption \ref{as:NR}.
\end{lem}

\begin{proof}
	Denote by $\salt{\mathcal{P}}$ the partition of $\snext\Lambda$ into $\salt{\Lambda}$ and its complement $\snext\Lambda \setminus \salt{\Lambda}$, and denote as shorthand $\salt{R} = R^{\snext\Lambda}_{\salt{\mathcal{P}}}$ and $\salt{\Gamma} = \Gamma^\Lambda_{\salt{\mathcal{P}}}$.
	
	We expand the resolvent $R^{\snext\Lambda}$ by alternating the resolvent formulas 
	\begin{align*}
		R^{\snext\Lambda} &= \salt R - \salt R\salt{\Gamma}R^{\snext\Lambda}, \\
		R^{\snext\Lambda} &= R^{\snext\Lambda}_\mathcal{P} - R^{\snext\Lambda}_\mathcal{P} \Gamma_\mathcal{P} R^{\snext\Lambda},
	\end{align*}
	to get the expansion
	\begin{equation} \label{eq:expR}
		R^{\snext\Lambda} = \salt R - \salt R\salt{\Gamma} R^{\snext\Lambda}_\mathcal{P} + \salt R\salt{\Gamma} R^{\snext\Lambda}_{\mathcal{P}} \Gamma_\mathcal{P} \salt R - \dots.
	\end{equation}
	Note that ${P}  \salt R = {P}  R^{\salt\Lambda}$, where ${P}$ is defined as in \eqref{eq:NRspecsep}; similarly, because the boundary points of $\mathcal P$ are in $\salt\Lambda$, $\Gamma_{\mathcal P}\salt R = \Gamma_{\mathcal P} R^{\salt\Lambda}$. Thus multiplying on the left 
	by ${P} $ gives
	\begin{align*}
		{P} R^{\snext\Lambda} &= {P} R^{\salt\Lambda} - {P} R^{\salt\Lambda}\salt{\Gamma} R^{\snext\Lambda}_\mathcal{P} + {P} R^{\salt\Lambda}\salt{\Gamma} R^{\snext\Lambda}_{\mathcal{P}} \Gamma_\mathcal{P} R^{\salt\Lambda} - \dots.
	\end{align*}
	By the Green's function decay on ${\Lambda}_{l/r} \supset \salt{\Lambda} \setminus {\Lambda}$, 
	\begin{align}
		\|\salt{\Gamma} R^{\snext\Lambda}_{\mathcal{P}}\Gamma_\mathcal{P}\| \leq \|\salt{\Gamma}R^{{\Lambda}_l}\Gamma_\mathcal{P}\| + \|\salt{\Gamma}R^{{\Lambda}_r}\Gamma_\mathcal{P}\| \leq \exp(-{\gamma} \ell),
	\end{align}
	and so we get that
	\begin{align*}
		\|{P} R^{\snext\Lambda} \Gamma_{\mathcal P}\|  &= \left\| {P} \sum_{k \geq 0} (R^{\salt\Lambda}\salt{\Gamma} R^{\snext\Lambda}_\mathcal{P} \Gamma_\mathcal{P})^k(R^{\salt\Lambda}-R^{\salt\Lambda}\salt{\Gamma}R^{\snext\Lambda}_\mathcal{P} )\Gamma_{\mathcal P}
		\right\| \\
		&\leq \sum_{k \geq 0}\left(\frac{e^{-{\gamma} \ell}}{{\snext\rho}}\right)^k \frac{2\varepsilon+e^{-{\gamma}\ell}}{{\snext\rho}} \\
		&\leq 2{\snext\rho}^{-1}.
	\end{align*}
	Using the same method to estimate the expansion \eqref{eq:expR} gives an upper bound on $\|R^{\snext\Lambda}\|$, which ensures $R^{\snext\Lambda}$ is well-defined.
\end{proof}

\begin{proof}[Proof of Lemma \ref{l:NRmultdecay}]
	Each interval $\snext\Lambda_i %
	= \Lambda_{i,l} \cup \Lambda_i \cup \Lambda_{i,r}$ satisfies the Green's function decay property for $(\snext\ell,\salt\gamma)$ by Lemma \ref{l:NRdecay}. Define the partitions $\mathcal{P}_1 := \bigcup_{i \text{ odd}} \snext\Lambda_i \cup \bigcup_{i \text{ even}} \Lambda_i$ and $\mathcal{P}_0 := \bigcup_{i \text{ even}} \snext\Lambda_i \cup 
	\bigcup_{i \text{ odd}} \Lambda_i$.
	For ease of notation, we define $\mathcal{P}_i = \mathcal{P}_0$
	whenever $i$ is an even integer, and $\mathcal{P}_i = \mathcal{P}_1$ whenever $i$ is an odd integer.
	
	We first show that $R^{\snext\Lambda}$ is well-defined. Partition $\snext\Lambda = \check\Lambda_0 \cup \check\Lambda_1$ such that $\bigcup_{i \text{ even}} \Lambda_i \subset \check\Lambda_0$ and $\bigcup_{i \text{ odd}} \Lambda_i \subset \check\Lambda_1$, and let $Q_0,Q_1$ be projections onto the coordinates $\check\Lambda_0,\check\Lambda_1$ respectively. Then, by Assumption \ref{as:NRmult}, we have upper bound estimates on $\|Q_0R^{\snext\Lambda}_{\mathcal P_0}\|$ and $\|Q_1R^{\snext\Lambda}_{\mathcal P_1}\|$, and we have $\|\Gamma_{\mathcal{P}_0}R^{\snext\Lambda}_{\mathcal{P}_0}\Gamma_{\mathcal{P}_1}\|,\|\Gamma_{\mathcal{P}_1}R^{\snext\Lambda}_{\mathcal{P}_1}\Gamma_{\mathcal{P}_0}\|\ll1$. Thus, we can write $R^{\snext\Lambda} = Q_0R^{\snext\Lambda} + Q_1R^{\snext\Lambda}$ and expand each term using an infinite resolvent expansion alternating between $\mathcal P_0$ and $\mathcal P_1$, the first term starting with $\mathcal P_0$ and the second term starting with $\mathcal P_1$. This will give an upper bound on $\|R^{\snext\Lambda}\|$, ensuring it is well-defined.
	
	Suppose $m,n \in \snext\Lambda$ and $|m-n| \geq \snext\ell$.
	Because the $\snext\Lambda_i$ cover $\snext\Lambda$ with overlaps 
	of size at least $2\snext\ell$, there must some $\snext\Lambda_i \ni m$ such that $m$ is separated from its boundary by at 
	least $\snext\ell$. Without loss of generality, let this be 
	$\snext\Lambda_0$.

	Suppose first that $\dist(n,\Lambda_i) > {\ell}$ for all $\Lambda_i$. We expand the resolvent first from the left using $\mathcal P_0$, then from the right using the partition $\mathcal P$ defined in Assumption \ref{as:NRmult}, and then from the left alternating between $\mathcal P_1$ and $\mathcal P_0$:
	\begin{equation}\label{eq:BREalt}
		R^{\snext\Lambda} = R^{\snext\Lambda}_{\mathcal{P}_0} - R^{\snext\Lambda}_{\mathcal{P}_0}\Gamma_{\mathcal{P}_0}R^{\snext\Lambda}_{\mathcal{P}} + R^{\snext\Lambda}_{\mathcal{P}_0}\Gamma_{\mathcal{P}_0}R^{\snext\Lambda}_{\mathcal{P}_1}\Gamma_{\mathcal{P}}R^{\snext\Lambda}_{\mathcal{P}} - R^{\snext\Lambda}_{\mathcal{P}_0}\Gamma_{\mathcal{P}_0}R^{\snext\Lambda}_{\mathcal{P}_1}\Gamma_{\mathcal{P}_1}R^{\snext\Lambda}_{\mathcal{P}_0}\Gamma_{\mathcal{P}}R^{\snext\Lambda}_{\mathcal{P}} + \dots.
	\end{equation}
	We now expand $R^{\snext\Lambda}(m,n)$ using \eqref{eq:BREalt}. Since $|m-n|\geq\snext\ell$, the first term is bounded by $e^{-\salt\gamma|m-n|}$ by Lemma \ref{l:NRdecay}. The second term gives a sum of terms of the form $\varepsilon R^{\snext\Lambda}_{\mathcal P_0}(m,p\pm1)R^{\snext\Lambda}_{\mathcal P}(p,n)$ where $p \in \Lambda_i$ for some $i$; then $R^{\snext\Lambda}_{\mathcal P}(p,n)=0$ because $n$ does not belong to any $\Lambda_i$. Subsequent terms are of the form
	\begin{equation} \label{eq:BREterm}
		(-1)^{k+1}\varepsilon^{k+1} R^{\snext\Lambda}_{\mathcal{P}_0}(p_0,q_0)R^{\snext\Lambda}_{\mathcal{P}_1}(p_1,q_1)\cdots R^{\snext\Lambda}_{\mathcal{P}_{k}}(p_k,q_k)R^{\snext\Lambda}_{\mathcal{P}}(p_{k+1},n)
	\end{equation}
	for $k\geq1$, where $p_0=m$, and for $0\leq i < k$, $q_i$ is a boundary point of the interval in $\mathcal{P}_{i}$ containing $p_i$, and $p_{i+1}=q_i\pm1$; $p_{k+1}$ is a boundary point of the interval in $\mathcal{P}$ containing $n$, and $q_k = p_{k+1}\pm1$.
	
	The structure of the intervals, our assumption on $m$, and 
	Lemma \ref{l:NRdecay} guarantee that $\log|R^{\snext\Lambda}_{\mathcal{P}_i}(p_i,q_i)|\leq-\salt\gamma|p_i-q_i|$ 
	for all $0 \leq i < k$. Similarly, Assumption \ref{as:NRmult} guarantees $\log|R^{\snext\Lambda}_{\mathcal{P}}(p_{k+1},n)|\leq-{\gamma}|p_{k+1}-n|$. To estimate $|R^{\snext\Lambda}_{\mathcal{P}_k}(p_k,q_k)|$, we note that $q_k$ must be near some $\Lambda_j$ (i.e., either $q_k \in \Lambda_j$ or $\dist(q_k,\Lambda_j)\leq{\ell}$), and we make the following considerations:
	\begin{itemize}
		\item If $k$ shares the same parity as $j$, then in order for $R^{\snext\Lambda}_{\mathcal{P}_k}(p_k,q_k)$ to be nonzero, we must have $p_k \in \Lambda_j \in \mathcal{P}_k$. 
		Then \eqref{eq:NRspecsep} applies, and $|p_k-q_k|\leq2{\ell}+|\log{\snext\rho}|\leq 3M$, giving $|R^{\snext\Lambda}_{\mathcal{P}_k}(p_k,q_k)|=|R^{\snext\Lambda_j}(p_k,q_k)| \leq 2\varepsilon^{-1}{\snext\rho}^{-1} \leq 2\varepsilon^{-1}{\snext\rho}^{-1}e^{3M\salt\gamma}e^{-\salt\gamma|p_k-q_k|}$.
		\item If $k$ and $j$ have opposite parities, then $|p_k-q_k| \geq 2\snext\ell-{\ell} \geq 5M \geq \snext\ell$. Thus either $R^{\snext\Lambda}_{\mathcal{P}_k}(p_k,q_k)=0$, or 
		Lemma \ref{l:NRdecay} applies, giving $\log|R^{\snext\Lambda}_{\mathcal{P}_k}(p_k,q_k)|\leq-\salt\gamma|p_k-q_k|$.
	\end{itemize}
	Finally, since $|m-n|\leq |p_{k+1}-n|+\sum_{i=0}^k|p_i-q_i|+k+1$, and 
	$\varepsilon^{k+1}\leq e^{-\salt\gamma (k+1)}$, each term of the form
	\eqref{eq:BREterm} is bounded by $2\varepsilon^{-1}{\snext\rho}^{-1}e^{3M\salt\gamma}e^{-\salt\gamma|m-n|}$. 
	By the same considerations, but using the fact that $|p_i-q_i|\geq\snext\ell$ for all $0\leq i<k$, each term of the form 
	\eqref{eq:BREterm} is bounded by $2\varepsilon^{-1}{\snext\rho}^{-1}e^{3M\salt\gamma}e^{-k\salt\gamma\snext\ell}$. We consider the first term $R^{\snext\Lambda}_{\mathcal{P}_0}(m,n)$ from \eqref{eq:BREalt} to be the 
	sole term of the expansion with $k=0$, and note that the same bounds apply.
	
	In the case that $n \in \Lambda_j$ or $\dist(n,\Lambda_j) \leq {\ell}$, we follow the same procedure but do not use $\mathcal P$, i.e., we use the expansion 
	\begin{equation}\label{eq:BREalt2}
		R^{\snext\Lambda} = R^{\snext\Lambda}_{\mathcal{P}_0} - R^{\snext\Lambda}_{\mathcal{P}_0}\Gamma_{\mathcal{P}_0}R^{\snext\Lambda}_{\mathcal{P}_1} + R^{\snext\Lambda}_{\mathcal{P}_0}\Gamma_{\mathcal{P}_0}R^{\snext\Lambda}_{\mathcal{P}_1}\Gamma_{\mathcal{P}_1}R^{\snext\Lambda}_{\mathcal{P}_0} - R^{\snext\Lambda}_{\mathcal{P}_0}\Gamma_{\mathcal{P}_0}R^{\snext\Lambda}_{\mathcal{P}_1}\Gamma_{\mathcal{P}_1}R^{\snext\Lambda}_{\mathcal{P}_0}\Gamma_{\mathcal{P}_0}R^{\snext\Lambda}_{\mathcal{P}_1} + \dots.
	\end{equation}
	to obtain terms of the form 
	\begin{equation} \label{eq:BREterm2}
		(-1)^{k}\varepsilon^{k} R^{\snext\Lambda}_{\mathcal{P}_0}(p_0,q_0)R^{\snext\Lambda}_{\mathcal{P}_1}(p_1,q_1)\cdots R^{\snext\Lambda}_{\mathcal{P}_{k}}(p_k,q_k)
	\end{equation}
	where $p_0=m$, $q_k=n$, and for $0\leq i < k$, $q_i$ is a boundary point of the interval in $\mathcal{P}_{i}$ containing $p_i$, and $p_{i+1}=q_i\pm1$. By considerations analogous to those in the first case, we obtain the same bounds $2\varepsilon^{-1}{\snext\rho}^{-1}e^{3M\salt\gamma}e^{-\salt\gamma|m-n|}$ and $2\varepsilon^{-1}{\snext\rho}^{-1}e^{3M\salt\gamma}e^{-k\salt\gamma\snext\ell}$ for each term of the form \eqref{eq:BREterm2}.
	
	Since each $p_i$ gives two choices for $q_i$, there 
	are at most $2^{k+1}$ nonzero terms of the form \eqref{eq:BREterm} or \eqref{eq:BREterm2} for 
	each nonzero value of $k$. Let $k_0 = \lceil\frac{|m-n|}{\snext\ell}\rceil$. For the at most $2^{k_0+1}$ terms with $k<k_0$, we use the bound $2{\snext\rho}^{-1}e^{3M\salt\gamma}e^{-\salt\gamma|m-n|}$, while we use the bound $2{\snext\rho}^{-1}e^{3M\salt\gamma}e^{-k\salt\gamma\snext\ell}$ for $k\geq k_0$, giving
	\begin{align*} 
		|R^{\snext\Lambda}(m,n)|&\leq 2\varepsilon^{-1}{\snext\rho}^{-1}e^{3M\salt\gamma}\left(2^{k_0+1}e^{-\salt\gamma|m-n|}+\sum_{k=k_0}^\infty 2^{k+1}e^{-k\salt\gamma\snext\ell}\right)\\
		&\leq 4\varepsilon^{-1}{\snext\rho}^{-1}e^{3M\salt\gamma}\left(2^{k_0}e^{-\salt\gamma|m-n|}+2^{k_0+1}e^{-\salt\gamma k_0\snext{\ell}}\right) \\
		&\leq 4\varepsilon^{-1}{\snext\rho}^{-1}e^{3M\salt\gamma}2^{k_0+2}e^{-\salt\gamma|m-n|}\\
		&\leq 2^{|m-n|/\snext\ell+5}\varepsilon^{-1}{\snext\rho}^{-1}e^{3M\salt\gamma}e^{-\salt\gamma|m-n|}.
	\end{align*}
	It follows that
	\begin{align*}
		\log|R^{\snext\Lambda}(m,n)|&\leq -\left(\salt\gamma - \frac{\log2}{\snext\ell}-\frac{5\log2+|\log\varepsilon|+|\log{\snext\rho}|+3M\salt\gamma}{|m-n|}\right)|m-n|\\
		&\leq-\left(\salt\gamma-\frac{10|\log\varepsilon|M}{\snext\ell}\right)|m-n| \\
		&\leq -\snext\gamma|m-n|.  \qedhere
	\end{align*}

\end{proof}

\newpage
\section{Proofs of preliminaries}

Below we collect proofs of the foundational results from the introduction.

\subsection{Perturbation theory lemmas}

\begin{proof}[Proof of Lemma \ref{l:PQclose}]
We proceed by contraposition.  Suppose $\rank(\chi) < \rank(P)$.  Then $\dim(\ker(\chi)) + \dim(\Ima(P)) > \dim(V)$ by rank-nullity.  Consequently, $\ker(\chi)\cap \Ima(P)$ contains a unit vector $P v$.  But then
\begin{align*}
\|(I-\chi)P v\| &= \|P v\| = 1,
\end{align*}
so $\|(I-\chi)P\| \geq 1$.
\end{proof}

\begin{proof}[Proof of Lemma \ref{l:approxevect}]
There must be an eigenvalue in $\overline B_\delta(E_*)$, since
\begin{align*}
\delta^2 \geq \|(A-E_*)\phi\|^2 = \|(A-E_*)\chi_\mathbb{R}(A)\phi\|^2 \geq \inf_{\lambda_j \in \, \spec A}|\lambda_j - E_*|^2.
\end{align*}
By assumption we have
\begin{align*}
\delta^2 \geq \|(A-E_*)\phi\|^2 &\geq \|(I-\chi_{ B_{\rho}(E_*)}(A))(A-E_*)\phi\|^2 \\
&\geq \rho^2\|(I-\chi_{ B_{\rho}(E_*)}(A))\phi\|^2 
\end{align*}
which is \eqref{eq:approxevect1}.  To see \eqref{eq:approxevect2}, notice that, since
\begin{align*}
\left\|(I - \chi_{ B_\rho(E_*)}(A))\phi\right\|^2 &= 1 - \|\chi_{ B_\rho(E_*)}(A)\phi\|^2 \leq (\delta/\rho)^2,
\end{align*}
it follows from $\chi_{ B_\rho(E_*)}(A) ^2 = \chi_{ B_\rho(E_*)}(A)$ that
\begin{align*}
\|\phi - \psi\|^2 &= 2(1 - \|\chi_{B_\rho(E_*)}(A)\phi\|) \\
&\leq 2(1 - \|\chi_{B_\rho(E_*)}(A)\phi\|^2) \\
&\leq 2(\delta/\rho)^2,
\end{align*}
which was the claim.
\end{proof}

\begin{proof}[Proof of Lemma \ref{l:feynman}]
The family of eigenpairs $(\mathbf{E},\psi)(\theta)$ in question are the implicit function defined by the vanishing of 
\begin{align*}
F(E,\phi,\theta) := \begin{bmatrix}
(A(\theta) - E)\phi \\
\langle \phi, \phi \rangle - 1
\end{bmatrix}.
\end{align*}
Indeed, since $E_*$ is simple, one has that
\begin{align*}
(D_{(E,\phi)} F) (E_*,\psi_*, \theta_*) = \begin{bmatrix}
A(\theta_*) - E_* & -\psi_* \\
2\psi_*^\top & 0
\end{bmatrix}
\end{align*}
is non-singular: denoting $P = \psi_* \psi_*^\top \oplus 1$ and $Q = I - P$, one computes via the Schur complement that
\begin{align*}
\det((D_{(E,\phi)} F) (E_*,\psi_*, \theta_*)) &= \det\left(\begin{bmatrix}
0 & -1 \\
2 & 0
\end{bmatrix}\right)\det\left(Q(A(\theta_*) - E_*)Q\right) \\
&= 2\prod_{\lambda_j \neq E_*} (\lambda_j - E_*) \neq 0.
\end{align*}
Thus the implicit function $(\mathbf{E},\psi)(\theta)$ is defined in a neighborhood of $\theta_*$ and twice-differentiable at $\theta_*$.

We differentiate the eigenvalue relation $(A-\mathbf{E})\psi = 0$ at $\theta_*$ to get that
\begin{align}\label{eq:evaleqnderiv}
(A' - \partial_\theta \mathbf{E})\psi = -(A-\mathbf{E})\partial_\theta\psi.
\end{align}
Equation \eqref{eq:feynman1} follows by taking inner products of the relation \eqref{eq:evaleqnderiv} with $\psi$; \eqref{eq:feynmanEvec} follows from the uniqueness of the implicit function since \eqref{eq:feynmanEvec} satisfies \eqref{eq:evaleqnderiv}; and \eqref{eq:feynman2} follows by differentiating \eqref{eq:feynman1}.  Finally, \eqref{eq:orthderiv} is immediate from \eqref{eq:feynmanEvec}.
\end{proof}

Finally, we have the simple proof of the Poisson formula:

\begin{proof}[Proof of Lemma \ref{l:poisson}]
	This follows by applying the resolvent $R^\Lambda(E)$ to the observation
	\begin{align*}
		(H^\Lambda - E)\psi &= \varepsilon\psi(a-1)\delta_a + \varepsilon\psi(b+1)\delta_b. \qedhere
	\end{align*}
\end{proof}

\subsection{Cauchy Interlacing Theorem}

The Cauchy Interlacing Theorem \ref{t:cauchinter} follows from the Min-Max Principle:
\begin{proof}[Proof of Theorem \ref{t:cauchinter}]
If $w_j$ is an eigenvector for $B$ with eigenvalue $\beta_j$ and $W_j^m := \Span\{w_j, \dots, w_m\}$, then one has that $P^*W_1^k$ is a $k$-dimensional subspace of $\bbR^n$ and
\begin{align*}
\beta_k &= \max_{y \in W_1^k}\frac{\langle By, y\rangle}{\|y\|^2} = \max_{y \in W_1^k}\frac{\langle PAP^*y, y\rangle}{\|y\|^2} = \max_{y \in W_1^k}\frac{\langle A (P^*y),P^*y\rangle}{\|P^*y\|^2} \\
&\geq \min_{U}\left\{ \max_{x}\left\{ \frac{\langle Ax, x\rangle}{\|x\|^2} \; :\; x \in U \setminus \{0\} \right\} : \dim(U) = k\right\} = \alpha_k.
\end{align*}
On the other hand, one has that $P^*W_k^m$ is an $m-k+1 = n - (n-m+k) + 1$-dimensional subspace, and one likewise has that
\begin{align*}
\beta_k &= \min_{y \in W_k^m}\frac{\langle A (P^*y),P^*y\rangle}{\|P^*y\|^2} \\
&\leq \max_{U}\left\{ \min_{x}\left\{ \frac{\langle Ax, x\rangle}{\|x\|^2} \; :\; x \in U \setminus \{0\} \right\} : \dim(U) = m-k+1 \right\} = \alpha_{n-m+k}. \qedhere
\end{align*}
\end{proof}

\subsection{Eigenvalue separation lemma}

This lemma is classical (cf. \cite{FroSpeWit90, KirSim85}), but we include a proof for the sake of completeness.  Consider the transfer matrix
\begin{align*}
A(V,E) &= \begin{bmatrix}
\frac{V_0 - E}{-\varepsilon} & -1\\
1 & 0
\end{bmatrix}
\end{align*}
with associated Schr\"odinger cocycle
\begin{align*}
M_{[a,b]}(V,E) := \prod_{k=b}^{a} A(S^k V,E)
\end{align*}
where $(S^k V)_m = V_{m+k}$, the product concatenates on the right, and $b \geq a$.

By definition, one has that $\psi = \psi(n)$ solves the formal Schr\"odinger difference equation if and only if
\begin{align*}
M_{[a,b]}(V,E) \begin{bmatrix}
\psi(a) \\
\psi(a-1)
\end{bmatrix} &= \begin{bmatrix}
\psi(b) \\
\psi(b-1)
\end{bmatrix}
\end{align*}
for all $n \in \bbZ$.  We also have the fundamental cocycle identity
\begin{align}\label{eq:cocycle}
M_{[a,b]}(V,E) &= M_{[c+1,b]}(V,E)M_{[a,c]}(V,E), \quad a \leq c < b.
\end{align}

The proof of Lemma \ref{l:evalsep} comes from the simplicity of the eigenvalues of $H^{\snext\Lambda}$ and the orthogonality of the corresponding eigenvectors; namely, forcing a sufficient proportion of the masses of two distinct eigenvectors into the same window pushes the corresponding eigenvalues apart.  We make this quantitative below:

\begin{proof}[Proof of Lemma \ref{l:evalsep}]
If $|\snext{E}_2 - \snext{E}_1| \geq \varepsilon\left(\left(\frac{m_{\Lambda}}{\varepsilon}\right)^2 + 2\right)/|\Lambda|^2$, the claim follows immediately; suppose for the remainder of the proof that
\begin{align*}
|\snext{E}_2 - \snext{E}_1| < \varepsilon\left(\left(\frac{m_{\Lambda}}{\varepsilon}\right)^2 + 2\right)/|\Lambda|^2.
\end{align*}

Denote $[a,b]:=\snext\Lambda$, $[c,d]:={\Lambda}$, and
$\vec{\psi}_k := [\snext{\psi}_k(c), \snext{\psi}_k(c-1)]^\top$.  Since the eigenvectors $\snext{\psi}_k$ are real, we may suppose they are normalized at the left endpoint $c$ of ${\Lambda}$ such that
\begin{align*}
\|\vec{\psi}_k\|^2 = 1, \quad k = 1, 2,
\end{align*}
and
\begin{align*}
\langle \vec{\psi}_1, \vec{\psi}_2 \rangle \geq 0.
\end{align*}
In particular, we note for later use that $\|\snext\psi_k\|^2 \geq \|\vec{\psi}_k\|^2 \geq 1$.  With this normalization, if $\theta \in (0,\pi/2]$ denotes the acute angle between $\vec{\psi}_1$ and $\vec{\psi}_2$, one has
\begin{align*}
1 - \langle \vec\psi_1, \vec\psi_2 \rangle &= 1 - \cos\theta \leq 1 - \cos^2\theta \\ 
&\leq \sin^2\theta \leq \sin\theta \\
&= |\snext\psi_1(c)\snext\psi_2(c-1) - \snext\psi_1(c-1)\snext\psi_2(c)| =: |W(\snext\psi_1,\snext\psi_2)(c)|
\end{align*}
and consequently
\begin{align*}
\frac{1}{2}\|\vec{\psi}_1 - \vec{\psi}_2\|^2 \leq |W(\snext\psi_1,\snext\psi_2)(c)|.
\end{align*}

By Green's identity, one has
\begin{align*}
\frac{\snext{E}_2 - \snext{E}_1}{\varepsilon}\sum_{j = c}^{b} \snext\psi_1(j)\snext\psi_2(j) &= {\frac1\varepsilon}\sum_{j=c}^{b} \snext\psi_1(j) (H^{\snext\Lambda}\snext\psi_2)(j) - (H^{\snext\Lambda}\snext\psi_1)(j)\snext\psi_2(j) \\
&=  W(\snext\psi_1,\snext\psi_2)(c) - W(\snext\psi_1,\snext\psi_2)({b}+1).
\end{align*}
Since $\snext\psi_k$ are eigenfunctions for $H^{[a,b]}$, they satisfy the Dirichlet boundary condition at $b$; in particular, $W(\snext\psi_1,\snext\psi_2)(b+1) = 0$, and we see
\begin{align*}
 |W(\snext\psi_1,\snext\psi_2)(c)| = \left|\frac{\snext{E}_2 - \snext{E}_1}{\varepsilon}\sum_{j=c}^b \snext\psi_1(j) \snext\psi_2(j)\right| \leq |\snext{E}_2 - \snext{E}_1| \|\snext\psi_1\|\|\snext\psi_2\|
\end{align*}
by the Cauchy-Schwarz inequality.  Combining these observations, we see that
\begin{align}\label{eq:evalsepobs1}
\frac{1}{\sqrt{2}}\|\vec{\psi}_1 - \vec{\psi}_2\| \leq \left(\frac{|\snext{E}_2 - \snext{E}_1|}{\varepsilon}\|\snext\psi_1\|\|\snext\psi_2\|\right)^{1/2}.
\end{align}

We will apply observation \eqref{eq:evalsepobs1} with a telescoping estimate on transfer matrices to get an upper estimate on $\|\snext\psi_1 - \snext\psi_2)\|_{{\Lambda}}^2$ in terms of $|\snext{E}_2 - \snext{E}_1|$.  Specifically, we have
\begin{align*}
\|M_{[c,d]}(V,\snext{E}_1) - M_{[c,d]}(V,\snext{E}_2)\| &\leq \sum_{j \in {\Lambda}} \frac{|\snext{E}_2 - \snext{E}_1|}{\varepsilon}\|M_{[c,j-1]}(V,\snext{E}_1)\|\|M_{[j+1,d]}(V,\snext{E}_2)\| \\
&\leq \frac{|\snext{E}_2 - \snext{E}_1|}{\varepsilon} |{\Lambda}| \left(\left(\frac{m_{{\Lambda}}}{\varepsilon}\right)^2 + 2\right)^{(|{\Lambda}|-1)/2 },
\end{align*}
and consequently
\begin{align*}
\|\snext\psi_1 &- \snext\psi_2\|_{{\Lambda}} \leq \sum_{j \in {\Lambda}} |\snext\psi_1(j) - \snext\psi_2(j)| \\
&\leq \sum_{j \in {\Lambda}} \|M_{[c,j]}(V,\snext{E}_1)\vec{\psi}_1 - M_{[c,j]}(V,\snext{E}_2)\vec{\psi}_2\| \\
&\leq \sum_{j \in {\Lambda}}\|M_{[c,j]}(V,\snext{E}_1) - M_{[c,j]}(V,\snext{E}_2)\| + \|M{[c,j]}(V,\snext{E}_2)(\vec{\psi}_1 - \vec{\psi}_2)\| \\
&\leq \frac{|\snext{E}_2-\snext{E}_1|}{\varepsilon}|{\Lambda}|^2\left(\left(\frac{m_{{\Lambda}}}{\varepsilon}\right)^2 + 2\right)^{(|{\Lambda}|-1)/2} + |{\Lambda}|\left(\left(\frac{m_{{\Lambda}}}{\varepsilon}\right)^2 + 2\right)^{|{\Lambda}|/2}\left(\frac{2|\snext{E}_2 - \snext{E}_1|}{\varepsilon}\|\snext\psi_1\|\|\snext\psi_2\|\right)^{1/2}.
\end{align*}

Since $|\snext{E}_2-\snext{E}_1| < \varepsilon\left(\left(\frac{m_{{\Lambda}}}{\varepsilon}\right)^2 + 2\right)/|{\Lambda}|^2$ by our very first assumption, it follows that
\begin{align*}
    \|\snext\psi_1-\snext\psi_2\|_{{\Lambda}} &\leq \left(\frac{|\snext{E}_2 - \snext{E}_1|}{\varepsilon}\right)^{1/2} |{\Lambda}|\left(\left(\frac{m_{{\Lambda}}}{\varepsilon}\right)^2 + 2\right)^{|{\Lambda}|/2}\left(1+\sqrt{2}(\|\snext\psi_1\|\|\snext\psi_2\|)^{1/2}\right)
\end{align*}
and thus 
\begin{align*}
\|\snext\psi_1 - \snext\psi_2\|_{{\Lambda}}^2 &\leq \frac{|\snext{E}_2 - \snext{E}_1|}{\varepsilon}|{\Lambda}|^2\left(\left(\frac{m_{{\Lambda}}}{\varepsilon}\right)^2+2\right)^{|{\Lambda}|}\left( 1 + \sqrt{2}(\|\snext\psi_1\|\|\snext\psi_2\|)^{1/2}\right)^2  \\
&< 6\frac{|\snext{E}_2 - \snext{E}_1|}{\varepsilon}|{\Lambda}|^2\left(\left(\frac{m_{{\Lambda}}}{\varepsilon}\right)^2+2\right)^{|{\Lambda}|} \|\snext\psi_1\|\|\snext\psi_2\|,
\end{align*}
where in the last line we used that $1 \leq (\|\snext\psi_1\|\|\snext\psi_2\|)^{1/2}$ and $(1 + \sqrt{2})^2 < 6$.

On the other hand, since $\snext\psi_1$ and $\snext\psi_2$ are eigenfunctions for different eigenvalues, they are orthogonal, and we have
\begin{align*}
\|\snext\psi_1 - \snext\psi_2\|_{{\Lambda}}^2 &= \|\snext\psi_1 - \snext\psi_2\|^2 - \|\snext\psi_1 - \snext\psi_2\|_{\snext{\Lambda}\setminus{\Lambda}}^2 \\
&\geq \left(\|\snext\psi_1\|^2 + \|\snext\psi_2\|^2\right) - \frac{1}{2}\left(\|\snext\psi_1\|^2 + \|\snext\psi_2\|^2\right) \\
&\geq \frac{1}{2}\left(\|\snext\psi_1\|^2 + \|\snext\psi_2\|^2\right) \\
&\geq \|\snext\psi_1\|\|\snext\psi_2\|
\end{align*}
where the first line above uses orthogonality of vectors with disjoint supports, the second uses orthogonality of distinct eigenvectors and the localization assumption $\|\snext\psi_k\|_{{\Lambda}}^2 \geq \|\snext\psi_k\|^2/2$, and the fourth line uses the inequality of arithmetic and geometric means.

Combining the upper and lower bounds on $\|\snext\psi_1 - \snext\psi_2\|_{{\Lambda}}^2$ above yields the claim.
\end{proof}


\newpage


\begin{thebibliography}{[aaa]}







\bibitem{BinKinVod16}
I. Binder, D. Kinzebulatov, and M. Voda.
\newblock Non-perturbative localization with quasiperiodic potential in continuous time.
\newblock {\em Commun. Math. Phys.} \textbf{351} (2017), 1149--1175.


\bibitem{Bou05}
J. Bourgain.
\newblock Positivity and continuity of the Lyapounov exponent for shifts on $\mathbb{T}^d$ with arbitrary frequency vector and real analytic potential with arbitrary frequency vector and real analytic potential. 
\newblock {\em J. Anal. Math.} \textbf{96} (2005), 313--355.


\bibitem{BouGol00}
J. Bourgain and M. Goldstein.
\newblock On nonperturbative localization with quasi-periodic potential.
\newblock {\em Ann. of Math.} \textbf{152} (2000), 835--879.


\bibitem{BouGolSch01}
J. Bourgain, M. Goldstein, and W. Schlag.
\newblock Anderson localization for Schr\"odinger operators on $\mathbb{Z}$ with potentials given by the skew-shift.
\newblock {\em Commun. Math. Phys.} \textbf{220} (2001), 583--621.

\bibitem{BouSch00} J.\ Bourgain and W.\ Schlag.  Anderson localization for Schr\"odinger operators on $\mathbb{Z}$ with strongly mixing potentials, \textit{Commun.\ Math.\ Phys.}\ \textbf{215} (2000), 143--175.


\bibitem{BDFGVWZ19-1}
V. Bucaj, D. Damanik, J. Fillman, V. Gerbuz, T. VandenBoom, F. Wang, and Z. Zhang.
\newblock {L}ocalization for the one-dimensional {A}nderson model via positivity and large deviations for the {L}yapunov exponent.
\newblock {\em Trans. Amer. Math. Soc.} \textbf{372} (2019), 3619--3667.

\bibitem{BDFGVWZ19-2}
V. Bucaj, D. Damanik, J. Fillman, V. Gerbuz, T. VandenBoom, F. Wang, and Z. Zhang.
\newblock Positive Lyapunov exponents and a large deviation theorem for continuum Anderson models, briefly.
\newblock {\em J. Func. Anal.} \textbf{277} (2019), 3179--3186.

\bibitem{CKM87} R.\ Carmona, A.\ Klein, and F.\ Martinelli. Anderson localization for Bernoulli and other singular potentials, \textit{Commun.\ Math.\ Phys.}\ \textbf{108} (1987), 41--66.




\bibitem{DFS20}
D. Damanik, J. Fillman, and S. Sukhtaiev.
\newblock Localization for Anderson models on metric and discrete tree graphs.
\newblock {\em Math. Ann.}, \textbf{376} (2020), 1337--1393.




\bibitem{DK89} H.\ von Dreifus and A.\ Klein. A new proof of localization in the Anderson tight binding model, \textit{Commun.\ Math.\ Phys.}\ \textbf{124} (1989), 285--299.

\bibitem{EK16} A.\ Elgart and A.\ Klein. An eigensystem approach to Anderson localization, \textit{J.\ Funct.\ Anal.}\ \textbf{271} (2016), 3465--3512.

\bibitem{FMSS85} J.\ Fr\"ohlich, F.\ Martinelli, E.\ Scoppola, and T.\ Spencer.  Constructive proof of localization in the Anderson tight binding model, \textit{Commun.\ Math.\ Phys.}\ \textbf{101} (1985), 21--46.

\bibitem{FroSpe83}
J. Fr\"ohlich and T. Spencer.
\newblock Absence of diffusion in the Anderson tight binding model for large disorder or low energy.
\newblock {\em Commun. Math. Phys.} \textbf{88} (1983), 151--184.

\bibitem{FroSpeWit90}
J. Fr\"ohlich, T. Spencer, and P. Wittwer.
\newblock 	Localization for a class of one-dimensional quasi-periodic Schr\"odinger operators.
\newblock {\em Comm. Math. Phys.} \textbf{132} (1990), 5--25.




\bibitem{GK01} F.\ Germinet and A.\ Klein. Bootstrap multiscale analysis and localization in random media, \textit{Commun.\ Math.\ Phys.}\ \textbf{222} (2001), 415--448.



\bibitem{goldstein} M.\ Goldstein and W.\ Schlag. H\"older continuity of the integrated density of states for quasi-periodic Schr\"odinger equations and averages of shifts of subharmonic functions, \textit{Ann.\ of Math.}\ \textbf{154} (2001), 155--203.

\bibitem{GK17} A. Gorodetski, V.\ Kleptsyn, Parametric F\"urstenberg theorem on random products of $\mathrm{SL}(2,\mathbb{R})$ matrices, \textit{Adv. Math.} \textbf{378} (2021), 107522.




\bibitem{Jit99}
S. Jitomirskaya.
\newblock Metal-insulator transition for the almost Mathieu operator.
\newblock {\em Ann. of Math.} \textbf{150} (1999), 1159--1175.

\bibitem{JitZhu19}
S. Jitomirskaya and X. Zhu.
\newblock Large deviations of the Lyapunov exponent and localization for the 1D Anderson model.
\newblock {\em Commun. Math. Phys.} \textbf{370} (2019), 311--324.

\bibitem{KirSim85}
W. Kirsch and B. Simon. 
\newblock Universal lower bounds on eigenvalue splittings for one dimensional Schrödinger operators. 
\newblock {\em Commun. Math. Phys.} \textbf{97} (1985), 453–460.



\bibitem{Schlag21} W.\ Schlag.  An introduction to multiscale techniques in the theory of Anderson localization, Part I, preprint (arXiv:2104.14248).

\bibitem{Schnol1981} I.\ Schnol. On the behavior of eigenfunctions, \textit{Doklady Akad.\ Nauk SSSR} (N.S.) \textbf{94} (1981), 389--392.

\bibitem{Simon1981JFA} B.\ Simon.  Spectrum and continuum eigenfunctions of Schr\"odinger operators, \textit{J.\ Funct.\ Anal.} \textbf{42} (1981), 66--83.



\bibitem{Sin87}
Y. Sinai.
\newblock Anderson localization for one-dimensional difference Schr\"odinger operator with quasiperiodic potential.
\newblock {\em J. Stat. Phys.} \textbf{46} (1987), 861--909.

\bibitem{SorSpe91}
E. Sorets and T. Spencer.
\newblock Positive Lyapunov exponenets for Schr\"odinger operators with quasi-periodic potentials.
\newblock {\em Commun. Math. Phys.} \textbf{142} (1991), 543--566.



\bibitem{WanZha15}
Y. Wang and Z. Zhang.
\newblock Uniform positivity and continuity of Lyapunov exponents for a class of $C^2$ quasiperiodic Schr\"odinger cocycles.
\newblock {\em J. Func. Anal.} \textbf{268} (2015), 2525--2585.

\bibitem{WanZha16}
Y. Wang and Z. Zhang.
\newblock Cantor spectrum for a class of $C^2$ quasiperiodic Schr\"odinger cocycles.
\newblock {\em Int. Math. Res. Not.} \textbf{2017} (2017), 2300--2336.




\end{thebibliography}
\end{document}